\documentclass[11pt,oneside,article]{memoir}
\input{preamble.tex}

\settrims{0pt}{0pt} 
\settypeblocksize{*}{35pc}{*} 
\setlrmargins{*}{*}{1} 
\setulmarginsandblock{1in}{1in}{*} 
\setheadfoot{\onelineskip}{2\onelineskip} 
\setheaderspaces{*}{1.5\onelineskip}{*} 
\checkandfixthelayout

\makeatletter
\def\blfootnote{\gdef\@thefnmark{}\@footnotetext}
\makeatother

\setsecheadstyle{\bfseries\large\raggedright}
\setsubsecheadstyle{\bfseries\raggedright}

\title{String diagrams for traced and compact categories are oriented 1-cobordisms}
\author{David I. Spivak\thanks{Supported by AFOSR grant FA9550--14--1--0031, ONR grant N000141310260, and NASA grant NNL14AA05C.} \and Patrick Schultz${}^*$\\ \and \small{\textit{Massachusetts Institute of Technology, Cambridge, MA 02139}} \and Dylan Rupel\thanks{Corresponding author}\thanks{Present address: University of Notre Dame, Notre Dame, IN 46556}\\ \small{\textit{Northeastern University, Boston, MA 02115}}}

\date{\vspace{-3ex}}

\begin{document}
\firmlists*

\maketitle
\blfootnote{\textit{Email addresses:} \href{mailto:dspivak@math.mit.edu}{dspivak@math.mit.edu}, %
  \href{schultzp@mit.edu}{schultzp@mit.edu}, %
  \href{drupel@nd.edu}{drupel@nd.edu}
}

\begin{abstract}
  We give an alternate conception of string diagrams as labeled 1-dimensional oriented cobordisms,
  the operad of which we denote by $\LCob{\LabSet}$, where $\LabSet$ is the set of string labels.
  The axioms of traced (symmetric monoidal) categories are fully encoded by $\LCob{\LabSet}$ in the sense that there is
  an equivalence between $\LCob{\LabSet}$-algebras, for varying $\LabSet$, and traced categories
  with varying object set. The same holds for compact (closed) categories, the difference being in terms of
  variance in $\LabSet$. As a consequence of our main theorem, we give a characterization of the
  2-category of traced categories solely in terms of those of monoidal and compact categories,
  without any reference to the usual structures or axioms of traced categories. In an appendix we
  offer a complete proof of the well-known relationship between the 2-category of monoidal
  categories with strong monoidal functors and the 2-category of monoidal categories whose object set is free
  with strict functors; similarly for traced and compact categories. \\

\end{abstract}

\tableofcontents*

\chapter{Introduction}
  \label{chap:intro}

Traced (symmetric monoidal) categories have been used to model processes with
feedback~\cite{Abramsky1} or operators with fixed points~\cite{PontoShulman}. A graphical calculus
for traced categories was developed by Joyal, Street, and Verity~\cite{JoyalStreetVerity} in which
string diagrams of the form
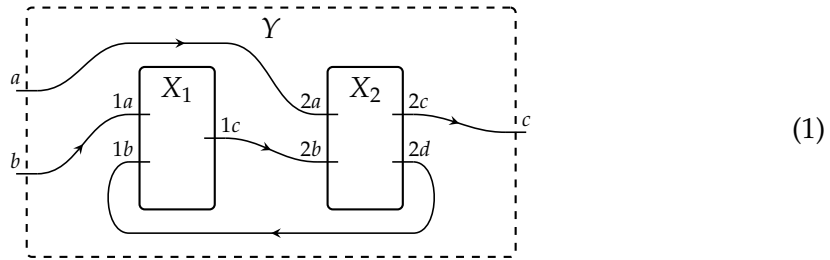
\begin{equation}\begin{tikzpicture}[wiring diagram,baseline]
    \label{dia:string_diagram}
  \node[bb={2}{1}, bb name=$X_1$] (X1) {};
  \node[bb={2}{2}, right=of X1, bb name=$X_2$] (X2) {};
  \node[bb={2}{1}, dashed, fit={(X1) (X2) ($(X1.north)+(0,1.5)$) ($(X1.south)-(0,1)$)},
        bb name=$Y$] (Y) {};
  \draw[label]
    node[above left=of Y_in1]     {$a$}
    node[above left=of Y_in2]     {$b$}
    node[above right=of Y_out1]   {$c$}
    node[above left=of X1_in1]    {$1a$}
    node[above left=of X1_in2]    {$1b$}
    node[above right=of X1_out1]  {$1c$}
    node[above left=of X2_in1]    {$2a$}
    node[above left=of X2_in2]    {$2b$}
    node[above right=of X2_out1]  {$2c$}
    node[above right=of X2_out2]  {$2d$};
  \draw[ar] (Y_in2') to (X1_in1);
  \draw[ar] (X1_out1) to (X2_in2);
  \draw[ar] (X2_out1) to (Y_out1');
  \draw[ar] let \p1=(X1.north west), \p2=(X1.north east), \n1={\y1+\bby}, \n2=\bbportlen in
    (Y_in1') to (\x1-\n2,\n1) -- (\x2+\n2,\n1) to (X2_in1);
  \draw[ar] let \p1=(X2.south east), \p2=(X1.south west), \n1={\y1-\bby}, \n2=\bbportlen in
    (X2_out2) to[in=0] (\x1+\n2,\n1) -- (\x2-\n2,\n1) to[out=180] (X1_in2);
\end{tikzpicture}\end{equation}
represent compositions in a traced category $\cat{T}$. That is, new morphisms are constructed from
old by specifying which outputs will be fed back into which inputs. These are related to Penrose
diagrams in $\ncat{Vect}$ and the word \emph{traced} originates in this vector space
terminology.

The string diagrams of \cite{JoyalStreetVerity} typically do not explicitly include the outer box
$Y$. If we include it, as in (\ref{dia:string_diagram}), the resulting \emph{wiring diagram} can be
given a seemingly new interpretation: it represents a 1-dimensional cobordism between oriented
0-manifolds. Indeed, the objects in $\Cob$ are signed sets $X=(\inp{X},\outp{X})$, each of which
can be drawn as a box with input wires $\inp{X}$ entering on the left and output wires $\outp{X}$
exiting on the right.
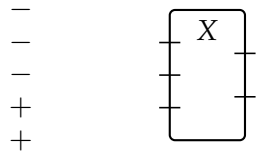
\begin{equation*}
  \begin{tikzpicture}[node distance=0 and 0, baseline=(current bounding box.center)]
    \node (A1) {$-$};
    \node[below=-.1 of A1] (A2) {$-$};
    \node[below=-.1 of A2] (A3) {$-$};
    \node[below=-.1 of A3] (B1) {$+$};
    \node[below=-.1 of B1] (B2) {$+$};
  \end{tikzpicture}
  \hspace{4em}
  \begin{tikzpicture}[wiring diagram, bby=1.2ex, baseline=(current bounding box.center)]
    \node[bb={3}{2},bb name=$X$] {};
  \end{tikzpicture}
\end{equation*}
Moreover, the wiring diagram itself in which boxes $X_1,\ldots,X_n$ are wired together inside a
larger box $Y$ can be interpreted as an oriented cobordism from $X_1\sqcup\cdots\sqcup X_n$ to $Y$.
In fact, this is more appropriately interpreted as a morphism in the (colored) operad $\Cob$
underlying the symmetric monoidal category of oriented 1-cobordisms. The
following shows the two approaches to drawing a 2-ary morphism $X_1,X_2\to Y$ in $\Cob$:
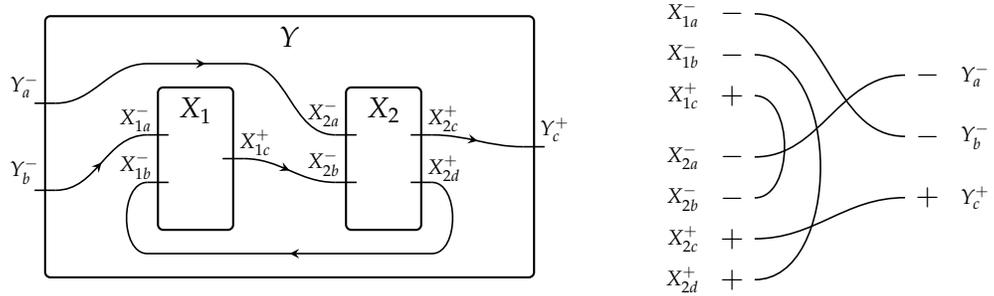
\begin{equation*}
  \begin{tikzpicture}[wiring diagram, baseline=(current bounding box.center)]
    \node[bb={2}{1}, bb name=$X_1$] (X1) {};
    \node[bb={2}{2}, right=of X1, bb name=$X_2$] (X2) {};
    \node[bb={2}{1}, fit={(X1) (X2) ($(X1.north)+(0,2)$) ($(X1.south)-(0,1)$)},bb name =$Y$] (Y) {};
    \draw[label]
      node[above left=of Y_in1]     {$\inp{Y}_a$}
      node[above left=of Y_in2]     {$\inp{Y}_b$}
      node[above right=of Y_out1]   {$\outp{Y}_c$}
      node[above left=1pt and -2pt of X1_in1]    {$\inp{X}_{1a}$}
      node[above left=1pt and -2pt of X1_in2]    {$\inp{X}_{1b}$}
      node[above right=1pt and -2pt of X1_out1]  {$\outp{X}_{1c}$}
      node[above left=3pt and -2pt of X2_in1]    {$\inp{X}_{2a}$}
      node[above left=2pt and -2pt of X2_in2]    {$\inp{X}_{2b}$}
      node[above right=1pt and -2pt of X2_out1]  {$\outp{X}_{2c}$}
      node[above right=0pt and -2pt of X2_out2]  {$\outp{X}_{2d}$};
    \draw[ar] (Y_in2') to (X1_in1);
    \draw[ar] (X1_out1) to (X2_in2);
    \draw[ar] (X2_out1) to (Y_out1');
    \draw[ar] let \p1=(X1.north west), \p2=(X1.north east), \n1={\y1+\bby}, \n2=\bbportlen in
      (Y_in1') to (\x1-\n2,\n1) -- (\x2+\n2,\n1) to (X2_in1);
    \draw[ar] let \p1=(X2.south east), \p2=(X1.south west), \n1={\y1-\bby}, \n2=\bbportlen in
      (X2_out2) to[in=0] (\x1+\n2,\n1) -- (\x2-\n2,\n1) to[out=180] (X1_in2);
  \end{tikzpicture}
  \qquad\quad
  \begin{tikzpicture}[x=1cm,y=1ex,node distance=1 and 1,semithick,every label quotes/.style={font=\everymath\expandafter{\the\everymath\scriptstyle}},every to/.style={out=0,in=180},baseline=(current bounding box.center)]
    \node ["$\inp{X}_{1a}$" left] (X1a) {$-$};
    \node [below=0 of X1a, "$\inp{X}_{1b}$" left] (X1b) {$-$};
    \node [below=0 of X1b, "$\outp{X}_{1c}$" left] (X1c) {$+$};
    \node [below=1.5 of X1c, "$\inp{X}_{2a}$" left] (X2a) {$-$};
    \node [below=0 of X2a, "$\inp{X}_{2b}$" left] (X2b) {$-$};
    \node [below=0 of X2b, "$\outp{X}_{2c}$" left] (X2c) {$+$};
    \node [below=0 of X2c, "$\outp{X}_{2d}$" left] (X2d) {$+$};
    \node [below right=1.5 and 2 of X1a, "$\inp{Y}_a$" right] (Ya) {$-$};
    \node [below=1.5 of Ya, "$\inp{Y}_b$" right] (Yb) {$-$};
    \node [below=1.5 of Yb, "$\outp{Y}_c$" right] (Yc) {$+$};
    \draw (X1a) to (Yb);
    \draw (X1b) to[in=0] (X2d);
    \draw (X1c) to[in=0] (X2b);
    \draw (X2a) to (Ya);
    \draw (X2c) to (Yc);
  \end{tikzpicture}
\end{equation*}

There is actually a bit more data in a string (or wiring) diagram for a traced category $\cat{T}$
than in a cobordism. Namely, each input and output of a box must be labeled by an object of
$\cat{T}$ and the wires connecting boxes must respect the labels (e.g.\ in
(\ref{dia:string_diagram}) objects $1c$ and $2b$ must be equal). We will thus consider the operad
$\LCob{\LabSet}$ of oriented 1-dimensional cobordisms over a fixed set of labels $\LabSet$. We also
write $\LCob{\LabSet}$ to denote the corresponding symmetric monoidal category.

In the table below, we record these two interpretations of a string diagram. Note the ``degree
shift'' between the second and third columns.
\begin{center}
  \setlength{\tabcolsep}{10pt}
  \begin{tabular}{lll}
    \toprule
    \multicolumn{3}{c}{Interpretations of string diagrams} \\
    \midrule
    String diagram & Traced category $\cat{T}$ & $\LCob{\LabSet}$ \\
    \midrule
    Wire label set, $\LabSet$ & Objects, $\LabSet\coloneqq\Ob(\cat{T})$ & Label set, $\LabSet$ \\
    Boxes,
    e.g.~\tikz[wiring diagram,bb port sep=1,bby=2.4pt,bb min width=5.5pt,
          bb port length=2pt,bb rounded corners=1pt,baseline=(B.south)]
      {\node[bb={1}{2}] (B) {};}
    & Morphisms in $\cat{T}$& Objects (oriented 0-mfds over $\LabSet$) \\
    String diagrams & Compositions in $\cat{T}$& Morphisms (cobordisms over $\LabSet$) \\
    Nesting & Axioms of traced cats & Composition (of cobordisms) \\
    \bottomrule
  \end{tabular}
\end{center}

In the last row above, each of the seven axioms of traced categories is vacuous from the cobordism
perspective in the sense that both sides of the equation correspond to the same cobordism (up to
diffeomorphism). For example, the axiom of \emph{superposition} reads:
\[
  \Tr^U_{X,Y}\big[f\big]\otimes g=\Tr^U_{X\otimes W,Y\otimes Z}\big[f\otimes g\big]
\]
for every $f\colon U\otimes X\to U\otimes Y$ and $g\colon W\to Z$, or diagramatically:
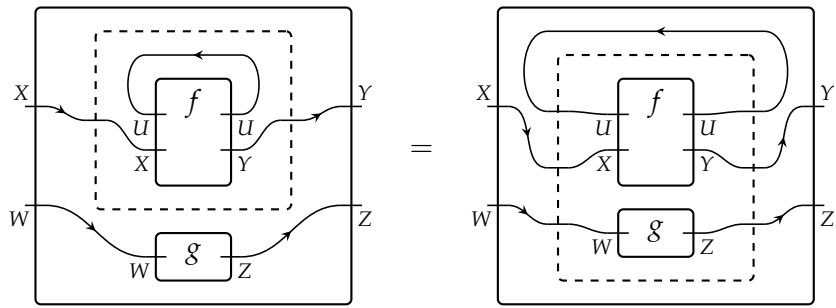
\[\tikzset{bbx=.8cm,bb port sep=1.5}
  \begin{tikzpicture}[wiring diagram,baseline=(current bounding box.center)]
    \node[bb={2}{2}, bb name=$f$] (X1) {};
    \node[bb port sep=1,bb={1}{1}, below=2 of X1, bb name=$g$] (X2) {};
    \node[bb={1}{1}, fit={(X1) ($(X1.north)+(0,1)$)}, dashed] (Z) {};
    \node[bb={2}{2}, fit={(Z) (X2)}] (Y) {};
    \draw[ar] (Y_in1') to (Z_in1);
    \draw (Z_in1') to (X1_in2);
    \draw[ar] (Y_in2') to (X2_in1);
    \draw (X1_out2) to (Z_out1');
    \draw[ar] (Z_out1) to (Y_out1');
    \draw[ar] (X2_out1) to (Y_out2');
    \draw[ar] let \p1=(X1.north east), \p2=(X2.north west), \n1={\y1+\bby}, \n2=\bbportlen in
      (X1_out1) to[in=0] (\x1+\n2,\n1) -- (\x2-\n2,\n1) to[out=180] (X1_in1);
    \draw[label]
      node[above left=of Y_in1] {$X$}
      node[below left=of Y_in2] {$W$}
      node[above right=of Y_out1] {$Y$}
      node[below right=of Y_out2] {$Z$}
      node[below left=of X1_in1] {$U$}
      node[below left=of X1_in2] {$X$}
      node[below right=of X1_out2] {$Y$}
      node[below right=of X1_out1] {$U$}
      node[below left=of X2_in1] {$W$}
      node[below right=of X2_out1] {$Z$};
  \end{tikzpicture}
  \quad=\quad
  \begin{tikzpicture}[wiring diagram,baseline=(current bounding box.center)]
    \node[bb={2}{2}, bb name=$f$] (X1) {};
    \node[bb port sep=1,bb={1}{1}, below=of X1, bb name=$g$] (X2) {};
    \node[bb={3}{3}, fit={(X1) (X2)}, dashed] (Z) {};
    \node[bb={2}{2}, fit={(Z) ($(Z.north)+(0,1)$)}] (Y) {};
    \draw[ar] (Y_in1') to (Z_in2);
    \draw (Z_in2') to (X1_in2);
    \draw[ar] (Y_in2') to (Z_in3);
    \draw (Z_in3') to (X2_in1);
    \draw (X1_out2) to (Z_out2');
    \draw[ar] (Z_out2) to (Y_out1');
    \draw (X2_out1) to (Z_out3');
    \draw[ar] (Z_out3) to (Y_out2');
    \draw (Z_in1') to (X1_in1);
    \draw (X1_out1) to (Z_out1');
    \draw[ar] let \p1=(Z.north east), \p2=(Z.north west), \n1={\y1+\bby}, \n2=\bbportlen in
      (Z_out1) to[in=0] (\x1+\n2,\n1) -- (\x2-\n2,\n1) to[out=180] (Z_in1);
    \draw[label]
      node[above left=of Y_in1] {$X$}
      node[below left=of Y_in2] {$W$}
      node[above right=of Y_out1] {$Y$}
      node[below right=of Y_out2] {$Z$}
      node[below left=of X1_in1] {$U$}
      node[below left=of X1_in2] {$X$}
      node[below right=of X1_out2] {$Y$}
      node[below right=of X1_out1] {$U$}
      node[below left=of X2_in1] {$W$}
      node[below right=of X2_out1] {$Z$};
  \end{tikzpicture}
\]

To make precise the relationship between these interpretations of string diagrams, we fix the set
$\LabSet$ of labels. Let $\TrCat$ denote the 1-category of traced categories and traced strict
monoidal functors. Write $\TrCat_{\LabSet}$ for the subcategory consisting of those traced
categories $\cat{T}$ for which the monoid of objects is free on the set $\LabSet$, with
identity-on-objects functors $\cat{T}\to\cat{T}'$ between them.

\begin{named}{Theorem 0}
    \label{thm:traced_is_cob_alg}
  There is an equivalence of 1-categories
  \begin{equation}
      \label{eq:single_fiber_tr}
    \LCob{\LabSet}\alg\equiv\TrCat_{\LabSet},
  \end{equation}
  where, given any monoidal category $\cat{M}$, we denote by
  $\cat{M}\alg\coloneqq\ncat{Lax}(\cat{M},\Set)$ the category of lax functors $\cat{M}\to\Set$ and
  monoidal natural transformations.
\end{named}

To build intuition for this statement note that the same data are required, and the same conditions
are satisfied, whether one is specifying a lax functor $P\in\LCob{\LabSet}\alg$ or a traced category
$\cat{T}\in\TrCat_{\LabSet}$ with objects freely generated by the set $\LabSet$. First, for each box
$X=(\inp{X},\outp{X})$ that might appear in a string diagram, both $P\colon\LCob{\LabSet}\to\Set$
and $\cat{T}$ require a set, $P(X)$ and $\Hom_{\cat{T}}(\inp{X},\outp{X})$, respectively. Second,
for each string diagram, both $P$ and $\cat{T}$ require a function: an action on morphisms in the
case of $P$ and a formula for performing the required compositions, tensors, and traces in the case
of $\cat{T}$. The condition that $P$ is functorial corresponds to the fact that $\cat{T}$ satisfies
the axioms of traced categories.

We will briefly specify how to construct a lax functor $P$ from a traced category
$(\cat{T},\otimes,I,\Tr)$ whose objects are freely generated by $\LabSet$. In what follows, we abuse
notation slightly: given a relative set $\iota\colon Z\to\LabSet$ we will use the same symbol $Z$ to
denote the tensor $\bigotimes_{z\in Z}\iota(z)$ in $\cat{T}$. For an oriented 0-manifold
$X=\inp{X}\sqcup \outp{X}$ over $\LabSet$, put $P(X)\coloneqq\Hom_{\cat{T}}(\inp{X},\outp{X})$.
Given a cobordism $\Phi\colon X\to Y$, we need a function $P(\Phi)\colon P(X)\to P(Y)$. To specify
it, note that for any cobordism $\Phi$ there exist $A,B,C,D,E\in\Ob(\cat{T})$ such that
$\inp{X}\cong C\otimes A$, $\outp{X}\cong C\otimes B$, $\inp{Y}\cong A\otimes D$, $\outp{Y}\cong
B\otimes D$, and $E$ is the set of floating loops in $\Phi$; thus $\Phi$ is essentially equivalent
to the cobordism shown on the left side of (\ref{eq:cob_and_trace_pic}).
\begin{equation}
    \label{eq:cob_and_trace_pic}
  \begin{tikzpicture}[
      x=1cm, y=1ex, node distance=1 and 1, semithick,
      every label quotes/.style={font=\everymath\expandafter{\the\everymath\scriptstyle}},
      every to/.style={out=0,in=180},
      baseline=(current bounding box.center)
    ]
    \node ["$A$" left] (X1a) {$-$};
    \node [above=0.25 of X1a] {$X$};
    \node [below=0 of X1a, "$C$" left] (X1b) {$-$};
    \node [below=0 of X1b, "$C$" left] (X2a) {$+$};
    \node [below=0 of X2a, "$B$" left] (X2b) {$+$};
    \node [right=2 of X1a, "$A$" right] (Y1a) {$-$};
    \node [above=0.25 of Y1a] {$Y$};
    \node [below=0 of Y1a, "$D$" right] (Y1b) {$-$};
    \node [below=0 of Y1b, "$D$" right] (Y2a) {$+$};
    \node [below=0 of Y2a, "$B$" right] (Y2b) {$+$};
    \node [right=1.45 of X2a, "$E$" left] {};
    \draw (X1a) to (Y1a);
    \draw (X1b) to[in=0] (X2a);
    \draw (X2b) to (Y2b);
    \draw (Y1b) to[in=180,out=180] (Y2a);
    \draw ($(X1b)+(1.25,-2.75)$) to[in=0] ($(X1b)+(1.25,-0.25)$);
    \draw ($(X1b)+(1.25,-0.25)$) to[in=180,out=180] ($(X1b)+(1.25,-2.75)$);
  \end{tikzpicture}
  \qquad\qquad
  \begin{tikzpicture}[wiring diagram, bby=1.4ex, baseline=(current bounding box.center)]
    \node[bb port sep=1.5, bb={2}{2}] (domain) {$f$};
    \node[bb={2}{2}, fit={(domain) ($(domain.north)+(0,3)$) ($(domain.south)-(0,1)$)}, bb name=$P(\Phi)(f)$] (codomain) {};
    \draw[ar] (codomain_in1') to (domain_in2);
    \draw[ar] (domain_out2) to (codomain_out1');
    \draw[ar] let \p1=(domain.north east), \p2=(domain.north west), \n1={\y1+\bby}, \n2=\bbportlen in
      (domain_out1) to[in=0] (\x1+\n2,\n1) -- (\x2-\n2,\n1) to[out=180] (domain_in1);  
    \draw[ar] let \p1=(domain.south west), \p2=(domain.south east), \n1={\y1-\bby}, \n2=\bbportlen in
      (codomain_in2') to[in=180] (\x1+\n2,\n1) -- (\x2-\n2,\n1) to[out=0] (codomain_out2'); 
    \draw[ar] let \p1=(domain.north east) in
      (\x1+.7*\bbx,\y1+\bby) to[in=0] (\x1+.7*\bbx,\y1+2*\bby) -- (\x1+.6*\bbx,\y1+2*\bby) to[out=180] (\x1+.6*\bbx,\y1+\bby) -- (\x1+.7*\bbx,\y1+\bby);
    \draw[label] let \p1=(domain.north east) in
      node[below left=of codomain_in1]     {$A$}
      node[below left=of codomain_in2]     {$D$}
      node[below right=of codomain_out1]    {$B$}
      node[below right=of codomain_out2]    {$D$}
      node[above left=.6 and 0 of codomain_out1']  {$E$}
      node[below left=of domain_in1]     {$C$}
      node[below left=of domain_in2]     {$A$}
      node[below right=of domain_out2]    {$B$}
      node[below right=of domain_out1]   {$C$};
  \end{tikzpicture}
\end{equation}
With the above notation, for $f\in P(X)$ we can follow the string diagram (right side of
(\ref{eq:cob_and_trace_pic})) and define
\begin{equation}
    \label{eq:cob algebra formula}
  P(\Phi)(f)\coloneqq\Tr_{A,B}^C[f]\otimes D\otimes\Tr^E_{I,I}[E],
\end{equation}
where we abuse notation and write $D$ and $E$ for the identity maps on these objects. One may easily
check, using each axiom of the trace \cite{JoyalStreetVerity} in an essential way, that
\eqref{eq:cob algebra formula} defines an algebra over $\LCob{\LabSet}$. We will not prove
\ref{thm:traced_is_cob_alg} directly as indicated here; to specify our proof strategy we must
introduce more notation.

\section{The main results}
  \label{sec:main_results}

The equivalence \eqref{eq:single_fiber_tr} has two significant conceptual drawbacks. First, the object set
of the traced category $\cat{T}$ is fixed; second, $\cat{T}$ is assumed to be freely generated by some set under tensor
products and functors are assumed to be strict. We refer to this latter condition using the term
\emph{objectwise-free}. Much of the work in this paper goes towards relaxing these two conditions; for now we continue to assume objectwise-freeness. 

To overcome the use of a fixed object set, we first explain what kind of object variance is
appropriate. There is an adjunction
\begin{equation}\label{eqn:FT_UT}
  \Adjoint{\Set}{\TrCat}{\FT}{\UT}
\end{equation}
inducing a monad $\TT$ on $\Set$, which is in fact isomorphic to the free monoid monad. This induces
a canonical functor from $\TrCat$ to the Eilenberg-Moore category $\Set^{\TT}$ of this monad,
sending each traced category to its underlying monoid of objects. We define
$\TrFrObCat\subset\TrCat$ to be the full subcategory spanned by the objectwise-free traced
categories. In other words, $\TrFrObCat$ is the pullback of $\TrCat\to\Set^{\TT}$ along the
fully-faithful functor $\Set_{\TT}\to\Set^{\TT}$ from the Kleisli category to the Eilenberg-Moore category,
or equivalently along the inclusion of free monoids into all monoids.

The compact category $\LCob{\LabSet}$ clearly varies functorially in $\LabSet\in\Set$, but it is not
much harder to see that it is also functorial in $\LabSet\in\Set_{\TT}$.  This gives rise to a functor
\[
  (\LCob{\bullet})\colon\Set_{\TT}\to\CompCat
\]
to the category $\CompCat$ of compact categories and strict functors, sending $\LabSet$ to
$\LCob{\LabSet}$, the free compact category on $\LabSet$ (e.g.\ see \cite{KellyLaplaza,Abramsky2}).
We can compose this with $\Lax(-,\Set)$ to obtain a functor which we denote
\begin{equation}
    \label{eqn:cob/bullet}
  (\LCob{\bullet})\alg\colon\Set_{\TT}^{\mathrm{op}}\too\Cat.
\end{equation}
By applying the Grothendieck construction (denoted by $\int$ here) to (\ref{eqn:cob/bullet}), we
obtain a fibration for which the fiber over a set $\LabSet$ is equivalent (by
\ref{thm:traced_is_cob_alg}) to $\TrCat_{\LabSet}$.
\begin{named}{Theorem A}
    \label{thm:TheoremA_statement}
  There is an equivalence of 1-categories
  \begin{align*}
    &\int^{\LabSet\in\Set_{\TT}}(\LCob{\LabSet})\alg\To{\equiv} \TrFrObCat.
  \end{align*}
\end{named}

This result, together with an analogous statement for compact categories, is proven in Section~\ref{sec:monoids_on_free}.

The fact that the traced categories appearing in \ref{thm:TheoremA_statement} are assumed
objectwise-free and the functors between them are strict is the second of two drawbacks mentioned
above. To address it, we prove that the 2-category $\TTrFrObCat$, of objectwise-free traced categories and strict functors, is
biequivalent to that of arbitrary traced categories and strong functors; see
Corollary~\ref{cor:object_frees}. This result seems to be well-known to experts but is difficult to
find in the literature.

In the course of proving \ref{thm:TheoremA_statement}, we will also establish generalizations
characterizing lax functors out of arbitrary compact categories. In order to state this
characterization, we prove (Theorem~\ref{thm:orthogonal} and Proposition~\ref{prop:CompProf_exact})
that the well-known $(\bo,\ff)$ factorization system of $\Cat$ restricts to a factorization system
on $\TrCat$; more precisely the left class consists of bijective-on-objects functors and the right
class consists of fully faithful functors.

Write $\TrCat^{\bo}$ for the full subcategory of the arrow category $\TrCat^{\to}$ spanned by the
bijective-on-objects functors. The existence of the factorization system implies that the domain
functor
\[
  \dom\colon\TrCat^{\bo}\twoheadrightarrow\TrCat
\]
is a fibration. For a fixed traced category $\cat{T}$, the fiber
$\TrCat^{\bo}_{\cat{T}/}\coloneqq\dom^{-1}(\cat{T})$ is the category of strict monoidal,
bijective-on-objects functors from $\cat{T}$ to another traced category, with the evident
commutative triangles as morphisms. Note that (with $\FT$ as in (\ref{eqn:FT_UT})) we have an isomorphism
$\TrCat_{\LabSet}\iso\TrCat^{\bo}_{(\FT\LabSet)/}$.

Recall from~\cite{JoyalStreetVerity} that traced categories can be thought of as full subcategories
of compact categories: the Int construction applied to a traced category $\cat{T}$ builds the
smallest compact category $\Int(\cat{T})$ of which $\cat{T}$ is a monoidal subcategory, we refer the reader to~\cite{JoyalStreetVerity} for more details. Generalizing
\eqref{eq:single_fiber_tr}, we can give a complete characterization of lax functors out of such
compact categories: for a fixed traced category $\cat{T}$ there is an equivalence of categories
\[
  \Lax(\Int(\cat{T}),\Set)\equiv\TrCat^{\bo}_{\cat{T}/}.
\]
In Section~\ref{sec:exactness_proofs} we show that these equivalences glue together to form an
equivalence of fibrations:

\begin{named}{Theorem B}
    \label{thm:TheoremB_statement}
  There is an equivalence of fibrations
  \begin{equation*}
    \begin{tikzcd}[column sep=-1em]
      \int\limits^{\mathclap{\cat{T}\in\TrCat}} \Lax(\Int(\cat{T}),\Set)
        \ar[rr,"\equiv"] \ar[dr,two heads]
      && \TrCat^{\bo} \ar[dl,two heads,"\dom" pos=.4] \\
      & \TrCat. &
    \end{tikzcd}
  \end{equation*}
\end{named}

\ref{thm:TheoremA_statement} will follow from this by restricting to just those $\cat{T}\in\TrCat$
which are free on a set $\LabSet$.

Our main tool in proving this result will be the 2-categorical notion of \emph{(proarrow)
equipments}, which we recall in Section~\ref{chap:background_equipments}. We will introduce what
appears to be a new definition of \emph{monoidal profunctors}, and the equipment thereof, in
Section~\ref{chap:equipments_monoidal_profunctors}.

\section{Plan of the paper}
Section~\ref{sec:equipments} reviews the definition of an equipment (or framed bicategory
\cite{Shulman}), while Section~\ref{sec:monoids_bimods} recalls monoids and bimodules in an
equipment. We define \emph{exact equipments} in Section~\ref{sec:exactness_boff}, which are central to
our proof strategy, and which we believe to be of independent interest. The material of this section
is original, though some of it appeared in the earlier unpublished \cite{Schultz2015}.

Section~\ref{sec:monoidal,compact,traced} briefly reviews monoidal, traced, and compact categories.
In Section~\ref{sec:monoidal_profunctors} we introduce our main objects of study, namely various equipments of monoidal profunctors
($\dMonProf, \dTrProf$, and $\dCompProf$), which we prove are exact
in Section~\ref{sec:exactness_proofs}. Section~\ref{sec:special_CompProf} is devoted to the special
properties of $\dCompProf$ which are at the core of our results---indeed one might view the rest of
the paper as a formal wrapper around the results in that section---and concludes with the proof of
\ref{thm:TheoremB}. In Section~\ref{sec:monoids_on_free} we deal with the issue of
objectwise-freeness before proving \ref{thm:TheoremA}. Section~\ref{sec:characterization_of_traced}
briefly records an interesting byproduct of the theory developed for the proofs of the main theorems:
a (we believe) new characterization, up to biequivalence, of the 2-category of traced monoidal
categories which makes no mention of a trace operation.

The appendix---Sections~\ref{sec:app_arrowObs}, \ref{sec:strict_vs_strong}, and \ref{sec:ObjectwiseFree}---contains material that is not essential for establishing the main results of the paper. The purpose of the appendix is to prove the
biequivalence between the 2-category $\MMonCatStrong$ of monoidal categories with arbitrary object
set and strong functors, on the one hand, and the 2-category $\MMonFrObCat$ of monoidal categories
with free object set and strict functors, on the other. We do the same for traced and compact categories, all in
Corollary~\ref{cor:object_frees}.

\section*{Acknowledgments}

Thanks go to Steve Awodey and Ed Morehouse for suggesting we formally connect the operad-algebra
picture in \cite{RupelSpivak} to string diagrams in traced categories. We also thank Mike Shulman
for many useful conversations, and Tobias Fritz, Justin Hilburn, Dmitry Vagner, and Christina
Vasilakopoulou for helpful comments on drafts of this paper. Finally, we thank the anonymous referee whose suggestions were extremely helpful in making this paper more readable.

\chapter{Background on equipments}
  \label{chap:background_equipments}

One of the main technical tools used in our proofs of the main theorems is \emph{monoidal
profunctors}, which are profunctors between monoidal categories with extra structure. When
studying monoidal/traced/compact categories together with their corresponding functors and natural
transformations, a bit of 2-category theory is often useful.  Analogously, when studying monoidal profunctors
between such structured categories together with their relation to functors and natural transformations, it is
helpful to make use of the theory of equipments.

In this section we review the material on equipments that will be needed in the sequel.

\section{Equipments}
  \label{sec:equipments}

A double category is a 2-category-like structure involving horizontal and vertical arrows, as well
as 2-cells. An equipment (sometimes called a \emph{proarrow equipment} or \emph{framed bicategory})
is a double category satisfying a certain fibrancy condition. In this section, we will spell this
out and give two relevant examples. An excellent reference is Shulman's paper \cite{Shulman}; see
also \cite{Wood1} and \cite{Wood2}.

\begin{definition}\label{def:double_cat}
  A \emph{double category}%
  \footnote{
    We will use many flavors of category in this paper, and we attempt to use different fonts to distinguish between them. We denote named
    1-categories, monoidal categories, and operads using bold roman letters, e.g.\ $\ncat{Cob}$, and
    unnamed 1-categories with script, e.g.\ $\cat{C}$. For named 2-categories or bicategories we do
    almost the same, but change the font of the first letter to calligraphic, such as
    $\nncat{T}{rCat}$; for unnamed 2-categories we use (unbold) calligraphic, e.g.\ $\ccat{D}$.
    Finally, for double categories we make the first letter blackboard bold, whether named (e.g.,
    $\dProf$) or unnamed (e.g.\ $\dcat{D}$). A minor exception occurs almost immediately, however: two 1-categories appear as part of the structure of a double category $\dcat{D}$, and we denote them as $\dcat{D}_0,\dcat{D}_1$ rather than using script font.
  }
  $\dcat{D}$ consists of the following data:
  \begin{itemize}
    \item A category $\dcat{D}_0$, which we refer to as the \emph{vertical category} of $\dcat{D}$.
      For any two objects $c,d\in\dcat{D}_0$, we will write $\dcat{D}_0(c,d)$ for the set of
      vertical arrows from $c$ to $d$. We refer to objects of $\dcat{D}_0$ as objects of $\dcat{D}$.
    \item A category $\dcat{D}_1$, equipped with two functors
      $\lframe,\rframe\colon\dcat{D}_1\to\dcat{D}_0$, called the \emph{left frame} and \emph{right
      frame} functors. Given an object $M\in\Ob(\dcat{D}_1)$ with $c=\lframe(M)$ and
      $c'=\rframe(M)$, we say that $M$ is a \emph{proarrow} (or \emph{horizontal arrow}) \emph{from
      $c$ to $c'$} and write $M\colon c\tickar c'$. A morphism $\phi\colon M\to N$ in $\dcat{D}_1$
      is called a \emph{2-cell}, and is drawn as follows, where $f=\lframe(\phi)$ and $f'=\rframe(\phi)$:
      \begin{equation} \begin{tikzcd}
          \label{eqn:2cell}
        c \ar[r,tick,"M" domA] \ar[d,"f"']
          & c'\ar[d,"f'"] \\
        d \ar[r,tick,"N"' codA]
          & d'
        \twocellA{\phi}
      \end{tikzcd} \end{equation}
    \item A \emph{unit} functor $\unit\colon\dcat{D}_0\to\dcat{D}_1$, which is a strict section of
      both $\lframe$ and $\rframe$, i.e.\ $\lframe\circ\unit=\id_{\dcat{D}_0}=\rframe\circ\unit$. We
      will often abuse notation by writing $c$ for the unit proarrow $\unit(c)\colon c\tickar c$,
      and similarly for vertical arrows.
    \item A functor $\odot\colon\dcat{D}_1\times_{\dcat{D}_0}\dcat{D}_1\to\dcat{D}_1$, called
      \emph{horizontal composition}, which is weakly associative and unital in the sense that there
      are coherent unitor and associator isomorphisms. See \cite{Shulman} for more details.
   \end{itemize}
  Given a double category $\dcat{D}$ there is a strict 2-category called the \emph{vertical
  2-category}, denoted $\VVer(\dcat{D})$, whose underlying 1-category is $\dcat{D}_0$ and whose
  2-morphisms $f\Rightarrow f'$ are defined to be 2-cells \eqref{eqn:2cell} where $M=\unit(c)$ and
  $N=\unit(d)$ are unit proarrows. There is also a \emph{horizontal bicategory}, denoted
  $\HHor(\dcat{D})$, whose objects and 1-cells are the objects and horizontal arrows of $\dcat{D}$,
  and whose 2-cells are the 2-cells of $\dcat{D}$ of the form \eqref{eqn:2cell} such that $f=\id_c$
  and $f'=\id_{c'}$.

  A \emph{strong double functor} $F\colon\dcat{C}\to\dcat{D}$ consists of functors
  $F_0\colon\dcat{C}_0\to\dcat{D}_0$ and $F_1\colon\dcat{C}_1\to\dcat{D}_1$ commuting with the
  frames $\lframe$,$\rframe$, which preserve the unit $\unit$ and the horizontal composition $\odot$ up
  to coherent isomorphism.
\end{definition}

Recall that a \emph{fibration} of categories $p\colon\cat{E}\to\cat{B}$ is a functor with a lifting
property: for every $f\colon b'\to b$ in $\cat{B}$ and object $e\in\cat{E}$ with $p(e)=b$, there
exists $e'\to e$ over $f$ that is \emph{cartesian}, i.e.\ universal in an appropriate sense. We
denote fibrations of 1-categories using two-headed arrows $\onto$.

\begin{definition}
    \label{def:equipment}
  An \emph{equipment} is a double category $\dcat{D}$ in which the frame functor
  \[
    (\lframe,\rframe)\colon\dcat{D}_1\onto\dcat{D}_0\times\dcat{D}_0
  \]
  is a fibration. Given vertical morphisms $f\colon c\to d$ and $f'\colon c'\to d'$ together with a proarrow $N\colon
  d\tickar d'$, a cartesian morphism $M\to N$ in $\dcat{D}_1$ over $(f,f')$ is a
  2-cell
  \[ \begin{tikzcd}
    c \ar[r,tick, "M" domA] \ar[d,"f"']
      & c'\ar[d,"f'"] \\
    d \ar[r,tick,"N"' codA]
      & d'
    \twocellalt{A}{\tn{cart}}
  \end{tikzcd} \]
  which we call a \emph{cartesian 2-cell}. We refer to $M$ as the \emph{restriction of $N$ along $f$
  and $f'$}, written $M=N(f,f')$.
\end{definition}

For any vertical morphism $f\colon c\to d$ in an equipment $\dcat{D}$, there are two canonical
proarrows $\comp{f}\colon c\tickar d$ and $\conj{f}\colon d\tickar c$, called respectively the
\emph{companion} and the \emph{conjoint} of $f$, defined by restriction:
\begin{equation}
  \label{dia:comapanion conjoint}
  \begin{tikzcd}
    c \ar[r,tick,"\comp{f}" domA] \ar[d,"f"']
    & d \ar[d,equal] \\
    d \ar[r,tick,"\unit(d)"' codA] & d
    \twocellA{\tn{cart}}
  \end{tikzcd}
  \quad
  \begin{tikzcd}
    d \ar[r,tick,"\conj{f}" domA] \ar[d,equal]
    & c \ar[d,"f"] \\
    d \ar[r,tick,"\unit(d)"' codA] & d.
    \twocellA{\tn{cart}}
  \end{tikzcd}
\end{equation}
In \cite{Shulman}, it is shown that all restrictions can be obtained by composing with companions
and conjoints. In particular, $N(f,f')\iso\comp{f}\odot N\odot\conj{f'}$ for any proarrow $N$.
Moreover, $\comp{f}$ and $\conj{f}$ form an adjunction in $\HHor(\dcat{D})$; we denote the unit and counit by:
\begin{equation}\label{eqn:eta_epsilon}
	\eta_f\colon\unit(c)\to\comp{f}\odot\conj{f}
	\qquad\tn{and}\qquad
	\epsilon_f\colon\conj{f}\odot\comp{f}\to\unit(d)
\end{equation}
\erase{Finally, note that the companion and conjoint can be used to convert any 2-cell to a globular 2-cell in four ways (see \cite[\S 4]{Shulman}), two of which we will use later:
\begin{equation}\label{eqn:globularize}
\begin{tikzcd}
  c\ar[r,tick,"M" domA]\ar[d,"f"']&c'\ar[d,"f'"]\\
  d\ar[r,tick,"N"' codA]&d'
  \twocellA{\phi}
\end{tikzcd}
\quad\leftrightarrow\quad
\begin{tikzcd}
  c\ar[r,tick,"M" domA]\ar[d,equal]&c'\ar[r,tick,"\comp{f'}"]&d'\ar[d,equal]\\
  c\ar[r,tick,"\comp{f}"']&d\ar[r,tick,"N"' codA]&d'
  \twocellA{\comp{\phi}}
\end{tikzcd}
\quad\leftrightarrow\quad
\begin{tikzcd}
  d\ar[d,equal]\ar[r,tick,"\conj{f}"]&c\ar[r,tick,"M" domA]& c'\ar[d,equal]\\
  d\ar[r,tick,"N"' codA]&d'\ar[r,tick,"\conj{f'}"']&c'
  \twocellA{\conj{\phi}}
\end{tikzcd}
\end{equation}}

Recall that a pseudo-pullback of a cospan $A_1\To{f_1}A\From{f_2}A_2$ in a 2-category $\ccat{C}$ is a diagram
\begin{equation}
    \label{eqn:pseudopullback}
  \begin{tikzcd}[sep=small]
    X \ar[r,"g_1"] \ar[d,"g_2"'] \ar[rd,phantom,"\overset{\alpha}{\footnotesize\cong}"]
      & A_2 \ar[d,"f_2"] \\
    A_1 \ar[r,"f_1"']
      & A
  \end{tikzcd}
\end{equation}
where the tuple $(X,g_1,g_2,\alpha)$ is universal, up to equivalence, for data of that shape.
Although this definition makes sense for any 2-category $\ccat{C}$, we will use it only in the
special case described in the next paragraph.

Let $\ccat{C}=\CCat$, the 2-category of small categories.  Suppose $f_2$ is a fibration and that the pullback square \eqref{eqn:pseudopullback}
strictly commutes, i.e.\ that $\alpha$ is the identity. It is a standard fact that a strict pullback
of a fibration along an arbitrary functor is a fibration.  Moreover, this strict pullback
is also a pseudo-pullback. The upshot is that in \eqref{eqn:pseudopullback}, $g_2$ is a pseudo-pullback if and only if for any
strict pullback $g_2'$ of $f_2$ along $f_1$, the induced map $g_2\to g_2'$ is an equivalence of
fibrations.

\begin{definition}
    \label{def:local_equivalence}
  An \emph{equipment functor} is a strong double functor between equipments (see Definition~\ref{def:double_cat}). 
  
  An equipment functor $F\colon\dcat{C}\to\dcat{D}$ is called a \emph{local equivalence} if the
  following (strictly commuting) square is a pseudo-pullback of categories:
  \begin{equation} \begin{tikzcd}
      \label{eqn:local_equiv}
    \dcat{C}_1 \ar[r,"F_1"] \ar[d,two heads,"{(\lframe,\rframe)}"'] \ar[dr,phantom,"\lrcorner" very near start]
      & \dcat{D}_1 \ar[d,two heads,"{(\lframe,\rframe)}"] \\
    \dcat{C}_0\times\dcat{C}_0 \ar[r,"F_0\times F_0"']
      & \dcat{D}_0\times\dcat{D}_0.
  \end{tikzcd} \end{equation}
   If moreover $F_0\colon\dcat{C}_0\to\dcat{D}_0$ is fully faithful, we say that $F$ is a
   \emph{fully faithful local equivalence}.
\end{definition}

\begin{remark}
    \label{rem:strict_vs_pseudo_pullback}
  As discussed above, if the square \eqref{eqn:local_equiv} is a strict pullback, it will be a
  pseudo-pullback, and hence a local equivalence. Any local equivalence can thus be replaced by
  an equivalent strict pullback. We will use this fact often; see 
  Definition~\ref{def:induced_locally_equivalent_equipment}.

  Also note that the frame fibration for $\dcat{C}$ is equivalent to a functor
  $\dcat{C}_0\times\dcat{C}_0\to\CCat$, sending $(c,d)$ to
  $\HHor(\dcat{C})(c,d)$ and similarly for $\dcat{D}$. In this language, $F$ is a local equivalence
  if and only if the induced functors $\HHor(\dcat{C})(c,d)\to\HHor(\dcat{D})(F_0(c),F_0(d))$ are
  equivalences of categories for every pair of objects $(c,d)$. The square \eqref{eqn:local_equiv}
  is a strict pullback precisely when these are isomorphisms of categories.
\end{remark}

\begin{definition}
    \label{def:induced_locally_equivalent_equipment}
  Let $\dcat{D}$ be a double category and $F_0\colon\dcat{C}_0\to\dcat{D}_0$ be a functor. A strict
  pullback of the form \eqref{eqn:local_equiv} defines a double category with vertical category $\dcat{C}_0$, which we denote $F_0^*(\dcat{D})$.
  
  If $\dcat{D}$ is an equipment, $F_0^*(\dcat{D})$ will be one as well since fibrations are stable under
  pullback. In this case we call $F_0^*(\dcat{D})$ the \emph{equipment induced by $F_0$}. By
  Remark~\ref{rem:strict_vs_pseudo_pullback}, the induced equipment functor
  $F_0^*(\dcat{D})\to\dcat{D}$ is a local equivalence.
\end{definition}

The structured categories of interest in this paper can be conveniently organized using the notion of
equipments. The prototype for all of these equipments is that of categories, functors, and
profunctors as in the following example.

\begin{example}
    \label{ex:profunctors}
  The equipment $\dProf$ is a double category whose vertical category $\dProf_0=\Cat$ is the
  category of small 1-categories and functors. Given categories $\cat{C},\cat{C}'\in\dProf_0$, a
  proarrow
  \[ \begin{tikzcd}
    \cat{C} \ar[r,tick,"M"] & \cat{C}'
  \end{tikzcd} \]
  in $\dProf_1$ is a profunctor, i.e.\ a functor $M\colon\op{\cat{C}}\times\cat{C}'\to\Set$. The
  left and right frame functors are given by $\lframe(M)=\cat{C}$ and $\rframe(M)=\cat{C}'$. A
  2-cell $\phi$ in $\dProf$, as to the left, denotes a natural transformation, as to the right, in
  \eqref{eqn:Prof2cells}:
  \begin{equation}
      \label{eqn:Prof2cells}
    \begin{tikzcd}
      \cat{C} \ar[r,tick,"M" domA] \ar[d,"F"']
        & \cat{C}'\ar[d,"F'"] \\
      \cat{D} \ar[r,tick,"N"' codA]
        & \cat{D}'
      \twocellA{\phi}
    \end{tikzcd}
    \hspace{.6in}
    \begin{tikzcd}[column sep=.8em, row sep=5ex]
      \op{\cat{C}}\times \cat{C}' \ar[dr,"M"'] \ar[rr,"\op{F}\times F'"]
        & \ar[d,phantom,"\overset{\phi}{\Rightarrow}" near start]
        & \op{\cat{D}}\times \cat{D}'\ar[dl,"N"] \\
      & \Set.
    \end{tikzcd}
  \end{equation}
  The unit functor $\unit\colon\Cat\to\dProf_1$ sends a category $\cat{C}$ to the hom profunctor
  $\Hom_{\cat{C}}\colon\op{\cat{C}}\times\cat{C}\to\Set$. Given two profunctors
  \[ \begin{tikzcd}
    \cat{C} \ar[r,tick,"M"] & \cat{D} \ar[r,tick,"N"] & \cat{E},
  \end{tikzcd} \]
  define the horizontal composition $M\odot N$ on objects $c\in\cat{C}$ and $e\in\cat{E}$ as the
  reflexive coequalizer of the diagram
  \begin{equation} \begin{tikzcd}
      \label{eqn:coendComp}
    \displaystyle\coprod_{d_1,d_2\in\cat{D}} M(c,d_1)\times\cat{D}(d_1,d_2)\times N(d_2,e)
      \ar[r,shift left] \ar[r,shift right]
    & \displaystyle\coprod_{d\in\cat{D}} M(c,d)\times N(d,e)\ar[l]
  \end{tikzcd} \end{equation}
  where the two rightward maps are given by the right and left actions of $\cat{D}$ on $M$ and $N$
  respectively, and the splitting is given by $\id_d\in\cat{D}(d,d)$. Given a profunctor $M\colon\cat{C}\tickar\cat{D}$ there are
  canonical isomorphisms $\Hom_{\cat{C}}\odot M \iso M \iso M\odot\Hom_{\cat{D}}$ which can be
  viewed as giving an action of $\Hom_{\cat{C}}$ and of $\Hom_{\cat{D}}$ on $M$, from the left and
  right respectively.

  At this point we have given $\dProf$ the structure of a double category. To see that $\dProf$ is
  an equipment, note that from a pair of functors $F\colon\cat{C}\to\cat{D}$,
  $F'\colon\cat{C}'\to\cat{D}'$ and a profunctor $N\colon \cat{D}\tickar\cat{D}'$ we may form the
  composite
  \[\begin{tikzcd}
    \op{\cat{C}}\times\cat{C}' \ar[r,"\op{F}\times F'"]
      &[1.5em] \op{\cat{D}}\times\cat{D}' \ar[r,"N"]
      & \Set,
  \end{tikzcd}\]
  denoted $N(F,F')\colon\cat{C}\tickar\cat{C}'$, such that
  \begin{equation}
      \label{eq:cartesian profunctor morphism}
    \begin{tikzcd}
      \cat{C} \ar[r,tick,"{N(F,F')}" domA] \ar[d,"F"']
        & \cat{C}'\ar[d,"F'"] \\
      \cat{D} \ar[r,tick,"N"' codA]
        & \cat{D}'
      \twocellA{\phi}
    \end{tikzcd}
  \end{equation}
  is a cartesian 2-cell. A simple Yoneda lemma argument yields $\VVer(\dProf)\equiv\CCat$.
\end{example}

\begin{remark}
    \label{rmk:profunctor_as_bimodule}
  There is a strong analogy relating profunctors between categories with bimodules between rings.
  Besides being a useful source of intuition, we can also exploit this analogy to provide a
  convenient notation for working with profunctors.

  If $M\colon\op{\cat{C}}\times\cat{D}\to\Set$ is a profunctor, then for any element $m\in M(c,d)$
  and morphisms $f\colon c'\to c$ and $g\colon d\to d'$, we can write $g\cdot m\in M(c,d')$ and
  $m\cdot f\in M(c',d)$ for the elements $M(\id,g)(m)$ and $M(f,\id)(m)$ respectively. Thus we think
  of the functoriality of $M$ as providing left and right actions of $\cat{D}$ and $\cat{C}$ on the
  elements of $M$. The equations $(g\cdot m)\cdot f = g\cdot (m\cdot f)$, $g'\cdot(g\cdot
  m)=(g'\circ g)\cdot m$, and $(m\cdot f)\cdot f'=m\cdot(f\circ f')$ clearly hold whenever they
  make sense.

  The reflexive coequalizer \eqref{eqn:coendComp} can be easily expressed in this notation: the elements of
  $(M\odot N)(c,e)$ are pairs $m\otimes n$ of elements $m\in M(c,d)$ and $n\in N(d,e)$ for some
  $d\in\cat{D}$ modulo the relation $(m\cdot f)\otimes n = m\otimes(f\cdot n)$, for $f\in\cat{D}$.

  Finally, a 2-cell $\phi$ of the form \eqref{eqn:Prof2cells} is function sending elements $m\in
  M(c,c')$ to elements $\phi(m)\in N(Fc,F'c')$ such that the equation $\phi(g\cdot m\cdot
  g)=F(g)\cdot\phi(m)\cdot F'(f)$ holds whenever it makes sense.
\end{remark}

\section{Monoids and bimodules}
  \label{sec:monoids_bimods}

In any equipment $\dcat{D}$, we can consider monads in the horizontal bicategory $\HHor(\dcat{D})$.
In \cite{Schultz2015}, a definition of \emph{exact equipment} was given axiomatizing a class of
equipments in which there is a close connection between monads in $\HHor(\dcat{D})$ and the objects
of $\dcat{D}$. We will review exact equipments and the results needed here in
Section~\ref{sec:exactness_boff}, but first we collect the necessary definitions and results about monads
in $\HHor(\dcat{D})$, which we refer to as \emph{monoids in $\dcat{D}$}.

\begin{definition}
    \label{def:monoids}
  Denote by $\Mon(\dcat{D})$ the category of monoids in $\dcat{D}$. More precisely, the objects are
  \emph{monoids}: 4-tuples $(c,M,i_M,m_M)$ consisting of an object $c$ of $\dcat{D}$ and a proarrow
  $M\colon c\tickar c$, together with unit and multiplication cells
  \begin{equation}
      \label{eqn:unit_and_mult}
    \begin{tikzcd}
      c \ar[r,tick,"c" domA] \ar[d,equal]
        & c \ar[d,equal] \\
      c \ar[r,tick,"M"' codA] & c
      \twocellA{i_M}
    \end{tikzcd}
    \qquad
    \begin{tikzcd}
      c \ar[r,tick,"M"] \ar[d,equal]
        & |[alias=domA]| c \ar[r,tick,"M"]
        & c \ar[d,equal] \\
      c \ar[rr,tick,"M"' codA]
        && c
      \twocellA{m_M}
    \end{tikzcd}
  \end{equation}
  satisfying the evident unit and associativity axioms. The morphisms are \emph{monoid
  homomorphisms}: pairs $(f,\vec{f}\mspace{2mu})$ consisting of a vertical arrow $f\colon c\to d$ in
  $\dcat{D}$ and a cell
  \[ \begin{tikzcd}
    c \ar[r,tick,"M" domA] \ar[d,"f"']
      & c \ar[d,"f"] \\
    d \ar[r,tick,"N"' codA]
      & d
    \twocellA{\vec{f}}
  \end{tikzcd} \]
  which respects the unit and multiplication cells of $M$ and $N$.
\end{definition}

There is an evident forgetful functor $\MOb\colon\Mon(\dcat{D})\to\dcat{D}_0$ sending a monoid
$M\colon c\tickar c$ to its underlying object $|M|\coloneqq c$. The following result is also in \cite{FioreGambinoKock}.

\begin{lemma}
    \label{lemma:Mon_und_fib}
  Let $\dcat{D}$ be an equipment. The forgetful functor $\MOb\colon\Mon(\dcat{D})\to\dcat{D}_0$ is a
  fibration and there is a morphism of fibrations
  \[ \begin{tikzcd}[column sep=tiny]
    \Mon(\dcat{D}) \ar[d,two heads,"\MOb"'] \ar[r]
      & \dcat{D}_1 \ar[d,two heads,"{(\lframe,\rframe)}"] \\
    \dcat{D}_0 \ar[r,"\Delta"']&\dcat{D}_0\times\dcat{D}_0.
  \end{tikzcd} \]
\end{lemma}
\begin{proof}
  Let $f\colon c\to d$ be a vertical morphism of $\dcat{D}$ and $N\colon d\tickar d$ a monoid in
  $\dcat{D}$. Since the 2-cell defining the restriction of $N$ along $f$ is cartesian, there is an
  induced monoid structure on $N(f,f)$ which in particular makes this cartesian 2-cell a monoid
  homomorphism. The result follows.
\end{proof}

\begin{lemma}
    \label{lem:Mon_pullback}
  For a local equivalence $F\colon\dcat{C}\to\dcat{D}$, the induced square
  \[ \begin{tikzcd}[column sep=large]
    \Mon(\dcat{C}) \ar[r,"\Mon(F)"] \ar[d,two heads,"\MOb"'] \ar[dr,phantom,"\lrcorner" very near start]
      & \Mon(\dcat{D}) \ar[d,two heads,"\MOb"] \\
    \dcat{C}_0 \ar[r,"F_0"']
      & \dcat{D}_0
  \end{tikzcd} \]
  is a pseudo-pullback of categories.
\end{lemma}
\begin{proof}
  By Remark~\ref{rem:strict_vs_pseudo_pullback}, we may assume that the pullback \eqref{eqn:local_equiv} in
  Definition~\ref{def:local_equivalence}, realizing $F\colon\dcat{C}\to\dcat{D}$ as a local
  equivalence, is strict. It is then straightforward to check directly that the above square is again
  a strict pullback and hence a pseudo-pullback.
\end{proof}

In all our cases of interest, $\Mon(\dcat{D})$ becomes the vertical part of another equipment. 
The following is a standard construction; see \cite{Shulman}.

\begin{definition}
    \label{def:monoids_and_modules}
  Let $\dcat{D}$ be an equipment with local reflexive coequalizers, i.e.\ such that each 1-category
  $\HHor(\dcat{D})(c,d)$ has reflexive coequalizers and $\odot$ preserves reflexive coequalizers in each variable. The
  equipment $\dMod(\dcat{D})$ of \emph{monoids and bimodules} is defined as follows:
  \begin{itemize}
    \item The vertical category $\dMod(\dcat{D})_0$ is the category $\Mon(\dcat{D})$ of monoids in
      $\dcat{D}$.
    \item The proarrows $B\colon M\tickar N$ are \emph{bimodules}: triples $(B,l_B,r_B)$
      consisting of a proarrow $B\colon c\tickar d$ in $\dcat{D}$ and cells
      \begin{equation*}
        \begin{tikzcd}
          c \ar[r,tick,"M"] \ar[d,equal]
            & |[alias=domA]| c \ar[r,tick,"B"]
            & d \ar[d,equal] \\
          c \ar[rr,tick,"B"' codA]
            && d
          \twocellA{l_B}
        \end{tikzcd}
        \qquad
        \begin{tikzcd}
          c \ar[r,tick,"B"] \ar[d,equal]
            & |[alias=domA]| d \ar[r,tick,"N"]
            & d \ar[d,equal] \\
          c \ar[rr,tick,"B"' codA]
          && d
          \twocellA{r_B}
        \end{tikzcd}
      \end{equation*}
      satisfying evident monoid action axioms.
    \item The horizontal composition $B_1\otimes_{M'} B_2$ of bimodules $B_1\colon M\tickar M'$ and
      $B_2\colon M'\tickar M''$ is given by the reflexive coequalizer in $\HHor(\dcat{D})(M,M'')$
      \[
      \begin{tikzcd}
        B_1\odot M'\odot B_2 \ar[r,shift left]\ar[r,shift right]
        & B_1\odot B_2 \ar[r]\ar[l]
        & B_1\otimes_{M'} B_2
      \end{tikzcd}
      \]
      together with the evident left $M$ and right $M''$ actions. Above, the splitting map $B_1\odot B_2\to B_1\odot M'\odot B_2$ comes from the unit $i_{M'}$ of the monoid.
    \item The 2-cells are \emph{bimodule homomorphisms}: cells in $\dcat{D}$
      \[ \begin{tikzcd}
        c \ar[r,tick,"B"] \ar[d,"f"' domA]
          & d \ar[d,"f'" codA] \\
        c' \ar[r,tick,"B'"']
          & d'
        \twocellA{\phi}
      \end{tikzcd} \]
      which are compatible with the various left and right monoid actions.
  \end{itemize}
  We will write $\Bimod{M}{N}\coloneqq\HHor(\dMod(\dcat{D}))(M,N)$ to denote the 1-category of $(M,N)$-bimodules and bimodule morphisms.
\end{definition}

The forgetful functor $\MOb\colon\Mon(\dcat{D})\to\dcat{D}_0$ extends to a \emph{lax} equipment
functor $\MOb\colon\dMod(\dcat{D})\to\dcat{D}$. We have not defined lax equipment functors---because we do not use them---and in particular we will we not use the equipment version of $\MOb$.

More importantly for our work, there is a local equivalence
$U\colon\dcat{D}\to\dMod(\dcat{D})$ sending $c$ to the unit $c\tickar c$ with the trivial monoid
structure. If $F\colon\dcat{C}\to\dcat{D}$ is an equipment functor, then there is an evident
equipment functor $\dMod(F)\colon\dMod(\dcat{C})\to\dMod(\dcat{D})$. In fact, we have the following
which is immediate from the definitions.

\begin{lemma}
    \label{lemma:FFLE_Mod}
  If $F\colon\dcat{C}\to\dcat{D}$ is a local equivalence, then so is the induced functor
  $\dMod(F)\colon\dMod(\dcat{C})\to\dMod(\dcat{D})$. If $F$ is a fully
  faithful local equivalence, then so is $\dMod(F)$.
\end{lemma}

\section{Exact equipments and $\bo$, $\ff$ factorization}
  \label{sec:exactness_boff}
For many equipments $\dcat{D}$, monoids in $\dcat{D}$ are category-like objects where the vertical
and horizontal arrows in $\dMod(\dcat{D})$ are (generalized or structured) functors and profunctors.
The canonical example is the equipment of spans of sets whose objects, vertical, and horizontal
arrows are sets, functions, and spans respectively. In this case, the equipment of monoids is
equivalent to $\dProf$, which may be easily verified directly (or see \cite{Leinster}).

One natural question is thus: what are monoids in equipments like $\dProf$ whose objects are already
category-like? One might guess they would be higher-dimensional structures like double categories.
However, it turns out that in these cases where $\dcat{D}$ is an equipment of category-like objects,
the equipment $\dMod(\dcat{D})$ of monoids in $\dcat{D}$ is very closely related to $\dcat{D}$
itself. In other words, monoids in an equipment are category-like objects, but taking monoids in an
already category-like equipment doesn't change much.

The article \cite{Schultz2015} proposed definitions of \emph{regular} and \emph{exact} equipments.
For this paper we will only be concerned with exact equipments, which are an attempt to axiomatize
the above property of equipments whose objects are category-like. In particular, given any monoid
$M$ in an exact equipment, there is a corresponding object called the \emph{collapse} of $M$
satisfying a certain vertical universal property. This collapse construction has formal similarities
to the quotient of an equivalence relation, and it is because of this (informal) analogy that the
names regular and exact where chosen. From this perspective, $\dMod(\dcat{D})$ can be thought of as
the exact completion of the equipment $\dcat{D}$.

In this section we recall the definition of an exact equipment from \cite{Schultz2015} and review
some basic properties. In Section~\ref{sec:exactness_proofs} we will prove that the equipments of
interest are exact. This will allow us to use the collapse construction as a crucial step in the connection
between lax monoidal set valued functors---which we view as proarrows in exact equipments---with
(traced or compact) monoidal categories---which will be objects of those exact equipments. We will
also need to make use of fully-faithful/bijective-on-objects factorization systems, which are shown in
this section to exist in any exact equipment. (In \cite{Schultz2015} it is shown that this
factorization system exists in any \emph{regular} equipment, as the regular-epi/mono
factorization system exists in any regular category).

\begin{definition}
    \label{def:embedding}
  Let $M\colon c\tickar c$ be a monoid in an equipment $\dcat{D}$. An \emph{embedding} of $M$ into
  an object $x\in\dcat{D}_0$ is a monoid homomorphism $(f,\vec{f}\mspace{2mu})$ from $M$ to the
  trivial monoid on $x$:
  \[ \begin{tikzcd}
    c \ar[r,tick,"M" domA] \ar[d,"f"']
      & c \ar[d,"f"] \\
    x \ar[r,tick,"x"' codA]
      & x.
    \twocellA{\vec{f}}
  \end{tikzcd} \]
  We will sometimes write an embedding as $(f,\vec{f}\mspace{2mu})\colon(c,M)\to x$, or even just
  $f\colon M\to x$ when clear from context. We will write $\Emb(M,x)$ for the set of embeddings from
  $M$ to $x$. This defines a functor $\Emb\colon\op{\Mon(\dcat{D})}\times\dcat{D}_0\to\Set$.
\end{definition}

\begin{lemma}
    \label{lemma:embed_for_LE}
  Suppose that $F\colon\dcat{C}\to\dcat{D}$ is a local equivalence induced by
  $F_0\colon\dcat{C}_0\to\dcat{D}_0$. Suppose $M\in\Mon(\dcat{C})$ is a monoid and $x\in\dcat{C}_0$
  is an object. For $N=\Mon(F)(M)$ and $y=F_0(x)$ we have a pullback square in $\Set$, natural in
  $M$ and $x$:
  \[ \begin{tikzcd}
    \Emb_{\dcat{C}}(M,x)\ar[r]\ar[d]\ar[rd,phantom,"\lrcorner" very near start]
      & \Emb_{\dcat{D}}(N,y)\ar[d] \\
    \dcat{C}_0(|M|,x)\ar[r]
      & \dcat{D}_0(|N|,y).
  \end{tikzcd} \]
\end{lemma}

\begin{definition}
  Let $M\colon c\tickar c$ be a monoid in an equipment $\dcat{D}$. A \emph{collapse} of $M$ is
  defined to be a universal embedding of $M$. That is, a collapse of $M$ is an object
  $\Col{M}\in\dcat{D}_0$ together with an embedding
  \[ \begin{tikzcd}
    c \ar[r,tick,"M" domA] \ar[d,"i_M"']
      & c \ar[d,"i_M"] \\
    \Col{M} \ar[r,tick,"\Col{M}"' codA]
      & \Col{M}
    \twocellA{\vec{\imath}_M}
  \end{tikzcd} \]
  such that any other embedding of $M$ factors uniquely through $\vec{\imath}_M$:
  \begin{equation}\label{eqn:universal_embedding}
    \begin{tikzcd}
      c \ar[r,tick,"M" domA] \ar[d,"f"']
        & c \ar[d,"f"] \\
      x \ar[r,tick,"x"']
        & x
      \twocellA{\vec{f}}
    \end{tikzcd}
    \quad = \quad
    \begin{tikzcd}
      c \ar[r,tick,"M" domA] \ar[d,"i_M"']
        & c \ar[d,"i_M"] \\
      \Col{M} \ar[r,tick,"\Col{M}"' {domB,codA}] \ar[d,"\tilde{f}"']
        & \Col{M} \ar[d,"\tilde{f}"] \\
      x \ar[r,tick,"x"' codB]
        & x.
      \twocellA{\vec{\imath}_M}
      \twocellB[pos=.6]{\id_{\tilde{f}}}
    \end{tikzcd}
  \end{equation}
  In other words, $\Col{M}$ represents the functor $\Emb(M,\textrm{--})\colon\dcat{D}_0\to\Set$.
\end{definition}

\begin{example}
    \label{ex:monoid_in_Prof}
  Collapses of monoids always exist in $\dProf$, and this is the prototypical example of collapse.

  Consider a monoid $M\colon \cat{C}\tickar \cat{C}$ in $\dProf$. The unit is a profunctor morphism
  $i\colon\Hom_{\cat{C}}\to M$. So for any $f\colon c\to d$ in $\cat{C}$ there is an element
  $i(f)\in M(c,d)$ such that
  \begin{equation}
      \label{eq:Prof_monoid_unit}
    g\cdot i(f)\cdot h = i(g\circ f\circ h)
  \end{equation}
  whenever this makes sense.

  The multiplication $M\odot M\to M$ is an associative operation assigning to any elements $m_1\in M(c,d)$ and
  $m_2\in M(d,e)$ an element $m_2\bullet m_1\in M(c,e)$ satisfying the
  following equations whenever they make sense:
  \begin{gather}
    (f\cdot m_2)\bullet(m_1\cdot h) = f\cdot(m_2\bullet m_1)\cdot h
      \label{eq:Prof_monoid_A}
    \\ (m_3\cdot g)\bullet m_1 = m_3\bullet(g\cdot m_1)
      \label{eq:Prof_monoid_B}
    \\ m\bullet i(f) = m\cdot f
          \quad\text{and}\quad
      i(g)\bullet m = g\cdot m
      \label{eq:Prof_monoid_C}
  \end{gather}
  Specifically, equations \eqref{eq:Prof_monoid_A} and \eqref{eq:Prof_monoid_B} simply say that
  $\bullet$ is a well-defined morphism $M\odot M\to M$, while \eqref{eq:Prof_monoid_C} says that
  $\bullet$ is unital with respect to $i$.

  The collapse $\Col{M}$ of $M$ is the category with the same objects as $\cat{C}$, with morphisms
  $\Col{M}(c,d)\coloneqq M(c,d)$, and with composition given by $\bullet$. The unit $i$ of $M$ gives
  a functor $i_M\colon\cat{C}\to\Col{M}$.
\end{example}

\begin{remark}
    \label{rem:suffices_for_monoid}
  The equations \eqref{eq:Prof_monoid_unit}--\eqref{eq:Prof_monoid_C} are actually overdetermined.
  It is easy to see that equations \eqref{eq:Prof_monoid_A} and \eqref{eq:Prof_monoid_B} follow from
  \eqref{eq:Prof_monoid_C} and the associativity of $\bullet$. Thus, when proving that
  $\bullet\colon M\odot M\to M$ and $i\colon\Hom_{\cat{C}}\to M$ form a monoid, it suffices to prove
  \eqref{eq:Prof_monoid_unit}, \eqref{eq:Prof_monoid_C}, and associativity of $\bullet$. These
  observations will be used to slightly simplify the proof of
  Proposition~\ref{prop:unit_implies_monoid}.
\end{remark}

\begin{remark}
    \label{rem:canonical_actions}
  For any monoid $M\colon c\tickar c$, the companion $\comp{i}_M\colon c\tickar \Col{M}$ (resp.\ the
  conjoint $\conj{i}_M\colon \Col{M}\tickar c$) of the embedding $i_M\colon c\to\Col{M}$ has the
  structure of a left (resp.\ right) $M$-module. Indeed, the horizontal composition of
  $\vec{\imath}_M$ and the left hand cartesian 2-cell from \eqref{dia:comapanion conjoint} defining
  $\comp{i}_M$ factors uniquely through some $l_{\comp{i}_M}$, as follows:
  \[ \begin{tikzcd}
    c \ar[r,tick,"M" domA] \ar[d,"i_M"']
      & c \ar[r,tick, "\comp{i}_M" domB] \ar[d,"i_M"']
      & \Col{M} \ar[d,equal] \\
    \Col{M} \ar[r,tick,"\Col{M}"' codA] \ar[d,equal]
      & |[alias=domC]| \Col{M} \ar[r,tick,"\Col{M}"' codB]
      & \Col{M} \ar[d,equal]\\
    \Col{M} \ar[rr,tick,"\Col{M}"' codC]
      && \Col{M}
    \twocellA{\vec{\imath}_M}
    \twocellB{\tn{cart}}
    \twocelliso{C}{}
  \end{tikzcd}
  \quad = \quad
  \begin{tikzcd}
    c \ar[r,tick,"M"] \ar[d,equal]
      & |[alias=domA]| c \ar[r,tick,"\comp{i}_M"]
      & \Col{M} \ar[d,equal] \\
    c \ar[rr,tick,"\comp{i}_M"' {codA,domB}] \ar[d,"i_M"']
      && \Col{M} \ar[d,equal]\\
    \Col{M} \ar[rr,tick,"\Col{M}"' codB]
      && \Col{M}.
    \twocellA{l_{\comp{i}_M}}
    \twocellB{\tn{cart}}
  \end{tikzcd}\]
  The right $M$-action on $\conj{i}_M$ is obtained similarly.
\erase{=======
  structure of a left (resp.\ right) $M$-module. Indeed, these are given by globularizing the collapse
  2-cell $\vec{\imath}_M$ as in \eqref{eqn:globularize}:\todo{I don't understand how to construct these diagrams.  It
  seems to me that we should consider the horizontal composition of $\vec{\imath}_M$ with the
  cartesian two-cell defining $\comp{i}_M$ to get a two-cell into $\Col{M}$, which then factors
  through $\comp{i}_M$ using the cartesian property.}
  \[
    \begin{tikzcd}
      c \ar[r,tick,"M"] \ar[d,equal]
        & |[alias=domA]| c \ar[r,tick,"\comp{i}_M"]
        & \Col{M} \ar[d,equal] \\
      c \ar[r,tick,"\comp{i}_M"']
        & |[alias=codA]| \Col{M} \ar[r,tick,"\Col{M}"']
        & \Col{M}
      \twocellA{}
    \end{tikzcd}
    \qquad\qquad
    \begin{tikzcd}
      \Col{M} \ar[r,tick,"\conj{i}_M"] \ar[d,equal]
        & |[alias=domA]| c \ar[r,tick,"M"]
        & c \ar[d,equal] \\
      \Col{M} \ar[r,tick,"\Col{M}"']
        & |[alias=codA]| \Col{M} \ar[r,tick,"\conj{i}_M"']
        & c
      \twocellA{}
    \end{tikzcd}
  \]}
\end{remark}

\begin{lemma}
  Let $M\colon c\tickar c$ and $N\colon d\tickar d$ be monoids in an equipment $\dcat{D}$, and assume
  they admit collapses $\Col{M}$ and $\Col{N}$, respectively. Then restriction induces a functor
  \[
    \HHor(\dcat{D})\big(\Col{M},\Col{N}\big)\to\Bimod{M}{N}.
  \]
\end{lemma}
\begin{proof}
  For a proarrow $X\colon\Col{M}\tickar\Col{N}$ of $\dcat{D}$, define $\tilde X\colon c\tickar d$ by
  the cartesian 2-cell
  \begin{equation}
      \label{eq:bimodule}
    \begin{tikzcd}
      c \ar[r,tick, "\tilde X" domA] \ar[d,"i_M"']
        & d\ar[d,"i_N"] \\
      \Col{M} \ar[r,tick,"X"' codA]
        & \Col{N}.
      \twocellA{\tn{cart}}
    \end{tikzcd}
  \end{equation}
  Then a 2-cell $X\Rightarrow Y$ immediately lifts to a 2-cell $\tilde X\Rightarrow \tilde Y$. Since
  \eqref{eq:bimodule} is cartesian, we obtain an equality
  \[ \begin{tikzcd}
    c \ar[r,tick,"M" domA] \ar[d,"i_M"']
      & c \ar[r,tick, "\tilde X" domB] \ar[d,"i_M"']
      & d\ar[d,"i_N"] \\
    \Col{M} \ar[r,tick,"\Col{M}"' codA] \ar[d,equal]
      & |[alias=domC]| \Col{M} \ar[r,tick,"X"' codB]
      & \Col{N} \ar[d,equal] \\
    \Col{M} \ar[rr,tick,"X"' codC]
      && \Col{N}
    \twocellA{\vec{\imath}_M}
    \twocellB{\tn{cart}}
    \twocelliso{C}{}
  \end{tikzcd}
  \quad = \quad
  \begin{tikzcd}
    c \ar[r,tick,"M"] \ar[d,equal]
      & |[alias=domA]| c \ar[r,tick,"\tilde X"]
      & d \ar[d,equal] \\
    c \ar[rr,tick,"\tilde X"' {codA,domB}] \ar[d,"i_M"']
      && d \ar[d,"i_N"]\\
    \Col{M} \ar[rr,tick,"X"' codB]
      && \Col{N}
    \twocellA{l_{\tilde X}}
    \twocellB{\tn{cart}}
  \end{tikzcd}\]
  giving the action of $M$ on $\tilde X$. The action $r_{\tilde X}$ of $N$ on $\tilde X$ is obtained
  similarly, and one easily checks the axioms making $\tilde X$ an $(M,N)$-bimodule.
\end{proof}

\begin{definition}\cite[Proposition 5.4]{Schultz2015}
    \label{def:exact_equipment}
  An equipment $\dcat{D}$ is \emph{exact} if the following hold:
  \begin{enumerate}
    \item every monoid $M\colon c\tickar c$ has a collapse $\Col{M}$ with $\vec{\imath}_M$
      cartesian;
    \item for every pair of monoids $M$ and $N$ the restriction functor
    \begin{equation}
        \label{eqn:exact_Hor_Bimod}
      \HHor(\dcat{D})\big(\Col{M},\Col{N}\big)\To{\raisebox{-1ex}{$\equiv$}}\Bimod{M}{N}
    \end{equation}
    is an equivalence of categories.
  \end{enumerate}
\end{definition}

\begin{remark}
    \label{rmk:exact_equiv_fibrations}
  The restriction functor \eqref{eqn:exact_Hor_Bimod} is clearly natural, giving a natural
  equivalence between pseudo-functors $\op{\Mon(\dcat{D})}\times\op{\Mon(\dcat{D})}\to\CCat$.  
  Equivalently this gives an equivalence of fibrations, the inverse of which gives rise to a strictly-commuting
  pseudo-pullback square
  \begin{equation*}
    \begin{tikzcd}
      \dMod(\dcat{D})_1 \ar[d,two heads,"{(\lframe,\rframe)}"'] \ar[r]
      & \dcat{D}_1 \ar[d,two heads,"{(\lframe,\rframe)}"] \\
      \Mon(\dcat{D})\times\Mon(\dcat{D}) \ar[r,"\Col{\textrm{--}}\times\Col{\textrm{--}}"']
      & \dcat{D}_0\times\dcat{D}_0.
    \end{tikzcd}
  \end{equation*}
  We will show in Proposition~\ref{prop:collapse_local_equivalence} that (under mild hypotheses)
  this preserves horizontal composition, thus defining a double functor and hence a local equivalence.
\end{remark}

\begin{example}
  It was proven in \cite[Proposition~5.2]{Schultz2015} that for any equipment $\dcat{D}$, its
  equipment $\dMod(\dcat{D})$ of monoids and bimodules is exact. Thus $\dProf$ is exact, since there
  is an equivalence $\dProf\cong\dMod(\dSpan)$, where $\dSpan$ is the equipment of spans in $\Set$;
  see \cite{Shulman}.
\end{example}

Exact equipments arising in practice almost always have local reflexive coequalizers, and in this
case it is possible to simplify the definition, as we show in Proposition~\ref{prop:exact_coeqs}.
Recall from Remark~\ref{rem:canonical_actions} the natural $M$-module structures on the companion
$\comp{i}_M\colon c\tickar \Col{M}$ and conjoint $\conj{i}_M\colon \Col{M}\tickar c$ of the collapse
embedding $i_M\colon c\to\Col{M}$. Recall also the notation $\unit(a)$ from Definition~\ref{def:double_cat}, and $\eta,\epsilon$ from \eqref{eqn:eta_epsilon}.

\begin{proposition}
    \label{prop:exact_coeqs}
  Suppose $\dcat{D}$ is an equipment with local reflexive coequalizers which satisfies Condition 1
  of Definition~\ref{def:exact_equipment}. Then $\dcat{D}$ satisfies Condition 2 if and only if for
  every monoid $M\colon c\tickar c$, the following diagram is a reflexive coequalizer in
  $\HHor(\dcat{D})(\Col{M},\Col{M})$:
  \begin{equation}
      \label{eqn:exact_coeq}
    \begin{tikzcd}
      \conj{i}_M\odot\comp{i}_M\odot\conj{i}_M\odot\comp{i}_M
        \ar[r,shift left=.32cm,"\epsilon_{i_M}\odot\conj{i}_M\odot\comp{i}_M"]
        \ar[r,shift right=.32cm,"\conj{i}_M\odot\comp{i}_M\odot\epsilon_{i_M}"']
      &[3.7em] \conj{i}_M\odot\comp{i}_M
        \ar[r,"\epsilon_{i_M}"]
        \ar[l,"\conj{i}_M\odot\eta_{i_M}\odot\comp{i}_M"{description,inner sep=1.5pt}]
      & \unit\Col{M}
    \end{tikzcd}
  \end{equation}
  or, equivalently, $\conj{i}_M\otimes_M\comp{i}_M\iso\unit\Col{M}$.
\end{proposition}
\begin{proof}
  By Condition~1 of Definition~\ref{def:exact_equipment}, we have $M\iso\comp{i}_M\odot\conj{i}_M$,
  so the final equivalence is just the definition of horizontal composition in $\dMod(\dcat{D})$;
  see Definition~\ref{def:monoids_and_modules} and Remark~\ref{rem:canonical_actions}. Below we will
  use the fact that $\otimes$ is defined as a reflexive
  coequalizer, and that, by definition of $\dcat{D}$ having local reflexive coequalizers, $\odot$
  preserves reflexive coequalizers in each variable. Finally, note that the restriction functor \eqref{eqn:exact_Hor_Bimod} is isomorphic
  to the functor $B\mapsto\comp{i}_M\odot B\odot\conj{i}_N$, with the left and right actions given
  by the left action of $M$ on $\comp{i}_M$ and right action of $N$ on $\conj{i}_N$.

  Assuming $\conj{i}_M\otimes_M\comp{i}_M\iso\unit\Col{M}$, we can construct an inverse to this
  restriction functor, sending an $(M,N)$-bimodule $B$ to $\conj{i}_M\otimes_M
  B\otimes_N\comp{i}_N$. It is easy to check that this gives an equivalence of categories:
  \begin{align*}
    \conj{i}_M\otimes_M (\comp{i}_M\odot B\odot\conj{i}_N)\otimes_N\comp{i}_N
      &\iso (\conj{i}_M\otimes_M\comp{i}_M)\odot B\odot (\conj{i}_N\otimes_N\comp{i}_N) \\
      &\iso \unit\Col{M}\odot B\odot \unit\Col{N} \\
      &\iso B \\
    \intertext{and}
    \comp{i}_M\odot (\conj{i}_M\otimes_M B\otimes_N\comp{i}_N)\odot \conj{i}_N
      &\iso (\comp{i}_M\odot\conj{i}_M)\otimes_M B\otimes_N (\comp{i}_N\odot\conj{i}_N) \\
      &\iso M\otimes_M B\otimes_N N \\
      &\iso B.
  \end{align*}

  Conversely, assuming the functor \eqref{eqn:exact_Hor_Bimod} is an equivalence of categories, then
  we can prove that $\conj{i}_M\otimes_M\comp{i}_M\iso\unit\Col{M}$ is an isomorphism by first
  applying the restriction functor:
  \begin{align*}
    \comp{i}_M\odot(\conj{i}_M\otimes_M\comp{i}_M)\odot\conj{i}_M
      &\iso (\comp{i}_M\odot\conj{i}_M)\otimes_M(\comp{i}_M\odot\conj{i}_M) \\
      &\iso M\otimes_M M \\
      &\iso M \\
      &\iso \comp{i}_M\odot\conj{i}_M \\
      &\iso \comp{i}_M\odot\unit\Col{M}\odot\conj{i}_M. \qedhere
  \end{align*}
\end{proof}

\begin{proposition}
    \label{prop:collapse_local_equivalence}
  If $\dcat{D}$ is an exact equipment with local reflexive coequalizers, then collapse induces an equipment
  functor $\ColDash\colon\dMod(\dcat{D})\to\dcat{D}$ which is a local equivalence.
\end{proposition}
\begin{proof}
  It is easy to use the universal property of collapse to construct, from any monoid homomorphism
  $(f,\vec{f}\mspace{2mu})\colon(c,M)\to(d,N)$, a vertical morphism $\Col{f}\colon\Col{M}\to\Col{N}$
  in $\dcat{D}$, thus defining a functor $\Mon(\dcat{D})\to\dcat{D}_0$.

  The functor $\ColDash$ is defined on horizontal arrows and 2-cells as in
  Remark~\ref{rmk:exact_equiv_fibrations}. It is straightforward to verify that this is a strong
  double functor, and hence a local equivalence, using the method of the proof of
  Proposition~\ref{prop:exact_coeqs}.
\end{proof}

With these definitions in place we can now introduce two distinguished classes of vertical morphisms
in an equipment $\dcat{D}$. When $\dcat{D}$ is exact, these will become the left and right classes in an orthogonal
factorization system on $\VVer(\dcat{D})$.

\begin{definition}\cite[Definitions~4.3~and~4.5]{Schultz2015}
    \label{def:boff}
  Let $\dcat{D}$ be an equipment and $f\colon c\to d$ a vertical morphism of $\dcat{D}$. Consider
  the restriction square and unit square shown below:
  \begin{equation*}
    \begin{tikzcd}
      c \ar[r,tick,"{d(f,f)}" domA] \ar[d,"f"']
      & c \ar[d,"f"]
      \\
      d \ar[r,tick,"d"' codA]
      & d
      \twocellA{\tn{cart}}
    \end{tikzcd}
    \qquad\qquad
    \begin{tikzcd}
      c \ar[r,tick,"c" domA] \ar[d,"f"']
      & c \ar[d,"f"]
      \\
      d \ar[r,tick,"d"' codA]
      & d
      \twocellA{\vec{\id_f}}
    \end{tikzcd}
  \end{equation*}
  We say that $f$ is $\bo$ if the restriction square, where $d(f,f)$ has the induced monoid
  structure, is a collapse. We say that $f$ is $\ff$ if the unit square is cartesian.
\end{definition}

In Section~\ref{sec:monoidal_profunctors} we will define equipments of profunctors on monoidal
categories, and we will verify their exactness directly in Section~\ref{sec:exactness_proofs}. The key
ingredient in verifying that the equipment of traced profunctors is exact will be orthogonal
factorization systems. Thus we briefly recall the notion of orthogonal factorization systems for
1-categories and strict 2-categories. Additional background on orthogonal factorization systems can
be found in \cite[Chapter 5.5]{BorceuxV1}. The main result below is that exact equipments admit
orthogonal factorization systems.

\begin{definition}
    \label{def:orthogonal}
  Let $\cat{V}$ be either $\Set$ or $\Cat$, and suppose that $\cat{C}$ is a $\cat{V}$-enriched
  category. An \emph{orthogonal factorization system in $\cat{C}$} consists of two distinguished
  classes of morphisms, $(\cat{L},\cat{R})$, with the following properties:
  \begin{itemize}
    \item Each morphism $f\in\cat{C}$ factors as $f=e\circ m$, where $m\in\cat{L}$ and
      $e\in\cat{R}$.
    \item If $m\colon a\to b$ in $\cat{L}$ and $e\colon c\to d$ in $\cat{R}$, then the left-hand
      square below is a pullback in $\cat{V}$:
      \begin{equation}
          \label{eqn:OrthFactSys}
        \begin{tikzcd}
          \cat{C}(b,c)\ar[r]\ar[d]\ar[rd,phantom,"\lrcorner" very near start]&\cat{C}(a,c)\ar[d,"e\circ -"]\\
          \cat{C}(b,d)\ar[r,"-\circ m"']&\cat{C}(a,d)
        \end{tikzcd}
        \hspace{.9in}
        \begin{tikzcd}
          a\ar[r,"\forall"]\ar[d,two heads,"m"']&c\ar[d,hook, "e"]\\
          b\ar[r,"\forall"']\ar[ur,dashed,"\exists!"]&d
        \end{tikzcd}
      \end{equation}
      In particular, for all solid arrow squares, as in the right-hand diagram, there exists a
      unique diagonal filler. We say that $m$ is ``left-orthogonal'' to $e$, or that $e$ is
      ``right-orthogonal'' to $m$, and denote this relation as $m\boxslash e$.
    \item If $m\boxslash e$ for all $e\in\cat{R}$, then $m\in\cat{L}$. Likewise, if $m\boxslash e$
      for all $m\in\cat{L}$, then $e\in\cat{R}$.
  \end{itemize}
  As shown, we often indicate morphisms in $\cat{L}$ using a two-headed arrow and morphisms in
  $\cat{R}$ using a hooked arrow.%
  \footnote{
    We sometimes also use the two-headed arrow symbol $\twoheadrightarrow$
    to indicate fibrations of categories (e.g.\ as we did in \ref{thm:TheoremB} or
    when defining the frame fibration for equipments, Definition~\ref{def:equipment}). Whether
    we mean a $\bo$ map in an equipment or a fibration of categories should be clear from context.
  }
\end{definition}

\begin{theorem}\cite[Theorem~4.17]{Schultz2015}
    \label{thm:orthogonal}
  If an equipment $\dcat{D}$ is exact, then the vertical 2-category $\VVer(\dcat{D})$ admits a
  2-orthogonal factorization system $(\bo,\ff)$ as in Definition~\ref{def:boff}. In particular,
  there is an orthogonal factorization system $(\bo,\ff)$ on the vertical 1-category $\dcat{D}_0$.
\end{theorem}

In an exact equipment, there is a close connection between monoids and $\bo$ morphisms. This
connection is formalized in Theorem~\ref{thm:Mod_vs_bo} below, which is a key ingredient in the
proofs of our main theorems.

\begin{definition}
  Let $\dcat{D}$ be an exact equipment. We define the equipment $\dcat{D}^{\bo}$ as follows: the
  vertical category $\dcat{D}_0^{\bo}\ss\dcat{D}_0^{\rightarrow}$ is the full subcategory of the
  arrow category of $\dcat{D}_0$ spanned by the arrows in the class $\bo$. As such, we have functors
  $\dom,\cod\colon\dcat{D}_0^{\bo}\to\dcat{D}_0$. The rest of the structure of $\dcat{D}^{\bo}$ is
  defined by setting $\dcat{D}^{\bo}\coloneqq\cod^*\dcat{D}$, i.e.\ by the following strict pullback of categories (see Definition~\ref{def:induced_locally_equivalent_equipment}):
  \[ \begin{tikzcd}[column sep=large]
    \dcat{D}_1^{\bo} \ar[r] \ar[d, two heads] \ar[dr,phantom,"\lrcorner" very near start]
      & \dcat{D}_1 \ar[d, two heads] \\
    \dcat{D}_0^{\bo}\times\dcat{D}_0^{\bo} \ar[r,"\cod\times\cod"']
      & \dcat{D}_0\times\dcat{D}_0.
  \end{tikzcd} \]
\end{definition}

Section~\ref{chap:equipments_monoidal_profunctors} begins with an outline of our proof of
\ref{thm:TheoremA} as a sequence of equivalences. The last two results of the present section
provide the step in that sequence which is not about traced or compact categories specifically, but
in fact holds for any exact equipment. They form the bridge we need between profunctors and
bijective-on-objects functors. They can be seen as a stronger statement of the idea (mentioned in
the introduction to this section) that in an exact equipment $\dcat{D}$, there is a very close
connection between monoids in $\dcat{D}$ and objects of $\dcat{D}$.

\begin{proposition}
    \label{prop:Mon_vs_bo}
  Let $\dcat{D}$ be an exact equipment. There is an equivalence of fibrations on the left such
  that the triangle on the right also commutes:
  \begin{equation*}
    \begin{tikzcd}[column sep=0em]
      \Mon(\dcat{D}) \ar[rr,"\equiv"] \ar[dr,two heads,"\MOb"' pos=.3]
        && \dcat{D}_0^{\bo} \ar[dl,two heads,"\dom" pos=.3] \\
      & \dcat{D}_0 &
    \end{tikzcd}
    \qquad\quad
    \begin{tikzcd}[column sep=0em]
      \Mon(\dcat{D}) \ar[rr,"\equiv"] \ar[dr,"\ColDash"' pos=.3]
        && \dcat{D}_0^{\bo} \ar[dl,"\cod" pos=.3] \\
      & \dcat{D}_0 &
    \end{tikzcd}
  \end{equation*}
\end{proposition}
\begin{proof}
  The functor $\dom\colon\dcat{D}_0^{\bo}\to\dcat{D}_0$ is a fibration via the factorization system
  in Theorem~\ref{thm:orthogonal}. The equivalence sends a monoid $(c,M)$ to the collapse morphism
  $i_M\colon c\twoheadrightarrow\Col{M}$, which is in $\bo$ by the exactness of $\dcat{D}$. Since
  $\vec{\imath}_M$ is the universal embedding \eqref{eqn:universal_embedding} of $M$, any monoid homomorphism
  $(f,\vec{f}\mspace{2mu})$ gives rise to a unique $\tilde{f}$ such that
  \begin{equation*}
    \begin{tikzcd}
      c \ar[r,tick,"M" domA] \ar[d,"f"']
        & c \ar[d,"f"] \\
      d \ar[r,tick,"N"{codA,domB}] \ar[d,two heads,"i_N"']
        & d \ar[d,two heads,"i_N"] \\
      \Col{N} \ar[r,tick,"\Col{N}"' codB]
        & \Col{N}
      \twocellA{\vec{f}}
      \twocellB{\vec{\imath}_N}
    \end{tikzcd}
    \quad = \quad
    \begin{tikzcd}
      c \ar[r,tick,"M" domA] \ar[d,two heads,"i_M"']
        & c \ar[d,two heads,"i_M"] \\
      \Col{M} \ar[r,tick,"\Col{M}"{codA,domB}] \ar[d,"\tilde{f}"']
        & \Col{M} \ar[d,"\tilde{f}"] \\
      \Col{N} \ar[r,tick,"\Col{N}"' codB]
        & \Col{N}.
      \twocellA{\vec{\imath}_M}
      \twocellB{\vec{\id}_{\tilde{f}}}
    \end{tikzcd}
  \end{equation*}
  Moreover, the pair $(f,\tilde{f})$ defines a morphism of arrows $i_M\to i_N$ in
  $\dcat{D}_0^{\bo}$. By \cite[Lemma 4.14]{Schultz2015}, if $\vec{f}$ is cartesian then so is
  $\vec{\id}_{\tilde{f}}$, and clearly the converse also holds. It follows that the left triangle is
  a morphism of fibrations since $\vec{f}$ being cartesian over $f$ implies $(f,\tilde{f})$ is as
  well.

  The inverse equivalence $\dcat{D}^\bo_0\to\Mon(\dcat{D})$ sends a $\bo$ map $f\colon c\to d$ to
  the restriction $d(f,f)$ with its induced monoid structure.
\end{proof}

\begin{theorem}
    \label{thm:Mod_vs_bo}
  Let $\dcat{D}$ be an exact equipment with local reflexive coequalizers. There is an equivalence of
  equipments
  \[ \begin{tikzcd}[row sep=2.5ex, column sep=0em]
    \dMod(\dcat{D}) \ar[rr,"\equiv"]\ar[rd,"\Col{-}"' pos=.3] && \dcat{D}^{\bo}\ar[ld,"\cod" pos=.3]\\
    &\dcat{D}
  \end{tikzcd} \]
\end{theorem}
\begin{proof}
  By Proposition~\ref{prop:collapse_local_equivalence} the equipment functor
  $\ColDash\colon\dMod(\dcat{D})\to\dcat{D}$ is a local equivalence, and
  $\cod\colon\dcat{D}^{\bo}\to\dcat{D}$ is a local equivalence by definition of $\dcat{D}^{\bo}$. It
  follows that the equivalence of fibrations from Proposition~\ref{prop:Mon_vs_bo} extends to an
  equivalence of equipments.
\end{proof}

\chapter{Equipments of monoidal profunctors}
  \label{chap:equipments_monoidal_profunctors}

With the needed material on general equipments out of the way, we now move on to topics specific to
monoidal, traced monoidal, and compact closed categories, culminating in the proofs of
\ref{thm:TheoremA} and \ref{thm:TheoremB}. In the introduction, we stated versions of these theorems
for traced categories, but versions about compact categories seem to be their most natural and
general statements. We will prove both.

After a brief review of monoidal, traced, and compact categories in
Section~\ref{sec:monoidal,compact,traced}, we define equipments $\dMonProf$, $\dTrProf$, and
$\dCompProf$ in Section~\ref{sec:monoidal_profunctors}, whose objects are respectively monoidal,
traced, and compact categories. \ref{thm:TheoremB} will be proven first through chains of
equivalences of fibrations over $\CompCat$ and $\TrCat$ respectively as follows:
\begin{equation*}
  \begin{tikzcd}[column sep=.8em]
    \int\limits^{\mathclap{\cat{C}\in\CompCat}} \Lax(\cat{C},\Set)
        \ar[r,"\equiv"]
      & \CPsh(\dCompProf)
        \ar[r,"\equiv"]
      & \Ptd(\dCompProf)
        \ar[r,"\equiv"]
      & \Mon(\dCompProf)
        \ar[r,"\equiv"]
      & \CompCat^{\bo} \\
    \int\limits^{\mathclap{\cat{T}\in\TrCat}} \Lax\big(\Int(\cat{T}),\Set\big)
        \ar[r,"\equiv"]
      & \CPsh(\dTrProf)
        \ar[rr,"\equiv"]
      && \Mon(\dTrProf)
        \ar[r,"\equiv"]
      & \TrCat^{\bo}
  \end{tikzcd}
\end{equation*}
The leftmost equivalences follow trivially from the definition of $\CPsh$,
Definition~\ref{def:copresheaves}. The rightmost equivalences will follow directly from
Proposition~\ref{prop:Mon_vs_bo} once we prove that $\dTrProf$ and $\dCompProf$ are exact in
Section~\ref{sec:exactness_proofs}. The middle equivalences are specific to traced and compact
categories; these are established in Section~\ref{sec:special_CompProf}.

To prove \ref{thm:TheoremA}, we need to deal with the issue of objectwise-freeness. This is
the purpose of Section~\ref{sec:monoids_on_free}. Finally in
Section~\ref{sec:characterization_of_traced} we give our ``traceless characterization of $\TTrCat$''. 

\section{Monoidal, Compact, and Traced Categories}
  \label{sec:monoidal,compact,traced}

We begin by reminding the reader of some categorical preliminaries: basic definitions and facts
about monoidal, traced, and compact categories, lax and strong functors, and the Int construction.
Standard references include \cite{KellyLaplaza}, \cite{JoyalStreet}, and \cite{JoyalStreetVerity}.

A \emph{strict monoidal category} $\cat{M}$ is a category equipped with a functor
$\otimes\colon\cat{M}\times\cat{M}\to\cat{M}$ and an object $I\in\cat{M}$, satisfying the usual
monoid axioms.%
\footnote{
  We also used the notation $\otimes$ to denote bimodule composition in
  Definition~\ref{def:monoids_and_modules}; hopefully the intended meaning of the symbol will be
  clear from context.
}
In other words, a strict monoidal category is a monoid object in the category $\Cat$. Such a
category $\cat{M}$ is \emph{symmetric} if there are in addition natural isomorphisms
\[
  \sigma_{X,Y}\colon X\otimes Y\to Y\otimes X
\]
satisfying equations $\sigma_{X,Y\otimes Z}=(\id_X\otimes\sigma_{X,Z})\circ(\sigma_{X,Y}\otimes
\id_Z)$ and $\sigma_{Y,X}\circ\sigma_{X,Y}=\id_{X\otimes Y}$.

\begin{warning}
    \label{warn:symmetric}
  Aside from the appendix, whenever we discuss monoidal categories in this article, we will mean
  symmetric strict monoidal categories.
\end{warning}

Let $\cat{M}$ and $\cat{N}$ be monoidal categories. A functor $F\colon\cat{M}\to\cat{N}$ is called
\emph{lax monoidal} if it is equipped with coherence morphisms
\begin{equation*}
  \begin{tikzcd}
    I_{\cat{N}} \rar{\mu} & F(I_{\cat{M}})
  \end{tikzcd}
  \quad\text{and}\quad
  \begin{tikzcd}
    F(X) \otimes_{\cat{N}} F(Y) \rar{\mu_{X,Y}} & F(X\otimes_{\cat{M}} Y)
  \end{tikzcd}
\end{equation*}
satisfying certain compatibility equations (see, e.g.\ \cite{Leinster,BorceuxV2}). If all coherence
morphisms are identities (resp.\ isomorphisms), then $F$ is \emph{strict} (resp.\ \emph{strong}). Let
$\Lax(\cat{M},\cat{N})$ denote the category of lax monoidal functors and monoidal transformations
from $\cat{M}$ to $\cat{N}$.

Write $\MMonCat$ for the 2-category of strict symmetric monoidal categories, strict symmetric
monoidal functors, and monoidal transformations. Let $\MonCat$ denote the underlying 1-category.

A \emph{compact category} is a (symmetric) monoidal category $\cat{C}$ with the property that for
every object $X\in\cat{C}$ there exists an object $X^*$ and morphisms $\eta_X\colon I\to X^*\otimes
X$ and $\epsilon_X\colon X\otimes X^*\to I$ such that the following diagrams commute:
\begin{equation*}
  \begin{tikzcd}[column sep=small]
    X\arrow[r,"\id_X"]\arrow[d,"\cong"'] & X \\
    X\otimes I\arrow[d,"X\otimes\eta_X"'] & I\otimes X\arrow[u,"\cong"'] \\
    X\otimes(X^*\otimes X)\arrow[r,"\cong"'] & (X\otimes X^*)\otimes X\arrow[u,"\epsilon_X\otimes X"']
  \end{tikzcd}
  \hspace{.6in}
  \begin{tikzcd}[column sep=small]
    X^*\arrow[r,"\id_{X^*}"]\arrow[d,"\cong"'] & X^*\\
    I\otimes X^*\arrow[d,"\eta_X\otimes X^*"'] & X^*\otimes I\arrow[u,"\cong"'] \\
    (X^*\otimes X)\otimes X^*\arrow[r,"\cong"'] & X^*\otimes (X\otimes X^*)\arrow[u,"X^*\otimes\epsilon_X"']
  \end{tikzcd}
\end{equation*}
We will denote by $\CCompCat$ the full sub-2-category of $\MMonCat$ spanned by the compact
categories and write $\UCM\colon\CCompCat\to\MMonCat$ for the corresponding forgetful functor. Let
$\CompCat$ denote the underlying 1-category.

Given a morphism $f\colon X\to Y$ in $\cat{C}$, we denote by $f^*\colon Y^*\to X^*$ the composite
\begin{equation*}
  Y^*\To{\eta_X}X^*\otimes X\otimes Y^* \To{f}X^*\otimes Y\otimes Y^*\To{\epsilon_Y}X^*.
\end{equation*}
It is easy to check that a strong functor $F\colon\cat{C}\to\cat{M}$ to a monoidal category
preserves all duals that exist in $\cat{C}$, i.e.\ there is a natural isomorphism $F(c^*)\iso
F(c)^*$. From this, it follows that if $F,G\colon\cat{C}\to\cat{C'}$ are functors between compact
categories, then any natural transformation $\alpha\colon F\to G$ is a natural isomorphism. Indeed,
for any object $c\in\cat{C}$, the inverse of the $c$-component $\alpha_c\colon Fc\to Gc$ is given by
the dual morphism $(\alpha_{c^*})^*\colon Gc\to Fc$ to the dual component. Thus all 2-cells in
$\CCompCat$ are invertible.

A \emph{trace structure} on a (symmetric) monoidal category $\cat{T}$ is a collection of functions
\begin{equation}
    \label{dia:trace_function}
  \Tr^U_{X,Y}\colon\Hom_{\cat{T}}(U\otimes X,U\otimes Y)\to\Hom_{\cat{T}}(X,Y)
\end{equation}
for $U,X,Y\in\Ob(\cat{T})$ satisfying seven equational axioms, we refer the reader to
\cite{JoyalStreetVerity} for more details. If $\cat{T}$ and $\cat{U}$ are traced categories, then a
\emph{(strict) traced functor} is simply a strict symmetric monoidal functor which commutes with the
trace operation.

In \cite{JoyalStreetVerity}, it is shown that every traced category $\cat{T}$ embeds as a full
subcategory of a compact category $\Int(\cat{T})$ whose objects are pairs
$(\inp{X},\outp{X})\in\Ob(\cat{T})\times\Ob(\cat{T})$ with morphisms given by
\begin{equation*}
  \Hom_{\Int(\cat{T})}\big((\inp{X},\outp{X}),(\inp{Y},\outp{Y})\big)
    = \Hom_{\cat{T}}(\inp{X}\otimes\outp{Y},\outp{X}\otimes\inp{Y})
\end{equation*}
and with compositions computed using the trace of $\cat{T}$.

\begin{remark}
    \label{rem:traced_2morphisms}
  Traced categories were first defined in~\cite{JoyalStreetVerity}, which defines the 2-morphisms
  between traced functors to simply be monoidal transformations. However, this choice does not
  behave appropriately with the $\Int$ construction (for example $\Int$ would not be 2-functorial).
  The error was corrected in \cite{HK}, where it was shown that the appropriate 2-morphisms between
  traced functors are natural \emph{isomorphisms}.
\end{remark}

We denote by $\TTrCat$ the corrected 2-category of traced categories (where 2-cells are invertible),
and we denote its underlying 1-category by $\TrCat$. Write $\UTM\colon\TTrCat\to\MMonCat$ for the
forgetful functor.

Every compact category $\cat{C}$ has a canonical trace structure, defined on a morphism $f\colon
U\otimes X\to U\otimes Y$ morally (up to symmetries and identities) to be $\epsilon_U\circ f\circ
\eta_U$. More precisely, one defines $\Tr^U_{X,Y}[f]$ to be the composite
\begin{equation*}
  \begin{tikzcd}[column sep=29pt]
    X\ar[r,"\eta_U\otimes X"]&
    U^*\otimes U\otimes X\ar[r,"U^*\otimes f"]&
    U^*\otimes U\otimes Y\ar[r,"\sigma_{U^*,U}\otimes Y"]&[6pt]
    U\otimes U^*\otimes Y\ar[r,"\epsilon_U\otimes Y"]&
    Y
  \end{tikzcd}
\end{equation*}
Thus we have a functor $\UCT\colon\CCompCat\to\TTrCat$. It is shown
in~\cite{JoyalStreetVerity}~and~\cite{HK} that this functor is the right half of a 2-adjunction
\begin{equation}
    \label{dia:traced_compact_adjunction}
  \begin{tikzcd}
    \TTrCat\arrow[r,shift left=.5ex, "\Int"]&\CCompCat\arrow[l,shift left=.5ex,"\UCT"].
  \end{tikzcd}
\end{equation}
Note that $\UCM=\UCT\UTM$. In Section~\ref{sec:characterization_of_traced} we will be able to
formally define the 2-category $\TTrCat$ without mentioning the trace structure
\eqref{dia:trace_function} or the usual seven axioms, but instead in terms of the relationship between
compact and monoidal categories.

\begin{remark}
    \label{rmk:fully_faithful_and_trace}
  We record the following facts, which hold for any traced category $\cat{T}$; each is shown in, or
  trivially derived from,~\cite{JoyalStreetVerity}:
  \begin{enumerate}[label={\upshape\roman*}.]
    \item The component $\cat{T}\to\Int(\cat{T})$ of the unit of the adjunction
      \eqref{dia:traced_compact_adjunction} is fully faithful. It follows that
      $\Int\colon\TTrCat\to\CCompCat$ is locally fully faithful.
    \item If $\cat{M}$ is a monoidal category and $F\colon\cat{M}\to\cat{T}$ is a fully faithful
      symmetric monoidal functor, then $\cat{M}$ has a unique trace for which $F$ is a traced
      functor.
    \item If $\cat{T}$ is compact then the counit
      $\Int(\cat{T})\To{\raisebox{-1ex}{$\equiv$}}\cat{T}$ is an equivalence.
    \item Suppose that $\cat{T'}$ is a traced category and that $F\colon \cat{T}\to \cat{T}'$ is a
      traced functor. Then $F$ is bijective-on-objects (resp.\ fully faithful) if and only if
      $\Int(F)$ is.
  \end{enumerate}
\end{remark}

\section{Monoidal profunctors}
  \label{sec:monoidal_profunctors}

Suppose $\cat{C}$ and $\cat{D}$ are monoidal categories. We define a \emph{monoidal profunctor} $M$
from $\cat{C}$ to $\cat{D}$ to be an ordinary profunctor (see Example~\ref{ex:profunctors}) $M\colon
\op{\cat{C}}\times\cat{D}\to\Set$ which is equipped with a lax-monoidal structure, where $\Set$ is
endowed with the cartesian monoidal structure. In the bimodule notation, this means that there is an
associative operation assigning to any elements $m_1\in M(c_1,d_1)$ and $m_2\in M(c_2,d_2)$ an
element $m_1\boxtimes m_2\in M(c_1\otimes c_2,d_1\otimes d_2)$ such that
\[
   (f_1\cdot m_1\cdot g_1)\boxtimes(f_2\cdot m_2\cdot g_2)
      = (f_1\otimes f_2)\cdot(m_1\boxtimes m_2)\cdot(g_1\otimes g_2),
\]
as well as a distinguished element $I_M\in M(I,I)$ such that $I_M\boxtimes m = m = m\boxtimes I_M$
for any $m\in M(c,d)$. If moreover $m_2\boxtimes m_1 = \sigma_{d_1,d_2}\cdot(m_1\boxtimes
m_2)\cdot\sigma_{c_1,c_2}^{-1}$, then one says $M$ is \emph{symmetric monoidal}.%
\footnote{
  We will generally suppress the word \emph{symmetric} since all monoidal categories and monoidal
  profunctors are symmetric by assumption; see Warning~\ref{warn:symmetric}.
}

A \emph{monoidal profunctor morphism} $\phi\colon M\to N$ is simply a monoidal transformation.
Spelling this out in bimodule notation, $\phi$ is an ordinary morphism of profunctors such that
$\phi(m_1\boxtimes m_2)=\phi(m_1)\boxtimes\phi(m_2)$ and $\phi(I_M)=I_N$.

We define a double category $\dMonProf$ whose objects are (symmetric) monoidal categories, vertical
arrows are strict (symmetric) monoidal functors, horizontal arrows are (symmetric) monoidal
profunctors, and 2-cells are defined as in \eqref{eqn:Prof2cells}, requiring $\phi$ to be a monoidal
transformation. It remains to check that the horizontal composition of monoidal profunctors is
monoidal. This follows from the fact that reflexive coequalizers---namely the ones from
\eqref{eqn:coendComp}---commute with products in $\Set$. Note that $\dMonProf$ is in fact an
equipment since the cartesian 2-cell \eqref{eq:cartesian profunctor morphism} is a monoidal
transformation if $N$, $F$, and $F'$ are monoidal functors. We leave it as an exercise for the
reader to check that there is an isomorphism of 2-categories $\VVer(\dMonProf)\iso\MMonCat$, i.e.\
that for any pair of strict symmetric monoidal functors $F,G\colon\cat{C}\to\cat{D}$, there is a
bijection between monoidal transformations
$\cat{C}(\textrm{--},\textrm{--})\to\cat{D}(F(\textrm{--}),G(\textrm{--}))$ and monoidal
transformations $F\to G$.

The fully faithful functors $\UCM\colon\CompCat\to\MonCat$ and $\Int\colon\TrCat\to\CompCat$,
defined above, induce equipments $\dCompProf\coloneqq\UCM^*(\dMonProf)$ and
$\dTrProf\coloneqq\Int^*(\dCompProf)$ as in
Definition~\ref{def:induced_locally_equivalent_equipment}. In particular, the vertical 1-categories
of these equipments are given by
\[
  \dMonProf_0=\MonCat,\quad\dCompProf_0=\CompCat,\quad\dTrProf_0=\TrCat.
\]

It may seem strange at first to define a proarrow $\cat{T}\tickar \cat{T}'$ between traced
categories to be a monoidal profunctor $\Int(\cat{T})\tickar\Int(\cat{T}')$. The next proposition serves as a
first sanity check on this definition, and the remainder of this paper provides further support.

\begin{proposition}
    \label{prop:2iso_traced}
  There is an isomorphism of 2-categories, $\VVer(\dTrProf)\iso\TTrCat$.
\end{proposition}
\begin{proof}
  Clearly these 2-categories have the same underlying 1-category, so it suffices to show that there
  is a bijection $\VVer(\dTrProf)(F,G)\iso\TTrCat(F,G)$ for any traced functors $F,G\colon
  \cat{T}\to \cat{T}'$ which preserve units and composition. By the definition of $\dTrProf$, we
  have $\VVer(\dTrProf)(F,G)=\CCompCat(\Int(F),\Int(G))$. The result then follows since $\Int$ is
  locally fully faithful \cite{JoyalStreetVerity}.
\end{proof}

Thus from the definitions and Proposition~\ref{prop:2iso_traced}, we see that the vertical
2-categories of these equipments are as expected:
\[
  \VVer(\dMonProf)\iso\MMonCat,\quad\VVer(\dCompProf)\iso\CCompCat,\quad\VVer(\dTrProf)\iso\TTrCat.
\]

\begin{proposition}
  Each of the equipments $\dMonProf$, $\dTrProf$, and $\dCompProf$ has local reflexive coequalizers.
\end{proposition}
\begin{proof}
  It suffices to prove this for $\dMonProf$, since the other two are locally equivalent to it.
  For any monoidal category $\cat{M}$, the category of lax monoidal functors $\cat{M}\to\Set$ is
  closed under reflexive coequalizers. This follows easily from the fact that reflexive coequalizers
  commute with finite products. (In fact, the same argument shows that the category $\cat{O}\alg$ of
  algebras for any colored operad $\cat{O}$ is closed under reflexive coequalizers.) Thus in
  particular the category of monoidal profunctors $\cat{M}\tickar\cat{N}$, i.e.\ the category of lax
  functors $\op{\cat{M}}\times\cat{N}\to\Set$, is closed under reflexive coequalizers.

  The fact that tensor product of monoidal profunctors preserves reflexive coequalizers follows from
  the fact that the tensor product itself is constructed as a reflexive coequalizer.
\end{proof}

\begin{remark}
  The equipments $\dMonProf$, $\dTrProf$, and $\dCompProf$ are in fact locally cocomplete. The
  category of profunctors $\cat{C}\tickar\cat{D}$ in any of these equipments is equivalent to the
  category of algebras for a monad on $\Set^{\op{\cat{C}}\times\cat{D}}$, and it is a general fact
  that if the category of algebras for a monad on a cocomplete category has reflexive coequalizers,
  then it has all colimits.
\end{remark}

\section{$\dMonProf$, $\dCompProf$, and $\dTrProf$ are exact}
  \label{sec:exactness_proofs}
Our next goal is to prove that the equipments $\dMonProf$, $\dTrProf$, and $\dCompProf$ from Section~\ref{sec:monoidal_profunctors} are
all exact, as in Definition~\ref{def:exact_equipment}. We then discuss a few consequences of exactness: the $(\bo,\ff)$ factorization systems (Proposition~\ref{prop:(bo,ff)_really_is}) and the relationship between monoids and $\bo$-maps in these equipments (Corollary~\ref{cor:Mod_vs_bo}).

\begin{proposition}
    \label{prop:MonProf_exact}
  The equipment $\dMonProf$ is exact.
\end{proposition}
\begin{proof}
  Suppose that $M\colon \cat{C}\tickar \cat{C}$ is a monoid in $\dMonProf$. One uses $M$ to
  construct a category $\Col{M}$ with the same objects as $\cat{C}$, and with hom sets defined by
  $\Col{M}(c,d)\coloneqq M(c,d)$ for any pair of objects $c,d\in\Ob(\cat{C})$. For any object $c$,
  the identity is provided by $i(\id_{\cat{C}})$, while the multiplication $\bullet$ on $M$ defines composition in $\Col{M}$.

  The unit of $M$ can also be used to construct an identity-on-objects functor
  $i_M\colon\cat{C}\to\Col{M}$ and an embedding 2-cell $\vec{\imath}_M$ sending any element of $M$
  to itself as a morphism of $\Col{M}$. It is easy to see that $\vec{\imath}_M$ is cartesian and
  that $(i_M,\vec{\imath}_M)$ is a collapse. The category $\Col{M}$ has a canonical monoidal
  structure, which on objects is just that of $\cat{C}$ and on morphisms is induced by the monoidal
  profunctor structure of $M$. It is also simple to verify the second part of
  Definition~\ref{def:exact_equipment}: an $(M,N)$-bimodule is precisely the data of a profunctor
  $\Col{M}\tickar\Col{N}$.
\end{proof}

\begin{proposition}
    \label{prop:CompProf_exact}
  The equipment $\dCompProf$ is exact.
\end{proposition}
\begin{proof}
  We can consider a monoid $M\colon \cat{C}\tickar \cat{C}$ in $\dCompProf$ as a monoid in
  $\dMonProf$; it has a collapse embedding $i_M\colon M\to \Col{M}$ by
  Proposition~\ref{prop:MonProf_exact}. The collapse $\Col{M}$ is a monoidal category and, by
  Theorem~\ref{thm:Mod_vs_bo}, $i_M$ is a (strict symmetric monoidal) $\bo$ functor. But any strong
  monoidal functor preserves duals, so every object of $\Col{M}$ has a dual and hence $\Col{M}$ is
  compact. The map $\UCM\colon\dCompProf\to\dMonProf$ is a fully faithful local equivalence and so
  $\Col{M}$ being a collapse in $\dMonProf$ implies it is a collapse in $\dCompProf$.
\end{proof}

We record the following consequence of Theorem~\ref{thm:orthogonal} in the current notation.

\begin{corollary}
  Each of the 2-categories $\MMonCat$ and $\CCompCat$ admits a 2-orthogonal $(\bo,\ff)$ factorization system.
\end{corollary}

We abuse notation slightly and use the same name ($\bo,\ff$) for the factorization systems on
different categories.

The exactness of $\dTrProf$ is significantly more difficult to establish. We use the exactness of
$\dCompProf$ and the close relationship between $\dTrProf$ and $\dCompProf$ provided by the $\Int$
construction. We first prove a lemma which is needed for the proof. Recall the adjunction
\eqref{dia:traced_compact_adjunction} and write $\eta_{\cat{T}}\colon\cat{T}\to\UCT\Int(\cat{T})$
for the unit component on $\cat{T}\in\TTrCat$.

\begin{lemma}
    \label{lem:Tr_bo_Int}
  Let $\cat{T}$ be a traced category, $\cat{C}$ a compact category, and
  $F\colon\Int(\cat{T})\twoheadrightarrow \cat{C}$ a bijective-on-objects monoidal functor. Consider
  the factorization in $\MonCat$ of $\UCM F\circ\UTM\eta_{\cat{T}}$ into a bijective-on-objects $G$
  followed by a fully faithful $H$, as follows:
  \[ \begin{tikzcd}[column sep=large]
    \UTM \cat{T} \ar[r,hook,"\UTM\eta_{\cat{T}}"] \ar[d,two heads,"\exists G"']
      & \UCM\Int(\cat{T}) \ar[d,two heads,"\UCM F"] \\
    \cat{M} \ar[r,hook,"\exists H"'] & \UCM \cat{C}.
  \end{tikzcd} \]
  There is a unique trace structure on $\cat{M}$, i.e.\ a unique traced category $\cat{T}'$ with
  $\UTM \cat{T}'=\cat{M}$, such that the factorization lifts to $\TrCat$:
  \[ \begin{tikzcd}[column sep=large]
    \cat{T} \ar[r,hook,"\eta_{\cat{T}}"] \ar[d,two heads,"G"']
      & \UCT\Int(\cat{T}) \ar[d,two heads,"\UCT F"] \\
    \cat{T}' \ar[r,hook,"H"'] & \UCT \cat{C}.
  \end{tikzcd} \]
  Moreover, there is an isomorphism $\alpha\colon\Int(\cat{T}')\iso \cat{C}$ such that
  $\UCT\alpha\circ \eta_{\cat{T}'}=H$ and $\alpha\circ\Int(G)=F$.
\end{lemma}
\begin{proof}
  This derives mainly from basic properties of the $\Int$ construction; see
  Remark~\ref{rmk:fully_faithful_and_trace}. Since $H\colon\cat{M}\to\UCM\cat{C}$ is fully faithful,
  the trace on $\UCT\cat{C}$ uniquely determines the desired trace structure on $\cat{T}'$ by which
  $H$ is a traced functor. It also follows that $G$ respects the trace in $\cat{T}$ since $\UCT
  F\circ\eta_{\cat{T}}$ does.

  For the final claim, consider the diagram
  \[ \begin{tikzcd}[column sep=4em, row sep=5ex]
    \cat{T} \ar[r,hook,"\eta_{\cat{T}}"] \ar[d,two heads,"G"']
      & \UCT\Int(\cat{T}) \ar[d,two heads,"\UCT\Int(G)"']
        \ar[ddr,two heads,bend left=25,"\UCT F"] &[between origins] \\
    \cat{T}' \ar[r,hook,"\eta_{\cat{T}'}"] \ar[drr,hook,bend right=15,"H"']
      & \UCT\Int(\cat{T}') \ar[dr,"\UCT\alpha"' {inner sep=0pt,pos=.25}] & \\[-10]
      &&\UCT \cat{C}
  \end{tikzcd} \]
  where $\alpha\colon\Int(\cat{T}')\to \cat{C}$ is the adjunct of $H$, which is fully faithful since
  $H$ is. Since $G$ is $\bo$, $\UCT\Int(G)$ will be $\bo$ as well. But $\UCT F$ is $\bo$, so $\UCT
  \alpha$ and hence $\alpha$ must be $\bo$ also. Since $\alpha$ is both $\ff$ and $\bo$, it is an
  isomorphism, completing the proof.
\end{proof}

\begin{proposition}
    \label{prop:TrProf_exact}
  The equipment $\dTrProf$ is exact.
\end{proposition}
\begin{proof}
  Let $M\colon \cat{T}\tickar \cat{T}$ be a monoid in $\dTrProf$. By definition of $\dTrProf$ this
  is a monoid $M\colon\Int(\cat{T})\tickar\Int(\cat{T})$ in $\dCompProf$, so $M=\Mon(\Int)(M)$ (in
  the language of Lemma~\ref{lemma:embed_for_LE}). Define $\Col{M}_{\MTCfont{C}}$ and
  $(i_M,\vec{\imath}_M)\colon (\Int(\cat{T}),M)\to\Col{M}_{\MTCfont{C}}$ to be the collapse
  embedding of $M$ in $\dCompProf$. Then applying Lemma~\ref{lem:Tr_bo_Int} with $F=i_M$ gives a
  traced category $\Col{M}_{\MTCfont{T}}$ and a $\bo$ functor
  $i'_M\colon\cat{T}\to\Col{M}_{\MTCfont{T}}$ as in the diagram below:
   \begin{equation} \begin{tikzcd}[column sep=large]
         \label{eqn:collapses_in_traced}
      \cat{T} \ar[r,hook,"\eta_{\cat{T}}"] \ar[d,two heads,"i'_M"']
         & \UCT\Int(\cat{T}) \ar[d,two heads,"\UCT i_M"] \\
      \Col{M}_{\MTCfont{T}} \ar[r,hook] & \UCT\Col{M}_{\MTCfont{C}}.
   \end{tikzcd} \end{equation}
  To see that $\Col{M}_{\MTCfont{T}}$ is a collapse in $\dTrProf$ we must establish the bijection
  $\Emb_{\mathrm{Tr}}(M,\cat{T}')\iso\TrCat(\Col{M}_{\MTCfont{T}},\cat{T}')$, natural in the traced
  category $\cat{T}'$.

  Using the adjunction bijection and precomposition with the inverse of
  $\alpha\colon\Int(\Col{M}_{\MTCfont{T}})\iso\Col{M}_{\MTCfont{C}}$ from Lemma~\ref{lem:Tr_bo_Int},
  we get an isomorphism
  \[
    \TrCat(\Col{M}_{\MTCfont{T}},\UCT\Int(\cat{T}'))
    \iso \CompCat(\Int(\Col{M}_{\MTCfont{T}}),\Int(\cat{T}'))
    \iso \CompCat(\Col{M}_{\MTCfont{C}},\Int(\cat{T}')).
  \]
  This isomorphism is the top right morphism in the diagram
  \[ \begin{tikzcd}[column sep=between origins,row sep=between origins]
    \TrCat(\Col{M}_{\MTCfont{T}},\cat{T}')
        \ar[r,"\eta_{\cat{T}'}\circ\textrm{--}"]
        \ar[d,"\textrm{--}\circ i'_M"']
        \ar[dr,phantom,"\lrcorner" very near start]
      &[2.4em,between borders] \TrCat(\Col{M}_{\MTCfont{T}},\UCT\Int(\cat{T}'))
        \ar[d,"\textrm{--}\circ i'_M"']
        \ar[r,"\iso"]
        \ar[dr,phantom,"\lrcorner" very near start]
      &[12.8em] \CompCat(\Col{M}_{\MTCfont{C}},\Int(\cat{T}')) \ar[d,"-\circ i_M"'] \\[5.4em]
    \TrCat(\cat{T},\cat{T}')
        \ar[r,"\eta_{\cat{T}'}\circ\textrm{--}"]
        \ar[rr,bend right=12,"\Int"']
      & \TrCat(\cat{T},\UCT\Int(\cat{T}'))
        \ar[r,"\iso"]
      & \CompCat(\Int(\cat{T}),\Int(\cat{T}'))
  \end{tikzcd} \]
  The right square commutes by the naturality of the $(\Int,\UCT)$ adjunction, together with the
  equality $i_M=\alpha\circ\Int(i'_M)$ from Lemma~\ref{lem:Tr_bo_Int}. The left square is a
  pullback, by the orthogonality of $i'_M\in\bo$ and $\eta_{\cat{T}'}\in\ff$, and the right square
  is a pullback because the top and bottom maps are isomorphisms. Hence the outer square is a
  pullback as well. Since $\Col{M}_{\MTCfont{C}}$ is a collapse in $\dCompProf$, there is a
  bijection $\Emb_{\mathrm{Cp}}(M,\Int(\cat{T}'))\iso\CompCat(\Col{M}_{\MTCfont{C}},\Int(\cat{T}'))$
  so by Lemma~~\ref{lemma:embed_for_LE}, the outer pullback produces the desired natural isomorphism
  $\Emb_{\mathrm{Tr}}(M,\cat{T}')\iso\TrCat(\Col{M}_{\MTCfont{T}},\cat{T}')$.

  Since the trivial monoid on $\Col{M}_{\MTCfont{T}}$ in $\dTrProf$ is by definition the trivial
  monoid on $\Int(\Col{M}_{\MTCfont{T}})\iso\Col{M}_{\MTCfont{C}}$ in $\dCompProf$, the collapse
  embedding $M\Rightarrow\Col{M}_{\MTCfont{T}}$ is (after composition with the isomorphism $\alpha$)
  just the collapse $(i_M,\vec{\imath}_M)$ in $\dCompProf$, and hence is cartesian. Since the
  inclusion $\dTrProf\to\dCompProf$ is a local equivalence and $\dCompProf$ is exact, the second
  condition of Definition~\ref{def:exact_equipment} follows immediately.
\end{proof}

\begin{proposition}
    \label{prop:(bo,ff)_really_is}
  In $\dMonProf$, $\dTrProf$, and $\dCompProf$, a vertical map is $\ff$ (resp.\ $\bo$) if and only if
  it is fully faithful (resp.\ bijective-on-objects) in the usual sense.
\end{proposition}
\begin{proof}
  It is clear that the forgetful double functor $U\colon\dMonProf\to\dProf$ creates cartesian 2-cells: a
  2-cell in $\dMonProf$ is cartesian if and only if its underlying 2-cell in $\dProf$ is cartesian.
  In particular, this implies that a vertical map in $\dMonProf$ is $\ff$ if and only if its
  underlying map in $\dProf$ is $\ff$, hence is fully faithful in the usual sense.

  By the construction of collapses in $\dMonProf$, it is easy to see that $U$ similarly creates
  collapse 2-cells. Thus a vertical map in $\dMonProf$ is $\bo$ if and only if its underlying map in
  $\dProf$ is $\bo$, hence is bijective-on-objects in the usual sense.

  Because the forgetful double functor $\UCM\colon\dCompProf\to\dMonProf$ is a fully faithful local
  equivalence, it follows that it too creates cartesian 2-cells, and from the construction of
  collapses in $\dCompProf$ it also creates collapse 2-cells. Hence a vertical map in $\dCompProf$
  is in $\bo/\ff$ if and only if its underlying map in $\dMonProf$ is.

  Likewise, $\Int\colon\dTrProf\to\dCompProf$ creates cartesian and collapse 2-cells. It only remains 
  to show that a traced functor $F\colon\cat{T}\to\cat{T}'$ is fully faithful (resp.\
  bijective-on-objects) in the usual sense if and only if $\Int(F)$ is. For fully faithfulness, this
  follows easily from the fact that the unit $\eta\colon\cat{T}\to\UCT\Int T$ is fully faithful. It
  is also clear that $\Int(F)$ is bijective-on-objects by construction when $F$ is. Finally, suppose
  $\Int(F)$ is bijective-on-objects. Because the unit $\eta$ is injective-on-objects, $F$ must be
  injective-on-objects. If $x\in\cat{T}'$ is any object, then there is an object $(t_1,t_2)\in\Int
  T$ such that $\Int(F)(t_1,t_2)=(x,I)$, but $\Int(F)(t_1,t_2)=(Ft_1,Ft_2)$, hence $Ft_1=x$, showing
  that $F$ is also surjective-on-objects.
\end{proof}

Finally, we record for reference the application of Theorem~\ref{thm:Mod_vs_bo} to the exact
equipments $\dTrProf$ and $\dCompProf$.

\begin{corollary}\label{cor:Mod_vs_bo}
There are equivalences of equipments
  \[
  \begin{tikzcd}[row sep=2.5ex, column sep=0em]
    \dMod(\dTrProf) \ar[rr,"\equiv"]\ar[rd,"\Col{-}"' pos=.3] && \dTrProf^{\bo}\ar[ld,"\cod" pos=.3]\\
    &\dTrProf
  \end{tikzcd}
  \quad\tn{and}\quad
  \begin{tikzcd}[row sep=2.5ex, column sep=0em]
    \dMod(\dCompProf) \ar[rr,"\equiv"]\ar[rd,"\Col{-}"' pos=.3] && \dCompProf^{\bo}\ar[ld,"\cod" pos=.3]\\
    &\dCompProf
  \end{tikzcd}
  \]
\end{corollary}

\section{Special properties of $\dCompProf$}
  \label{sec:special_CompProf}

In Section~\ref{sec:exactness_proofs}, we saw that the equipments $\dTrProf$ and $\dCompProf$ are
exact, connecting traced and compact categories to the machinery reviewed in
Section~\ref{chap:background_equipments}. All that has come before can be seen as general machinery
used to translate the statement of \ref{thm:TheoremB} into a more convenient form. In this section,
we collect the properties unique to traced and compact categories, and use these to prove
\ref{thm:TheoremB} via this translation.

Before we get to that, though, we must quickly say how we talk about (co)presheaves in the language
of equipments.

\subsection{Internal copresheaves and endo-proarrows in an equipment}

Copresheaves on a category $\cat{C}$ can be identified with profunctors $1\tickar\cat{C}$ in
$\dProf$. Motivated by this, we will think of proarrows $1\tickar c$ in any equipment $\dcat{D}$
with a terminal object 1 as ``internal copresheaves'' on the object $c$. For each object, there is a
category of copresheaves $\HHor(\dcat{D})(1,c)$. We can give a direct construction of the
bifibration over $\dcat{D}_0$ whose fiber over an object $c$ is the category of copresheaves on $c$:

\begin{definition}
    \label{def:copresheaves}
  Let $\dcat{D}$ be an equipment with a terminal object $1\in\dcat{D}_0$.%
  \footnote{
    In fact, such a definition makes sense for any object of $d\in\dcat{D}_0$, but we will only use the case $d=1$.
  }
  We define the category $\CPsh(\dcat{D})$, bifibered over $\dcat{D}_0$, by the strict pullback of
  categories
  \begin{equation*}
    \begin{tikzcd}
      \CPsh(\dcat{D}) \ar[r] \ar[d,two heads,"\MOb"']
          \ar[dr,phantom,"\lrcorner",very near start]
        & \dcat{D}_1 \ar[d,two heads,"{(\lframe,\rframe)}"] \\
      1\times\dcat{D}_0 \ar[r,"1\times\dcat{D}_0"']
        & \dcat{D}_0\times\dcat{D}_0.
    \end{tikzcd}
  \end{equation*}
\end{definition}

\begin{lemma}
    \label{lem:Psh_pullback}
  Let $F\colon\dcat{C}\to\dcat{D}$ be an equipment functor. Suppose that $\dcat{C}_0$ and
  $\dcat{D}_0$ have terminal objects which are preserved by $F_0$. Then there is an induced morphism
  of fibrations
  \begin{equation} \begin{tikzcd}
      \label{eq:CPsh_square}
    \CPsh(\dcat{C}) \ar[r,"\tilde{F}"] \ar[d,two heads,"\MOb"'] 
      & \CPsh(\dcat{D}) \ar[d,two heads,"\MOb"] \\
    \dcat{C}_0 \ar[r,"F_0"']
      & \dcat{D}_0.
  \end{tikzcd} \end{equation}
  Moreover, if $F$ is a local equivalence, then~\eqref{eq:CPsh_square} is a pseudo-pullback.
\end{lemma}
\begin{proof}
  Consider the cube
  \[ \begin{tikzcd}[row sep={35,between origins}, column sep={55,between origins}]
    \CPsh(\dcat{C}) \ar[rr,"\tilde{F}"] \ar[dr] \ar[dd,two heads]
      &[-10] & \CPsh(\dcat{D}) \ar[dd,two heads] \ar[dr] &[-10] \\[-5]
    & \dcat{C}_1 \ar[rr,crossing over,"F_1" near start]
      && \dcat{D}_1 \ar[dd,two heads] \\
    1\times\dcat{C}_0 \ar[rr,"1\times F_0" pos=.75]
        \ar[dr,"1\times\dcat{C}_0"' {pos=.25,inner sep=2pt}]
      && 1\times\dcat{D}_0 \ar[dr,"1\times\dcat{D}_0"' {pos=.25,inner sep=2pt}] & \\[-5]
    & \dcat{C}_0\times\dcat{C}_0 \ar[rr,"F_0\times F_0"']
        \ar[from=uu,crossing over,two heads]
      && \dcat{D}_0\times\dcat{D}_0.
  \end{tikzcd} \]
  Since $F_0$ preserves terminal objects, the bottom face of the cube commutes. The left and right
  faces of the cube are strict pullbacks by definition, hence there is a unique $\tilde{F}$ making
  the cube commute.

  If $F$ is a local equivalence, then the front face is a pseudo-pullback. The left and right faces
  are strict pullbacks along fibrations, hence pseudo-pullbacks (see
  Remark~\ref{rem:strict_vs_pseudo_pullback}). It follows that the back face is a pseudo-pullback as
  well.
\end{proof}

\erase{
  \begin{example}
    If $\dcat{D}=\dProf$ then an object of $\CPsh(\dcat{D})$ is a pair $(\cat{C},H)$, where
    $\cat{C}$ is a category and $H\colon \cat{C}\to\Set$ is a functor. A morphism
    $(\cat{C},H)\to(\cat{C}',H')$ is a pair $(F,F^\sharp)$ of a functor and a natural
    transformation, as in the lax triangle
    \[ \begin{tikzcd}[column sep=.6cm]
      \cat{C}\ar[rr,"F"]\ar[rd,"H"']&\ar[d,phantom,"\overset{F^\sharp}{\Rightarrow}" near start]&\cat{C}'\ar[dl,"H'"]\\
      &\Set.
    \end{tikzcd} \]
  \end{example}
}

\begin{definition}
    \label{def:ptd}
  Given an equipment $\dcat{D}$, we define the fibration of \emph{endo-proarrows} by the strict pullback
  \[
  \begin{tikzcd}
    \End(\dcat{D}) \ar[d,two heads, "\MOb"'] \ar[r] \ar[dr,phantom,"\lrcorner" very near start]
      & \dcat{D}_1 \ar[d,two heads,"{(\lframe,\rframe)}"] \\
    \dcat{D}_0 \ar[r,"\Delta"']
      & \dcat{D}_0\times\dcat{D}_0.
  \end{tikzcd}
  \]
  We also define a fibration $\Ptd(\dcat{D})\twoheadrightarrow\dcat{D}_0$ whose objects are
  \emph{pointed endo-proarrows}, i.e.\ endo-proarrows $M\colon c\tickar c$ in $\dcat{D}$ equipped
  with a unit $i_M\colon \unit(c)\Rightarrow M$ as in \eqref{eqn:unit_and_mult} (but not a multiplication), and
  whose morphisms are 2-cells which preserve the units.
\end{definition}

\subsection{Copresheaves coincide with monoids in $\dCompProf$ and $\dTrProf$}

The component of the proof of \ref{thm:TheoremB} which is specific to traced and compact categories
is the following equivalence of fibrations:
\[ \begin{tikzcd}[column sep=-.5em]
  \CPsh(\dCompProf) \ar[rr,"\equiv"] \ar[dr,two heads,"\MOb"']
    && \Mon(\dCompProf) \ar[dl,two heads,"\MOb"] \\
  & \CompCat. &
\end{tikzcd} \]
We prove this equivalence in this section (which implies the corresponding equivalence for traced
categories).
To do so we introduce a third fibration---that of pointed endo-proarrows---and
establish its equivalence with each of $\CPsh(\dCompProf)$ and $\Mon(\dCompProf)$ in
Proposition~\ref{Prop:ptd_prof_equivalence} and Proposition~\ref{prop:unit_implies_monoid} below.
The introduction of $\Ptd(\dCompProf)$ is merely a convenient way of organizing the proof:
recovering a copresheaf from a monoidal endo-profunctor requires only a unit (and not a
multiplication), while for any monoidal endo-profunctor on a compact category, a unit extends
uniquely to a multiplication.

\begin{remark}
  One can think of compact categories as a categorification of groups, where duals of objects act
  like inverses of group elements. From this perspective, the results of this section can be seen as
  categorifications of basic facts from group theory.

  We can think of profunctors between compact categories as playing the role of relations between
  groups which are stable under multiplication. Pointed endo-profunctors act like reflexive
  relations, and monoids in profunctors act like reflexive and transitive relations. In fact, one
  can define an equipment of groups, group homomorphisms, and equivariant relations, in which
  monoids are precisely reflexive transitive relations. It is easy to see that copresheaves, i.e.\
  equivariant relations $1\tickar G$, are the same as subgroups of $G$.

  In this way, the equivalence $\CPsh(\dCompProf)\equiv\Mon(\dCompProf)$ categorifies the standard
  fact that a subgroup determines, and is determined by, the conjugacy congruence. The equivalence
  $\Ptd(\dCompProf)\equiv\Mon(\dCompProf)$ would seem to be saying that every reflexive relation
  (stable under multiplication) on a group is in fact transitive, which while true is perhaps less
  familiar than the conjugacy relation. 

  Note that in the definition of a Mal'cev category (see
  \cite{BorceuxBourn}) the last property above is singled out as characterizing categories in which some
  amount of classical group theory can be developed. By analogy, we might think of this section as
  proving that $\dCompProf$ is a ``Mal'cev equipment''.
\end{remark}

It will be helpful to work out what a monoid in $\dMonProf$ looks like using the bimodule notation
for profunctors. A unit for a monoidal profunctor $M\colon\cat{C}\tickar\cat{C}$ is a unit
$i\colon\Hom_{\cat{C}}\to M$ as in Example~\ref{ex:monoid_in_Prof} where
\begin{equation}
    \label{eq:MnProf_monoid_unit}
  i(\id_{I_{\cat{C}}})=I_M \quad\text{and}\quad i(f\otimes g)=i(f)\boxtimes i(g)
\end{equation}
for any morphisms $f$ and $g$ in $\cat{C}$. Similarly, the multiplication $\bullet$ on $M$ must
satisfy
\begin{gather}
  I_M\bullet I_M=I_M \label{eq:MnProf_monoid_I} \\
  (m_2\boxtimes m'_2)\bullet(m_1\boxtimes m'_1) = (m_2\bullet m_1)\boxtimes(m'_2\bullet m'_1)
    \label{eq:MnProf_monoid_exchange}
\end{gather}
for any $m_1\in M(c,d)$, $m'_1\in M(c',d')$, $m_2\in M(d,e)$, and $m'_2\in M(d'e')$, in addition to
the requirements from Example~\ref{ex:monoid_in_Prof}.

\begin{remark}
    \label{rem:suffices_for_monoidal_monoid}
  Equation \eqref{eq:MnProf_monoid_I} follows immediately from \eqref{eq:Prof_monoid_C} and the
  identification $i(\id_{I_{\cat{C}}})=I_M$. Thus, to prove that $i$ and $\bullet$ form a monoid in
  $\dMonProf$, it suffices to show \eqref{eq:MnProf_monoid_unit} and
  \eqref{eq:MnProf_monoid_exchange}, in addition to the requirements discussed in
  Remark~\ref{rem:suffices_for_monoid}.
\end{remark}

The next lemma shows that for any compact category $\cat{C}$ the standard natural equivalence
$\Hom_{\cat{C}}(a,b)\equiv\Hom_{\cat{C}}(I,a^*\otimes b)$ can be
extended to any pointed monoidal endo-profunctor on $\cat{C}$.

\begin{lemma}
    \label{Lem:comp_prof_bijection}
  Let $\cat{C}$ be a compact category. Given any pointed endo-profunctor $i\colon\Hom_\cat{C}\to M$
  in $\Ptd(\dCompProf)$, there is a natural bijection $M(a,b)\iso M(I,a^*\otimes b)$ for any objects
  $a,b\in \cat{C}$.
\end{lemma}
\begin{proof}
  Given $m\in M(a,b)$, we can construct an element
  \[
    \big(i(\id_{a^*})\boxtimes m\big)\cdot\eta_a \in M(I,a^*\otimes b).
  \]
  Conversely, given $m'\in M(I,a^*\otimes b)$, we can construct an element
  \[
    (\epsilon_a\otimes\id_b)\cdot\big(i(\id_a)\boxtimes m'\big) \in M(a,b).
  \]
  It is simple to check that this defines a natural bijection.
\end{proof}

With the fibration $\MOb\colon\End(\dCompProf)\onto\CompCat$ from Definition~\ref{def:ptd}, we can
define the functors
\begin{equation} \begin{tikzcd}[column sep=-1em]
    \label{eqn:functors_1C_CC}
  \CPsh(\dCompProf) \ar[rr,shift left,"F"] \ar[dr,two heads,"\MOb"']
    && \End(\dCompProf) \ar[ll,shift left,"U"] \ar[dl,two heads,"\MOb"] \\
  & \CompCat &
\end{tikzcd} \end{equation}
where $FM\colon\op{\cat{C}}\times \cat{C}\to\Set$ is defined by $FM(a,b)\coloneqq M(a^*\otimes b)$
while $UN\colon \cat{C}\to\Set$ is given by $UN(a)\coloneqq N(I,a)$. It is simple to check that
$F$ and $U$ are morphisms of fibrations, i.e.\ that they preserve cartesian morphisms.

\begin{proposition}
    \label{Prop:canonical unit}
  The functor $F\colon\CPsh(\dCompProf)\to\End(\dCompProf)$ factors through $\Ptd(\dCompProf)$.
\end{proposition}
\begin{proof}
  Let $M\colon \cat{C}\to\Set$ be an object in $\CPsh(\dCompProf)$. Since $M$ is a monoidal
  profunctor, there is a given unit element $I_M\in M(I)$. Thus given any $f\colon c\to d$ in $
  \cat{C}$, we can define the element $i(f)\in FM(c,d)=M(c^*\otimes d)$ via
  \[
    i(f)\coloneqq\big((\id_{c^*}\otimes f)\circ\eta_c\big)\cdot I_M.
  \]
  It is easy to check that this construction of a unit $i$ is functorial.
\end{proof}

Thus, we have induced functors $F,U\colon\CPsh(\dCompProf)\leftrightarrows\Ptd(\dCompProf)$ giving
the diagram
\begin{equation} \begin{tikzcd}[row sep=tiny]
    \label{eqn:ptd_functors_1C_CC}
  & \Ptd(\dCompProf) \ar[dd] \ar[dl,shift left,"U"] \\
  \CPsh(\dCompProf) \ar[ur,shift left,"F"] \ar[dr,shift left,"F"] & \\
  & \End(\dCompProf) \ar[ul,shift left,"U"]
\end{tikzcd} \end{equation}
in which the triangle involving the $F$'s and the triangle involving the $U$'s both commute.

\begin{proposition}
    \label{Prop:ptd_prof_equivalence}
  The functors $F$ and $U$ from (\ref{eqn:ptd_functors_1C_CC}) form an equivalence of fibrations
  \[ \begin{tikzcd}[column sep=-.5em]
    \CPsh(\dCompProf) \ar[rr,"\equiv"] \ar[dr,two heads,"\MOb"']
      && \Ptd(\dCompProf) \ar[dl,two heads,"\MOb"] \\
    & \CompCat &
  \end{tikzcd} \]
\end{proposition}
\begin{proof}
  If $M\in\CPsh(\dCompProf)$, i.e.\ $M$ is a lax functor $\cat{C}\to\Set$ for some compact
  $\cat{C}$, then $U(FM)(a)=(FM)(I,a)=M(I^*\otimes a)\iso M(a)$ for any $a\in \cat{C}$. On the
  other hand, given $N\in\Ptd(\dCompProf)$, we have $F(UN)(a,b)=N(I,a^*\otimes b)$, and the
  equivalence follows from Lemma~\ref{Lem:comp_prof_bijection}.
\end{proof}

To make the proof of Proposition~\ref{prop:unit_implies_monoid} easier to follow, we make use of an
extension of the standard string diagrams for (compact) monoidal categories to monoidal profunctors,
as well as monoids in $\dMonProf$. We summarize the use of these string diagrams in
Table~\ref{tab:string_diagrams}. We will only use these diagrams in the proof of
Proposition~\ref{prop:unit_implies_monoid}, and there only informally, as an aid to follow the
rigorous equational proofs.

\begin{table}
  \centering
  \setlength{\extrarowheight}{3pt}
  \tikzset{every picture/.style={
    string diagram,
    baseline,
    ampersand replacement=\&,
    execute at end picture={
      \node[fit=(current bounding box),inner xsep=1pt,inner ysep=4pt,fill=none] {};
    }
  }}
  \begin{tabular}{c}
    \toprule
    \begin{tabular}{@{\hspace{6pt}}cccc@{\hspace{6pt}}}
      \multicolumn{2}{c}{Profunctors (Rem.~\ref{rmk:profunctor_as_bimodule})}
      & \multicolumn{2}{c}{Monoids in $\dProf$ (Ex.~\ref{ex:monoid_in_Prof})} \\
      \cmidrule[\lightrulewidth](r){1-2}\cmidrule[\lightrulewidth](l){3-4}
      \begin{tikzpicture}
        \matrix {
          \node [draw, circle] (f) {f}; \&[.2cm]
          \node [draw]         (m) {m}; \\
        };
        \begin{scope}[on background layer]
          \node[fit={(m)}] (back) {};
        \end{scope}
        \draw[ar] ($(f)-(1cm,0)$) to["c'"] (f);
        \draw[ar] (f) to["c"] (m.west);
        \draw[ar] (m) to["d"] +(1cm,0);
        \node [caption] {$m\cdot f$};
      \end{tikzpicture}
      &
      \begin{tikzpicture}
        \matrix {
          \node [draw]        (m) {m}; \&[.2cm]
          \node [draw,circle] (g) {g}; \\
        };
        \begin{scope}[on background layer]
          \node[fit={(m)}] (back) {};
        \end{scope}
        \draw[ar] ($(m)-(1cm,0)$) to["c"] (m);
        \draw[ar] (m) to["d"] (g);
        \draw[ar] (g) to["d'"] +(1cm,0);
        \node [caption] {$g\cdot m$};
      \end{tikzpicture}
      &
      \begin{tikzpicture}
        \matrix {
          \node [draw] (m1) {m_1}; \&
          \node [draw] (m2) {m_2}; \\
        };
        \begin{scope}[on background layer]
          \node[fit={(m1) (m2)}] (back) {};
        \end{scope}
        \draw[ar] ($(m1)-(1cm,0)$) to["c"] (m1);
        \draw[ar] (m1) to["d"] (m2);
        \draw[ar] (m2) to["e"] +(1cm,0);
        \node [caption] {$m_2\bullet m_1$};
      \end{tikzpicture}
      &
      \begin{tikzpicture}
        \matrix {
          \node [draw,circle] (f)  {f};\\
        };
        \begin{scope}[on background layer]
          \node[fit={(f)}] (back) {};
        \end{scope}
        \draw[ar] ($(f)-(1cm,0)$) -- (f);
        \draw[ar] (f) -- +(1cm,0);
        \node [caption] {$i(f)$};
      \end{tikzpicture}
    \end{tabular}
    \\ \addlinespace

    Monoid equations \eqref{eq:Prof_monoid_unit}--\eqref{eq:Prof_monoid_C}
    \\ \cmidrule[\lightrulewidth](lr){1-1}
    \hspace{\fill}
    \begin{tikzpicture}
      \matrix {
        \node [draw,circle] (h)  {h};   \&
        \node [draw]        (m1) {m_1}; \&
        \node [draw]        (m2) {m_2}; \&
        \node [draw,circle] (f)  {f};   \\
      };
      \begin{scope}[on background layer]
        \node[fit={(m1) (m2)}] (back) {};
      \end{scope}
      \draw[ar] ($(h)-(1cm,0)$) -- (h);
      \draw[ar] (h) -- (m1);
      \draw[ar] (m1) -- (m2);
      \draw[ar] (m2) -- (f);
      \draw[ar] (f) -- +(1cm,0);
      \node [caption] {$(f\cdot m_2)\bullet (m_1\cdot h) = f\cdot(m_2\bullet m_1)\cdot h$};
    \end{tikzpicture}
    \hspace{\fill}
    \begin{tikzpicture}
      \matrix {
        \node [draw]        (m1) {m_1}; \&
        \node [draw,circle] (g)  {g};   \&
        \node [draw]        (m3) {m_3}; \\
      };
      \begin{scope}[on background layer]
        \node[fit={(m1) (g) (m3)}] (back) {};
      \end{scope}
      \draw[ar] ($(m1)-(1cm,0)$) -- (m1);
      \draw[ar] (m1) -- (g);
      \draw[ar] (g) -- (m3);
      \draw[ar] (m3) -- +(1cm,0);
      \node [caption] {$(m_3\cdot g)\bullet m_1 = m_3\bullet(g\cdot m_1)$};
    \end{tikzpicture}
    \hspace{\fill}
    \\
    \begin{tikzpicture}
      \matrix {
        \node [draw, circle] (f) {f}; \&[.2cm]
        \node [draw]         (m) {m}; \\
      };
      \begin{scope}[on background layer]
        \node[fit={(f) (m)}] (back) {};
      \end{scope}
      \draw[ar] ($(f)-(1cm,0)$) to (f);
      \draw[ar] (f) to (m.west);
      \draw[ar] (m) to +(1cm,0);
      \node [right=.8cm of m] (equals) {=};
      \matrix [right=.8cm of equals] {
        \node [draw, circle] (f) {f}; \&[.2cm]
        \node [draw]         (m) {m}; \\
      };
      \begin{scope}[on background layer]
        \node[fit={(m)}] (back) {};
      \end{scope}
      \draw[ar] ($(f)-(1cm,0)$) to (f);
      \draw[ar] (f) to (m.west);
      \draw[ar] (m) to +(1cm,0);
      \node [caption] {$m\bullet i(f) = m\cdot f$};
    \end{tikzpicture}
    \\ \addlinespace

    Monoidal profunctors (Sec.~\ref{sec:monoidal_profunctors})
    \\ \cmidrule[\lightrulewidth](lr){1-1}
    \hspace{\fill}
    \begin{tikzpicture}
      \matrix {
        \node [draw, circle, white!80!black, fill=white!80!black] (I) {};\\
      };
      \begin{scope}[on background layer]
        \node[fit={(I)}] (back) {};
      \end{scope}
      \draw [dashed] ($(I)-(1cm,0)$) -- +(2cm,0);
      \node at ($(back.west)+(-.25cm,.2cm)$) {\scriptsize $I$};
      \node at ($(back.east)+(.25cm,.2cm)$) {\scriptsize $I$};
      \node [caption] {$I_M\in M(I,I)$};
    \end{tikzpicture}
    \hspace{\fill}
    \begin{tikzpicture}
      \matrix {
        \node [draw] (m1) {m_1};\\
        \node [draw] (m2) {m_2};\\
      };
      \begin{scope}[on background layer]
        \node[fit={(m1) (m2)}] (back) {};
      \draw[ar] ($(m1)-(1cm,0)$) to["c_1"pos=.2] (m1);
      \draw[ar] ($(m2)-(1cm,0)$) to["c_2"pos=.2] (m2);
      \draw[ar] (m1) to["d_1"pos=.8] +(1cm,0);
      \draw[ar] (m2) to["d_2"pos=.8] +(1cm,0);
      \end{scope}
      \node [caption] {$m_1\boxtimes m_2$};
    \end{tikzpicture}
    \hspace{\fill}
    \begin{tikzpicture}
      \matrix {
        \node [draw, circle, white!80!black, fill=white!80!black] (I) {};\\
        \node [draw] (m) {m};\\
      };
      \begin{scope}[on background layer]
        \node[fit={(I) (m)}] (back) {};
      \end{scope}
      \draw [dashed] ($(I)-(1cm,0)$) -- +(2cm,0);
      \draw[ar] ($(m)-(1cm,0)$) -- (m);
      \draw[ar] (m) -- +(1cm,0);
      \node [right=.6cm of back] (equals) {=};
      \matrix [right=.6cm of equals] {
        \node [draw] (m)  {m};\\
      };
      \begin{scope}[on background layer]
        \node[fit={(m)}] (back) {};
      \end{scope}
      \draw[ar] ($(m)-(1cm,0)$) -- (m);
      \draw[ar] (m) -- +(1cm,0);
      \node [caption] {$I\boxtimes m =  m$};
    \end{tikzpicture}
    \hspace{\fill}
    \\
    \hspace{\fill}
    \begin{tikzpicture}[string diagram,ampersand replacement=\&]
      \matrix {
        \node [draw, circle] (f1) {f_1};  \&
        \node [draw]         (m1)  {m_1}; \&
        \node [draw, circle] (g1) {g_1};  \\
        \node [draw, circle] (f2) {f_2};  \&
        \node [draw]         (m2)  {m_2}; \&
        \node [draw, circle] (g2) {g_2};  \\
      };
      \begin{scope}[on background layer]
        \node[fit={(m1) (m2)}] (back) {};
        \draw[ar] ($(f1)-(1cm,0)$) -- (f1);
        \draw[ar] ($(f2)-(1cm,0)$) -- (f2);
        \draw[ar] (f1) -- (m1);
        \draw[ar] (f2) -- (m2);
        \draw[ar] (m1) -- (g1);
        \draw[ar] (m2) -- (g2);
        \draw[ar] (g1) -- +(1cm,0);
        \draw[ar] (g2) -- +(1cm,0);
      \end{scope}
      \node [caption] {
        $\begin{gathered}
          (g_1\cdot m_1\cdot f_1)\boxtimes (g_2\cdot m_2\cdot f_2) \\
          = (g_1\otimes g_2)\cdot(m_1\boxtimes m_2)\cdot (f_1\otimes f_2)
        \end{gathered}$
      };
    \end{tikzpicture}
    \hspace{\fill}
    \begin{tikzpicture}[string diagram,ampersand replacement=\&]
      \matrix {
        \node [draw] (m1) {m_1};\\
        \node [draw] (m2) {m_2};\\
      };
      \begin{scope}[on background layer]
        \node[fit={(m1) (m2)}] (back) {};
      \end{scope}
      \draw[ar] ($(m1)-(1cm,0)$) to["c_1"pos=.2] (m1);
      \draw[ar] ($(m2)-(1cm,0)$) to["c_2"'pos=.2] (m2);
      \draw[ar] (m1) to["d_1"pos=.8] +(1cm,0);
      \draw[ar] (m2) to["d_2"'pos=.8] +(1cm,0);
      \node [right=.6cm of back] (equals) {=};
      \matrix [right=.8cm of equals] {
        \node [draw] (m2) {m_2};\\
        \node [draw] (m1) {m_1};\\
      };
      \begin{scope}[on background layer]
        \node[fit={(m1) (m2)}] (back') {};
      \end{scope}
      \draw ($(m2)-(1.2cm,0)$) to["c_1"pos=0] (m1-|back'.west) -- (m1);
      \draw[over] ($(m1)-(1.2cm,0)$) to["c_2"'pos=0] (m2-|back'.west);
      \draw (m2-|back'.west) -- (m2);
      \draw (m1) -- (m1-|back'.east) to["d_1"pos=1] ($(m2)+(1.2cm,0)$);
      \draw (m2) -- (m2-|back'.east);
      \draw[over] (m2-|back'.east) to["d_2"'pos=1] ($(m1)+(1.2cm,0)$);
      \node [caption] {
        $m_1\boxtimes m_2 = \sigma_{d_1,d_2}\cdot(m_2\boxtimes m_1)\cdot\sigma_{c_1,c_2}^{-1}$
      };
    \end{tikzpicture}
    \hspace{\fill}
    \\ \addlinespace

    Monoids in $\dMonProf$ \eqref{eq:MnProf_monoid_unit}--\eqref{eq:MnProf_monoid_exchange}
    \\ \cmidrule[\lightrulewidth](lr){1-1}
    \hspace{\fill}
    \begin{tikzpicture}[string diagram,ampersand replacement=\&]
      \matrix {
        \node [draw, circle, white!80!black, fill=white!80!black] (I) {};\\
      };
      \begin{scope}[on background layer]
        \node[fit={(I)}] (back) {};
      \end{scope}
      \draw [dashed] ($(I)-(1cm,0)$) -- +(2cm,0);
      \node at ($(back.west)+(-.25cm,.2cm)$) {\scriptsize $I$};
      \node at ($(back.east)+(.25cm,.2cm)$) {\scriptsize $I$};
      \node [caption] {$I_M=i(\id_I)$};
    \end{tikzpicture}
    \hspace{\fill}
    \begin{tikzpicture}[string diagram,ampersand replacement=\&]
      \matrix {
        \node [draw,circle] (f1) {f_1};\\
        \node [draw,circle] (f2) {f_2};\\
      };
      \begin{scope}[on background layer]
        \node[fit={(f1) (f2)}] (back) {};
      \end{scope}
      \draw[ar] ($(f1)-(1cm,0)$) to["c_1"pos=.2] (f1);
      \draw[ar] ($(f2)-(1cm,0)$) to["c_2"pos=.2] (f2);
      \draw[ar] (f1) to["d_1"pos=.8] +(1cm,0);
      \draw[ar] (f2) to["d_2"pos=.8] +(1cm,0);
      \node [caption] {$i(f_1\otimes f_2) = i(f_1)\boxtimes i(f_2)$};
    \end{tikzpicture}
    \hspace{\fill}
    \begin{tikzpicture}[string diagram,ampersand replacement=\&]
      \matrix {
        \node [draw] (m1) {m_1};   \&
        \node [draw] (m2) {m_2};   \\
        \node [draw] (m1') {m_1'}; \&
        \node [draw] (m2') {m_2'}; \\
      };
      \begin{scope}[on background layer]
        \node[fit={(m1) (m2) (m1') (m2')}] (back) {};
      \end{scope}
      \draw[ar] ($(m1)-(1cm,0)$) -- (m1);
      \draw[ar] (m1) -- (m2);
      \draw[ar] (m2) -- +(1cm,0);
      \draw[ar] ($(m1')-(1cm,0)$) -- (m1');
      \draw[ar] (m1') -- (m2');
      \draw[ar] (m2') -- +(1cm,0);
      \node [caption] {
        $\begin{gathered}
          (m_2\boxtimes m_2')\bullet (m_1\boxtimes m_1') \\
          = (m_2\bullet m_1)\boxtimes(m_2'\bullet m_1')
        \end{gathered}$
      };
    \end{tikzpicture}
    \hspace{\fill}
    \\ \bottomrule
  \end{tabular}
  \caption{String diagrams for structured profunctors.\label{tab:string_diagrams}}
\end{table}

\begin{proposition}
    \label{prop:unit_implies_monoid}
  The forgetful functor $\Mon(\dCompProf)\to\Ptd(\dCompProf)$ is an equivalence of fibrations over
  $\CompCat$.
\end{proposition}
\begin{proof}
  It is clear that this forgetful functor, which we refer to as $U$ in the proof, is a morphism of
  fibrations, so we must show that $U$ is an equivalence of categories.

  To define an inverse functor $U^{-1}$, consider an object of $\Ptd(\dCompProf)$, i.e.\ a
  profunctor $N\colon \cat{C}\tickar \cat{C}$ with basepoint $i\colon\Hom_\cat{C}\to N$. We can
  define a multiplication on $N$ by the formula
  \[
    n_2\bullet n_1 \coloneqq (\epsilon_d\otimes\id_e)\cdot(n_1\boxtimes
    i(\id_{d^*})\boxtimes n_2)\cdot(\id_c\otimes \eta_d)
  \]
  for any $n_1\in N(c,d)$ and $n_2\in N(d,e)$, or in picture form:
  
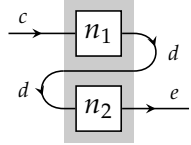
\begin{equation*}
    \begin{tikzpicture}[string diagram]
      \matrix {
        \node [draw]        (n1)  {n_1}; \\
        \node [draw]        (n2)  {n_2}; \\
      };
      \begin{scope}[on background layer]
        \node[fit={(n1) (n2)}] (back) {};
      \end{scope}
      \draw[ar] ($(n1)-(1.2cm,0)$) to["c"pos=.2] (n1);
      \draw[ar] (n1) -- (n1-|back.east) to[in=0,looseness=2,"d"] (back.east);
      \draw (back.east) -- (back.west);
      \draw[ar] (back.west) to[out=180,looseness=2,"d"'] (n2-|back.west) -- (n2);
      \draw[ar] (n2) to["e"pos=.8] +(1.2cm,0);
    \end{tikzpicture}
  \end{equation*}

  It is straightforward to check that this multiplication is associative.
  Remark~\ref{rem:suffices_for_monoidal_monoid} says that, in order to show that $N$ together with
  $i$ and $\bullet$ define an object in $\Mon(\dCompProf)$, we must additionally show that this
  multiplication satisfies the equations \eqref{eq:Prof_monoid_C} and
  \eqref{eq:MnProf_monoid_exchange}. We will begin by showing that $n\bullet i(f)=n\cdot f$ for any
  $n\in N(d,e)$ and $f\colon c\to d$:
  \begin{equation*}
    \begin{tikzpicture}[string diagram]
      \matrix {
        \node [draw,circle] (f) {f}; \\
        \node [draw]        (n) {n}; \\
      };
      \begin{scope}[on background layer]
        \node[fit={(f) (n)}] (back) {};
      \end{scope}
      \draw[ar] ($(f)-(1cm,0)$) -- (f);
      \draw[ar] (f) -- (f-|back.east) to[in=0,looseness=2] (back.east);
      \draw (back.east) -- (back.west);
      \draw[ar] (back.west) to[out=180,looseness=2] (n-|back.west) -- (n);
      \draw[ar] (n) to +(1cm,0);
    \end{tikzpicture}
    \quad = \quad
    \begin{tikzpicture}[string diagram]
      \matrix {
        \node [draw,circle] (f) {f}; &
        \node [fill=none] (phantom a) {}; \\
        \node (phantom b) {}; &
        \node [draw]        (n) {n}; \\
      };
      \begin{scope}[on background layer]
        \node[fit={(phantom a) (n)}] (back) {};
      \end{scope}
      \draw[ar] ($(f)-(1cm,0)$) -- (f);
      \draw (f.east) to[in=0,looseness=2] ($(f.east)!.5!(phantom b.east)$)
        -- ($(f.west)!.5!(phantom b.west)$) to[out=180,looseness=2] (phantom b.west);
      \draw[ar] (phantom b.west) -- (n.west);
      \draw[ar] (n) to +(1cm,0);
    \end{tikzpicture}
    \quad = \quad
    \begin{tikzpicture}[string diagram]
      \matrix {
        \node [draw,circle] (f) {f}; &
        \node [draw]        (n) {n}; \\
      };
      \begin{scope}[on background layer]
        \node[fit={(n)}] {};
      \end{scope}
      \draw[ar] ($(f)-(1cm,0)$) -- (f);
      \draw[ar] (f) -- (n);
      \draw[ar] (n) to +(1cm,0);
    \end{tikzpicture}
  \end{equation*}
  \begin{align*}
    n\bullet i(f)
    &= (\epsilon_d\otimes\id_e) \cdot \bigl(i(f)\boxtimes i(\id_{d^*})\boxtimes n\bigr)
          \cdot (\id_c\otimes \eta_d) \\
    &= (\epsilon_d\otimes\id_e) \cdot \bigl(i(f\otimes\id_{d^*})\boxtimes n\bigr)
          \cdot (\id_c\otimes \eta_d) \\
    &= \bigl((\epsilon_d\cdot i(f\otimes \id_{d^*}))\boxtimes (\id_{e}\cdot n)\bigr)
          \cdot (\id_c\otimes \eta_d) \\
    &= \bigl(i(\epsilon_d\circ (f\otimes \id_{d^*}))\boxtimes (n\cdot\id_d)\bigr)
          \cdot (\id_c\otimes \eta_d) \\
    &= \bigl(i(\id_I)\boxtimes n\bigr)
          \cdot \bigl(((\epsilon_d\circ (f\otimes \id_{d^*}))\otimes\id_d)
             \circ(\id_c\otimes \eta_d)\bigr) \\
    &= \bigl(I_N\boxtimes n\bigr)
          \cdot \bigl((\epsilon_d\otimes\id_d)\circ(\id_d\otimes\eta_d)
             \circ(f\otimes\id_I)\bigr) \\
    &= \bigl(I_N\boxtimes n\bigr) \cdot (f\otimes\id_I) \\
    &= n\cdot f.
  \end{align*}
  The equation $i(f)\bullet n=f\cdot n$ follows similarly, so we have verified
  \eqref{eq:Prof_monoid_C}.

  Finally, we must check \eqref{eq:MnProf_monoid_exchange}. Recall that this says
  \[
    (n_2\boxtimes n'_2)\bullet(n_1\boxtimes n'_1)=(n_2\bullet n_1)\boxtimes(n'_2\bullet n'_1)
  \]
  for any $n_1\in N(c,d)$, $n'_1\in N(c',d')$, $n_2\in N(d,e)$, and $n'_2\in N(d',e')$, which we
  prove below:
  \begin{equation*}
    \begin{tikzpicture}[string diagram]
      \matrix {
        \node [draw] (n1)  {n_1};  \\
        \node [draw] (n1') {n'_1}; \\[.2cm]
        \node [draw] (n2)  {n_2};  \\
        \node [draw] (n2') {n'_2}; \\
      };
      \begin{scope}[on background layer]
        \node[fit={(n1) (n2')}] (back) {};
      \end{scope}
      \begin{scope}[xshift=1cm]
      \coordinate (ld') at ($(n1'.south west-|back.west)!.333!(n2.north west-|back.west)$);
      \coordinate (ld) at ($(n1'.south west-|back.west)!.667!(n2.north west-|back.west)$);
      \coordinate (rd') at ($(n1'.south east-|back.east)!.333!(n2.north east-|back.east)$);
      \coordinate (rd) at ($(n1'.south east-|back.east)!.667!(n2.north east-|back.east)$);
      \end{scope}
      \draw[ar] ($(n1)-(1cm,0)$)  to["c"pos=.2]  (n1);
      \draw[ar] ($(n1')-(1cm,0)$) to["c'"pos=.2] (n1');
      \draw[ar] (n1)  -- (n1-|back.east)  to[in=0,looseness=1.1,"d"pos=.1] (rd);
      \draw[ar] (n1') -- (n1'-|back.east) to[in=0,looseness=2,"d'"pos=.05] (rd');
      \draw (rd') -- (ld');
      \draw (rd) -- (ld);
      \draw[ar] (ld') to[out=180,looseness=1.1,"d'"'pos=.9] (n2'-|back.west) -- (n2');
      \draw[ar] (ld) to[out=180,looseness=2,"d"'pos=.95]  (n2-|back.west)  -- (n2);
      \draw[ar] (n2)  to["e"pos=.8]  +(1cm,0);
      \draw[ar] (n2') to["e'"pos=.8] +(1cm,0);
    \end{tikzpicture}
    \quad = \quad
    \begin{tikzpicture}[string diagram]
      \matrix {
        \node [draw] (n1)  {n_1};  \\
        \node [draw] (n2)  {n_2};  \\[.2cm]
        \node [draw] (n1') {n'_1}; \\
        \node [draw] (n2') {n'_2}; \\
      };
      \begin{scope}[on background layer]
        \node[fit={(n1) (n2')}] (back) {};
      \end{scope}
      \coordinate (ld') at ($(n1'.south west-|back.west)!.5!(n2'.north west-|back.west)$);
      \coordinate (ld) at ($(n1.south west-|back.west)!.5!(n2.north west-|back.west)$);
      \coordinate (rd') at ($(n1'.south east-|back.east)!.5!(n2'.north east-|back.east)$);
      \coordinate (rd) at ($(n1.south east-|back.east)!.5!(n2.north east-|back.east)$);
      \coordinate (ll1) at ($(n1.west)-(1cm,0)$);
      \coordinate (ll1') at (ll1|-n2);
      \coordinate (lld') at ($(n2.south-|ll1)!.333!(n1'.north-|ll1)$);
      \coordinate (lld) at ($(n2.south-|ll1)!.667!(n1'.north-|ll1)$);
      \coordinate (ll2) at (ll1|-n1');
      \coordinate (ll2') at (ll1|-n2');
      \coordinate (rr1) at ($(n1.east)+(1cm,0)$);
      \coordinate (rr1') at (rr1|-n2);
      \coordinate (rrd') at ($(n2.south-|rr1)!.333!(n1'.north-|rr1)$);
      \coordinate (rrd) at ($(n2.south-|rr1)!.667!(n1'.north-|rr1)$);
      \coordinate (rr2) at (rr1|-n1');
      \coordinate (rr2') at (rr1|-n2');
      \draw[ar] ($(ll1)-(.5cm,0)$)  to["c"pos=.2]  (n1);
      \draw[ar=.2] ($(ll1')-(.5cm,0)$) to["c'"pos=.2] (ll1')
        to[looseness=.8] (n1'-|back.west) -- (n1');
      \draw[ar=.45] (n1') -- (n1'-|back.east)
        to[in=200,looseness=.8] (rr1')
        to[in=60,out=20,looseness=1.6,"d'"pos=.05] (rrd')
        to[out=240,in=0,looseness=.8] (rd');
      \draw (n1) -- (rr1);
      \draw[ar=.3,over] (rr1) to[in=-25,looseness=1.1,"d"pos=.1] (rrd)
        to[out=155,in=0,looseness=.8] (rd);
      \draw (rd') -- (ld');
      \draw (rd) -- (ld);
      \draw[ar=.6] (ld') to[out=180,in=-25,looseness=.8] (lld')
        to[out=155,looseness=1.1,"d'"'pos=.9] (ll2')
        to[looseness=.8] (n2');
      \draw[ar=.6,over] (ld) to[out=180,in=60,looseness=.8] (lld)
        to[out=240,in=200,looseness=1.6,"d"'pos=.95] (ll2)
        to[out=20,looseness=.8] (n2-|back.west);
      \draw (n2-|back.west) -- (n2);
      \draw (n2) -- (n2-|back.east);
      \draw[ar=.8,over] (n2-|back.east) to[looseness=.8] (rr2) to["e"] +(.5cm,0);
      \draw[ar] (n2') to["e'"pos=.8] ($(rr2')+(.5cm,0)$);
    \end{tikzpicture}
    \quad = \quad
    \begin{tikzpicture}[string diagram]
      \matrix {
        \node [draw] (n1)  {n_1};  \\
        \node [draw] (n2)  {n_2};  \\[.2cm]
        \node [draw] (n1') {n'_1}; \\
        \node [draw] (n2') {n'_2}; \\
      };
      \begin{scope}[on background layer]
        \node[fit={(n1) (n2')}] (back) {};
      \end{scope}
      \coordinate (ld') at ($(n1'.south west-|back.west)!.5!(n2'.north west-|back.west)$);
      \coordinate (ld) at ($(n1.south west-|back.west)!.5!(n2.north west-|back.west)$);
      \coordinate (rd') at ($(n1'.south east-|back.east)!.5!(n2'.north east-|back.east)$);
      \coordinate (rd) at ($(n1.south east-|back.east)!.5!(n2.north east-|back.east)$);
      \draw[ar] ($(n1)-(1cm,0)$)  to["c"pos=.2]  (n1);
      \draw[ar] ($(n1')-(1cm,0)$) to["c'"pos=.2] (n1');
      \draw[ar] (n1)  -- (n1-|back.east)  to[in=0,looseness=2,"d"pos=.1] (rd);
      \draw[ar] (n1') -- (n1'-|back.east) to[in=0,looseness=2,"d'"pos=.05] (rd');
      \draw (rd') -- (ld');
      \draw (rd) -- (ld);
      \draw[ar] (ld') to[out=180,looseness=2,"d'"'pos=.9] (n2'-|back.west) -- (n2');
      \draw[ar] (ld) to[out=180,looseness=2,"d"'pos=.95]  (n2-|back.west)  -- (n2);
      \draw[ar] (n2)  to["e"pos=.8]  +(1cm,0);
      \draw[ar] (n2') to["e'"pos=.8] +(1cm,0);
    \end{tikzpicture}
  \end{equation*}
  \begin{align*}
    &(n_2\boxtimes n'_2)\bullet(n_1\boxtimes n'_1) \\
    &= (\epsilon_{d\otimes d'}\otimes\id_{e\otimes e'})
       \cdot \left((n_1\boxtimes n'_1)\boxtimes i(\id_{d^*\otimes d'^{*}})
          \boxtimes(n_2\boxtimes n'_2)\right)
       \cdot (\id_{c\otimes c'}\otimes\eta_{d\otimes d'}) \\
    &= \bigl((\epsilon_d\otimes\id_{e\otimes I\otimes e'})
        \circ(\id_d\otimes\sigma_{I,d^*\otimes e}\otimes\id_{e'})\bigr) \\
    &\qquad \cdot \Bigl[n_1\boxtimes\bigl(\epsilon_{d'}\cdot(n'_1\boxtimes i(\id_{d'^{*}}))\bigr)
        \boxtimes \bigl((i(\id_{d*})\boxtimes n_2)\cdot\eta_d\bigr)\boxtimes n'_2\Bigr] \\
    &\qquad \cdot \bigl((\id_c\otimes\sigma_{I,c'\otimes d'^*}\otimes\id_{d'})
        \circ (\id_{c\otimes I\otimes c'}\otimes\eta_{d'}) \bigr) \\
    &= (\epsilon_d\otimes\id_{e\otimes I\otimes e'})
        \cdot \Bigl[n_1\boxtimes\bigl((i(\id_{d*})\boxtimes n_2)\cdot\eta_d\bigr)
        \boxtimes \bigl(\epsilon_{d'}\cdot(n'_1\boxtimes i(\id_{d'^{*}}))\bigr)
        \boxtimes n'_2\Bigr] \\
    &\qquad \cdot (\id_{c\otimes I\otimes c'}\otimes\eta_{d'}) \\
    &= (\epsilon_d\otimes\id_e\otimes\epsilon_{d'}\otimes\id_{e'})
        \cdot \Bigl[\bigl(n_1\boxtimes i(\id_{d^*})\boxtimes n_2\bigr)
        \boxtimes \bigl(n'_1\boxtimes i(\id_{d'^*})\boxtimes n'_2\bigr)\Bigr] \\
    &\qquad \cdot(\id_c\otimes\eta_d\otimes\id_{c'}\otimes\eta_{d'}) \\
    &= \Bigl[(\epsilon_d\otimes\id_e)
        \cdot \bigl(n_1\boxtimes i(\id_{d^*})\boxtimes n_2\bigr)
        \cdot (\id_c\otimes\eta_d)\Bigr] \\
    &\qquad \boxtimes \Bigl[(\epsilon_{d'}\otimes\id_{e'})
        \cdot \bigl(n'_1\boxtimes i(\id_{d'^*})\boxtimes n'_2\bigr)
        \cdot (\id_{c'}\otimes\eta_{d'})\Bigr].
  \end{align*}
  Thus we have shown that the multiplication $\bullet$ defines a monoid $U^{-1}(N)$.

  To define $U^{-1}$ on morphisms, suppose that $M\in\MonProf(\cat{C}, \cat{C})$ is another monoidal
  profunctor with unit, and that $\phi\colon M\to N$ is a monoidal profunctor morphism which
  preserves units. Then $\phi$ also preserves the canonical multiplications:
  \begin{align*}
    \phi(n_2\bullet n_1) &=
    \phi\bigl[(\epsilon_d\otimes\id_e)\cdot(n_1\boxtimes i(\id_{d^*})\boxtimes n_2)
       \cdot(\id_c\otimes \eta_d)\bigr] \\
    &= (\epsilon_d\otimes\id_e)\cdot\phi\bigl(n_1\boxtimes i_M(\id_{d^*})\boxtimes n_2\bigr)
       \cdot(\id_c\otimes \eta_d) \\
    &= (\epsilon_d\otimes\id_e)
       \cdot\bigl(\phi(n_1)\boxtimes \phi(i_M(\id_{d^*}))\boxtimes \phi(n_2)\bigr)
       \cdot(\id_c\otimes \eta_d) \\
    &= (\epsilon_d\otimes\id_e)
       \cdot(\phi(n_1)\boxtimes i_N(\id_{d^*}))\boxtimes \phi(n_2))
       \cdot(\id_c\otimes \eta_d) \\
    &= \phi(n_2)\bullet\phi(n_1)
  \end{align*}

  Clearly $U\circ U^{-1}=\id_{\Ptd(\dCompProf)}$. For the other direction, consider a monoid
  $M\colon \cat{C}\tickar \cat{C}$ with unit $i$ and multiplication $\star$. Then the multiplication
  $\bullet$ defined above in fact coincides with $\star$:
  \begin{align*}
    \begin{tikzpicture}[string diagram,ampersand replacement=\&]
      \matrix {
        \node [draw] (n1) {n1}; \&
        \node [draw] (n2) {n2}; \\
      };
      \begin{scope}[on background layer]
        \node[fit={(n1) (n2)}] {};
      \end{scope}
      \draw[ar] ($(n1)-(1cm,0)$) to["c"pos=.2] (n1);
      \draw[ar] (n1) to["d"] (n2);
      \draw[ar] (n2) to["e"pos=.8] +(1cm,0);
    \end{tikzpicture}
    \quad &= \quad
    \begin{tikzpicture}[string diagram,looseness=1.5,ampersand replacement=\&]
      \matrix {
        \node [draw] (n1) {n1}; \&[.5cm]
        \node [draw] (n2) {n2}; \\
      };
      \coordinate (mid) at ($(n1.east)!.5!(n2.west)$);
      \coordinate (top) at ($(mid)+(0,.3cm)$);
      \coordinate (bot) at ($(mid)-(0,.3cm)$);
      \begin{scope}[on background layer]
        \node[fit={(n1) (n2) (top) (bot)}] {};
      \end{scope}
      \draw[ar] ($(n1)-(1cm,0)$) to["c"pos=.2] (n1);
      \draw[ar] (n1) to (top);
      \draw (top) to[in=0,"d"{pos=.45,inner sep=.5pt}] (mid) to[out=180] (bot);
      \draw[ar] (bot) to (n2);
      \draw[ar] (n2) to["e"pos=.8] +(1cm,0);
    \end{tikzpicture}
    \\[5pt] &= \quad
    \begin{tikzpicture}[string diagram,looseness=1.5,ampersand replacement=\&]
      \matrix {
        \node [draw]        (n1)  {n_1}; \&[-.4cm]
        \coordinate[left=.1cm] (counit); \\
        \coordinate[right=.1cm] (unit); \&
        \node [draw]        (n2)  {n_2}; \\
      };
      \coordinate (mid) at ($(n1)!.5!(n2)$);
      \begin{scope}[on background layer]
        \node[fit={(n1) (n2)}] (back) {};
      \end{scope}
      \draw[ar] ($(n1)-(1.2cm,0)$) to["c"pos=.2] (n1);
      \draw[ar=.7] (n1) to (counit);
      \draw (counit) to[in=0,"d"pos=.1] (counit|-mid) -- (unit|-mid) to[out=180] (unit);
      \draw[ar] (unit) -- (n2);
      \draw[ar] (n2) to["e"pos=.8] +(1.2cm,0);
    \end{tikzpicture}
    \quad = \quad
    \begin{tikzpicture}[string diagram,ampersand replacement=\&]
      \matrix {
        \node [draw]        (n1)  {n_1}; \\
        \node [draw]        (n2)  {n_2}; \\
      };
      \begin{scope}[on background layer]
        \node[fit={(n1) (n2)}] (back) {};
      \end{scope}
      \draw[ar] ($(n1)-(1.2cm,0)$) to["c"pos=.2] (n1);
      \draw[ar] (n1) -- (n1-|back.east) to[in=0,looseness=2,"d"pos=.4] (back.east);
      \draw (back.east) -- (back.west);
      \draw[ar] (back.west) to[out=180,looseness=2,"d"'pos=.6] (n2-|back.west) -- (n2);
      \draw[ar] (n2) to["e"pos=.8] +(1.2cm,0);
    \end{tikzpicture}
  \end{align*}
  \begin{align*}
    n_2\star n_1
      &= n_2 \star \Bigl[\bigl((\epsilon_d\otimes\id_d)
        \circ(\id_d\otimes\eta_d)\bigr) \cdot n_1 \Bigr] \\
    &= \bigl[n_2\cdot(\epsilon_d\otimes\id_d)\bigr]
        \star \bigl[(\id_d\otimes\eta_d)\cdot n_1\bigr] \\
    &= \bigl(i(\epsilon_d)\boxtimes n_2\bigr) \star \bigl(n_1\boxtimes i(\eta_d)\bigr) \\
    &= (\epsilon_d\otimes\id_e)
        \cdot \Bigl[\bigl(i(\id_d)\boxtimes i(\id_{d^*})\boxtimes n_2\bigr)
        \star \bigl(n_1\boxtimes i(\id_{d^*})\boxtimes i(\id_d)\bigr)\Bigr] \\
    & \qquad \cdot (\id_c\otimes\eta_d) \\
    &= (\epsilon_d\otimes\id_e) \cdot \bigl(n_1\boxtimes i(\id_{d^*})\boxtimes n_2\bigr)
        \cdot (\id_c\otimes \eta_d) \\
    &= n_2\bullet n_1
  \end{align*}
  Thus $U^{-1}\circ U=\id_{\Mon(\dCompProf)}$, and $U$ is an equivalence (in fact, isomorphism) of
  categories.
\end{proof}

\subsection{Deducing \ref{thm:TheoremB}}

We now have all the pieces in place to prove \ref{thm:TheoremB}. Recall that we use the notation
$\int$ to denote the Grothendieck construction.

\begin{named}{Theorem B}
    \label{thm:TheoremB}
  There are equivalences of fibrations
  \begin{equation*}
    \begin{tikzcd}[column sep=-1em]
      \int\limits^{\mathclap{\cat{C}\in\CompCat}} \Lax(\cat{C},\Set)
          \ar[rr,"\equiv"] \ar[dr,two heads]
        && \CompCat^{\bo} \ar[dl,two heads,"\dom" pos=.4] \\
      & \CompCat &
    \end{tikzcd}
    \quad\tn{and}\quad
    \begin{tikzcd}[column sep=-1em]
      \int\limits^{\mathclap{\cat{T}\in\TrCat}} \Lax\big(\Int(\cat{T}),\Set\big)
          \ar[rr,"\equiv"] \ar[dr,two heads]
        && \TrCat^{\bo} \ar[dl,two heads,"\dom" pos=.4] \\
      & \TrCat &
    \end{tikzcd}
  \end{equation*}
\end{named}
\begin{proof}
  Essentially by definition, we have isomorphisms of fibrations
  \begin{equation*}
      \int\limits^{\mathclap{\cat{C}\in\CompCat}}\Lax\big(\cat{C},\Set\big)
          \To{\iso}\CPsh(\dCompProf) 
    \quad\text{and}\quad
      \int\limits^{\mathclap{\cat{T}\in\TrCat}}\Lax\big(\Int(\cat{T}),\Set\big)\To{\iso}\CPsh(\dTrProf) 
  \end{equation*}
  over $\CompCat$ and $\TrCat$, respectively.
  Since $\dCompProf$ and $\dTrProf$ are exact by
  Proposition~\ref{prop:CompProf_exact}~and~\ref{prop:TrProf_exact}, we may apply
  Proposition~\ref{prop:Mon_vs_bo} to get equivalences of fibrations over $\CompCat$ and $\TrCat$:
  \begin{equation*}
      \Mon(\dCompProf) \To{\equiv} \CompCat^{\bo} 
    \quad\tn{and}\quad
      \Mon(\dTrProf) \To{\equiv}  \TrCat^{\bo} 
  \end{equation*}
  It then suffices to prove that there are equivalences of fibrations over $\CompCat$ and $\TrCat$:
  \begin{equation}\label{eqn:mon_prof_equivalence}
    \CPsh(\dCompProf) \To{\equiv}  \Mon(\dCompProf) 
  \quad\tn{and}\quad
    \CPsh(\dTrProf) \To{\equiv}  \Mon(\dTrProf) 
  \end{equation}
the first of which is a direct consequence of Propositions~\ref{Prop:ptd_prof_equivalence} and~\ref{prop:unit_implies_monoid}. 

For the second equivalence, we have that $\Int\colon\dTrProf\to\dCompProf$ is a local equivalence, and it preserves the terminal
  object. Thus using Lemma~\ref{lem:Mon_pullback} and Lemma~\ref{lem:Psh_pullback} we construct the
  desired equivalence $\CPsh(\dTrProf)\to\Mon(\dTrProf)$ as the pullback along $\Int$ of the
  equivalence $\CPsh(\dCompProf)\to\Mon(\dCompProf)$.

\end{proof}

\section{Objectwise-freeness}
  \label{sec:monoids_on_free}

If we momentarily denote the free traced category on a set $\LabSet$ as $F(\LabSet)$, a corollary of
\ref{thm:TheoremB} is an isomorphism
$\Lax(\Int(F(\LabSet)),\Set)\cong\TrCat^{\bo}_{F(\LabSet)/}\cong\TrCat_\LabSet$, where
$\TrCat_{\LabSet}$ is the category (defined in the introduction) of traced categories whose monoid
of objects is free on the set $\LabSet$. This was called \ref{thm:traced_is_cob_alg} in the
introduction. Our goal in the present section is to prove \ref{thm:TheoremA}, which allows the set
$\LabSet$ to vary. In order to handle the cases of traced and compact categories uniformly, we first
formalize what it means for an object in a general exact equipment to itself be objectwise-free.

Consider an exact equipment $\dcat{D}$, and let $\dom\colon\dcat{D}_0^{\bo}\onto\dcat{D}_0$ denote
the domain fibration. Suppose we are given an adjunction to a category $\cat{S}$:
\begin{equation*}
  \Adjoint{\cat{S}}{\dcat{D}_0.}{F}{U}
\end{equation*}
Let $T=UF$ be the monad on $\cat{S}$ corresponding to this adjunction, and write $\cat{S}_T$ for the
Kleisli category for $T$, i.e.\ the full subcategory of free objects $Fs$ in $\dcat{D}_0$. Let
$k_T\colon\cat{S}_T\to\dcat{D}_0$ denote the inclusion, and define $k_T^{\bo}$ to be the strict
pullback of $k_T$ along $\dom$:
\[ \begin{tikzcd}
  (\dcat{D}_T^{\bo})_0 \ar[d] \ar[dr,phantom,"\lrcorner" very near start] \ar[r,"k_T^{\bo}"]
    &\dcat{D}_0^{\bo} \ar[d,"\dom"] \\
  \cat{S}_T \ar[r,"k_T"']
    & \dcat{D}_0
\end{tikzcd} \]

\begin{definition}
    \label{def:kleisli_equipment}
  The fully faithful functors $k_T\colon\cat{S}_T\to\dcat{D}_0$ and
  $k_T^{\bo}\colon(\dcat{D}_T^{\bo})_0\to\dcat{D}_0^{\bo}$ induce equipments $\dcat{D}_T\coloneqq
  k_T^*\dcat{D}$ and $\dcat{D}_T^{\bo}\coloneqq (k_T^{\bo})^*\dcat{D}^{\bo}$ (as in
  Definition~\ref{def:induced_locally_equivalent_equipment}), as well as fully faithful local
  equivalences, which we denote
  \begin{equation}
      \label{eqn:FFLE_DT}
    \varphi_T\colon \dcat{D}_T\to\dcat{D}
    \qquad \text{and} \qquad
    \varphi_T^{\bo}\colon \dcat{D}_T^{\bo}\to\dcat{D}^{\bo}.
  \end{equation}
\end{definition}

We can now apply the previous abstract definition to define equipments of objectwise-free monoidal,
compact, and traced categories, preparing for the proof of \ref{thm:TheoremA}. Consider the
free-forgetful adjunctions%
\footnote{\label{foot:adjunctions}
  These three adjunctions in fact extend to 2-adjunctions; see
  Corollary~\ref{cor:2_adjunctions_MonTrComp}.
}
\begin{equation}
    \label{eqn:adjunctions}
  \Adjoint{\Set}{\MonCat}{\FM}{\UM}
  \qquad
  \Adjoint{\Set}{\CompCat}{\FC}{\UC}
  \qquad
  \Adjoint{\Set}{\TrCat}{\FT}{\UT}
\end{equation}
and write $\TM,\TC,$ and $\TT$ for the corresponding monads on $\Set$. Note that $\TM$ and $\TT$ are
both isomorphic to the free monoid monad, while $\TC$ is isomorphic to the free
monoid-with-involution monad.%
\footnote{
   Note that $\TM$ is not the free-\emph{commutative}-monoid monad, even though the objects of
   $\MonCat$ are \emph{symmetric} monoidal categories, because the symmetries are encoded by natural
   isomorphisms, not equalities.
}
Following Definition~\ref{def:kleisli_equipment}, we have equipments:
\[
  \dFrMonProf  \coloneqq \dMonProf_{\TM}
  \qquad
  \dFrCompProf \coloneqq \dCompProf_{\TC}
  \qquad
  \dFrTrProf   \coloneqq \dTrProf_{\TT}
\]
The functor $\dFrMonProf\to\dMonProf$ is a fully faithful local equivalence, meaning it can be
identified with the full sub-equipment of $\dMonProf$ spanned by the monoidal categories which are
\emph{free on a set}; similarly for $\dFrCompProf$ and $\dFrTrProf$. We will write
\begin{align*}
  \FrMonCat&\coloneqq\dFrMonProf_0=\Set_{\TM}\\
  \FrCompCat&\coloneqq\dFrCompProf_0=\Set_{\TC}\\
  \FrTrCat&\coloneqq\dFrTrProf_0=\Set_{\TT}
\end{align*}
for the vertical 1-categories. Note that each of these categories has a terminal object.

Similarly, we can define
\begin{gather*}
  \MMonFrObCat \coloneqq \VVer\left(\dMonProf_{\TM}^{\bo}\right) \qquad
  \CCompFrObCat \coloneqq \VVer\left(\dCompProf_{\TC}^{\bo}\right) \\
  \TTrFrObCat \coloneqq \VVer\left(\dTrProf_{\TT}^{\bo}\right).
\end{gather*}
The functor $\dMonProf_{\TM}^{\bo}\to\dMonProf^{\bo}$ is also a fully faithful local equivalence.
Hence $\MMonFrObCat$ is the full sub-2-category of $\MMonCat^{\bo}$ spanned by those
bijective-on-objects monoidal functors $\cat{M}'\onto\cat{M}$ for which $\cat{M}'=\FM(\LabSet)$ is
free on a set. Note that this is equivalent to the full sub-2-category of $\MMonCat$ spanned by
those monoidal categories $\cat{M}$ which are \emph{objectwise-free}, i.e.\ whose underlying monoid
of objects is free. Likewise for $\CCompFrObCat$ and $\TTrFrObCat$.

\begin{remark}
    \label{rem:PROP}
  We have defined a 2-category $\MMonFrObCat$ of objectwise-free monoidal categories, which are also
  known as (colored) PROPs (see, e.g.\ \cite{HackneyRobertson} for more on PROPs). However, the
  morphisms between PROPs are more restrictive than those defined above, because they must "send
  colors to colors". To define an equipment of PROPs, consider the functor $\FM\colon\Set\to\MonCat$
  and let $\dPROP\coloneqq\FM^*\dMonProf$ be the induced equipment. Similarly, one can define traced
  and compact (colored) PROPs as $\FT^*\dTrProf$ and $\FC^*\dCompProf$ respectively.

  Although we will not prove it here, one can prove a variant of \ref{thm:TheoremA}, namely that
  there are equivalences of categories
  \begin{equation*}
    \int^{\LabSet\in\Set}(\LCob{\LabSet})\alg \to \CompPROP
    \quad\text{and}\quad
    \int^{\LabSet\in\Set}(\LCob{\LabSet})\alg \to \TrPROP.
  \end{equation*}
  See \cite{JoyalKock} for another approach to compact PROPs.
\end{remark}

The proof of \ref{thm:TheoremA} depends on connecting monoids in $\dFrCompProf$ and $\dFrTrProf$ to
$\CCompFrObCat$ and $\TTrFrObCat$. We first prove this connection in the abstract setting, then
specialize to the cases of interest in Corollary~\ref{cor:TrCat_ObjectFree}.

\begin{proposition}
    \label{prop:objectfree_Mod_bo}
  With the setup as in Definition~\ref{def:kleisli_equipment}, suppose also that $\dcat{D}$ is exact
  and has local reflexive coequalizers. There is a commutative diagram of equipments, in which the vertical
  functors are equivalences and the horizontal functors are local equivalences:
  \[ \begin{tikzcd}[column sep=large]
    \dMod(\dcat{D}_T) \ar[r,"\dMod(\varphi_T)"] \ar[d,"\equiv"']
      & \dMod(\dcat{D}) \ar[d,"\equiv"] \\
    \dcat{D}_T^{\bo} \ar[r,"\varphi_T^{\bo}"'] & \dcat{D}^{\bo}.
  \end{tikzcd} \]

  Suppose moreover that $U(\bo)\subseteq\mathrm{iso}(\cat{S})$, the set of isomorphisms in
  $\cat{S}$. Then the following composite is a fully faithful local equivalence:
  \[ \begin{tikzcd}
    \dMod(\dcat{D}_T)\ar[r,"\dMod(\varphi_T)"]
      &[1.5em] \dMod(\dcat{D}) \ar[r,"\ColDash"] & \dcat{D}.
  \end{tikzcd} \]
\end{proposition}
\begin{proof}
  By Lemma~\ref{lemma:FFLE_Mod}, $\dMod(\varphi_T)\colon\dMod(\dcat{D}_T)\to\dMod(\dcat{D})$ is a
  fully faithful local equivalence. The remainder of the first claim follows from
  Theorem~\ref{thm:Mod_vs_bo} and the definition of $\dcat{D}_T^{\bo}$.

  For the second claim, assume $U(\bo)\subseteq\mathrm{iso}(\cat{S})$. From
  Theorem~\ref{thm:Mod_vs_bo} and the first part of the proposition, it suffices to consider the
  composition
  \[ \begin{tikzcd}
    \dcat{D}_T^{\bo} \ar[r,"\varphi_T^{\bo}"]
      & \dcat{D}^{\bo} \ar[r,"\cod"]
      & \dcat{D}.
  \end{tikzcd} \]
  By definition, both $\varphi_T^{\bo}$ and $\cod$ are local equivalences, hence the composition is
  also.

  To see that $(\cod\varphi_T^{\bo})_0$ is fully faithful, consider a pair of objects $p\colon
  Fs\twoheadrightarrow D$ and $p'\colon Fs'\twoheadrightarrow D'$ in $(\dcat{D}_T^{\bo})_0$, and a
  vertical morphism $f\colon D\to D'$ in $\dcat{D}_0$. In the square
  \[ \begin{tikzcd}
    \dcat{D}_0(Fs,Fs') \ar[r,"p'\circ\textrm{--}"] \ar[d,"\iso"']
      & \dcat{D}_0(Fs,D') \ar[d,"\iso"] \\
    \cat{S}(s,UFs') \ar[r,"Up'\circ\textrm{--}"']
      & \cat{S}(s,UD')
  \end{tikzcd} \]
  which commutes by naturality of the adjunction bijection, the bottom function is a bijection since
  $U(p')$ is an isomorphism for any $p'\in\bo$. Hence the top function is a bijection, which shows
  that there exists a unique lift of $f$ to a morphism in $(\dcat{D}_T^{\bo})_0$:
  \[ \begin{tikzcd}
    Fs \ar[r,"\hat{f}"] \ar[d,two heads,"p"']
      & Fs' \ar[d,two heads,"p'"] \\
    D \ar[r,"f"'] & D'
  \end{tikzcd} \]
  as desired.
\end{proof}

\begin{corollary}
    \label{cor:TrCat_ObjectFree}
  There are fully faithful local equivalences of equipments, in the left column, and equivalences of
  2-categories, in the right column:
  \begin{align*}
    \dMod(\dFrMonProf)  &\to \dMonProf &\MMon(\dFrMonProf)&\equiv\MMonFrObCat\\
    \dMod(\dFrTrProf)   &\to \dTrProf  &\MMon(\dFrTrProf)&\equiv\TTrFrObCat\\
    \dMod(\dFrCompProf) &\to \dCompProf &\MMon(\dFrCompProf)&\equiv\CCompFrObCat.
  \end{align*}
\end{corollary}
\begin{proof}
  The left column comes from the second part of Proposition~\ref{prop:objectfree_Mod_bo}, while the
  right column follows by applying $\VVer$ to the equivalence
  $\dMod(\dcat{D}_T)\equiv\dcat{D}_T^{\bo}$ in the first part of
  Proposition~\ref{prop:objectfree_Mod_bo}.
\end{proof}

The last piece needed for the proof of \ref{thm:TheoremA} is to show that the equivalence between
copresheaves and monoids in $\dCompProf$ and $\dTrProf$ proven in Section~\ref{sec:special_CompProf}
restricts to $\dFrCompProf$ and $\dFrTrProf$.

\begin{lemma}
    \label{lem:FrCompProf_Psh_Mon}
  There are equivalences of fibrations
  \begin{equation*}
    \begin{tikzcd}[column sep=-1.25em]
      \CPsh(\dFrCompProf) \ar[rr,"\equiv"] \ar[dr,two heads,"\MOb"']
        && \Mon(\dFrCompProf) \ar[dl,two heads,"\MOb"] \\
      & \FrCompCat &
    \end{tikzcd}
    \quad\tn{and}\quad
    \begin{tikzcd}[column sep=-1.25em]
      \CPsh(\dFrTrProf) \ar[rr,"\equiv"] \ar[dr,two heads,"\MOb"']
        && \Mon(\dFrTrProf) \ar[dl,two heads,"\MOb"] \\
      & \FrTrCat &
    \end{tikzcd}
  \end{equation*}
\end{lemma}
\begin{proof}
  The equipment functor $\varphi_{\MTCfont{C}}\colon\dFrCompProf\to\dCompProf$ (resp.\
  $\varphi_{\MTCfont{T}}\colon\dFrTrProf\to\dTrProf$) is by definition a local equivalence, and it
  preserves terminal objects in the vertical category. Thus using Lemma~\ref{lem:Mon_pullback} and
  Lemma~\ref{lem:Psh_pullback} we construct the desired equivalence
  $\CPsh(\dFrCompProf)\to\Mon(\dFrCompProf)$ as the pullback along $\varphi_{\MTCfont{C}}$ of the
  equivalence $\CPsh(\dCompProf)\to\Mon(\dCompProf)$ from \eqref{eqn:mon_prof_equivalence}, and similarly for the traced case.
\end{proof}

Finally, we are ready to prove \ref{thm:TheoremA}.

\begin{named}{Theorem A}
    \label{thm:TheoremA}
  There are equivalences of 1-categories
  \begin{equation*}
    \int^{\LabSet\in\Set_{\TC}}(\LCob{\LabSet})\alg \to \CompFrObCat
    \quad\text{and}\quad
    \int^{\LabSet\in\Set_{\TT}}(\LCob{\LabSet})\alg \to \TrFrObCat.
  \end{equation*}
\end{named}
\begin{proof}
  First note that, essentially by definition (as well as the fact that $\Cob/\LabSet$ is the free compact category on the set $\LabSet$; see \cite{KellyLaplaza,Abramsky2}), there are isomorphisms of fibrations
  \begin{equation*}
    \begin{tikzcd}[column sep=-2em]
      \int\limits^{\mathclap{\LabSet\in\Set_{\TC}}}(\LCob{\LabSet})\alg
          \ar[rr,phantom,"\iso"] \ar[dr,two heads]
        && \CPsh(\dFrCompProf) \ar[dl,two heads] \\
      & \FrCompCat. &
    \end{tikzcd}
    \quad\text{and}\quad
    \begin{tikzcd}[column sep=-2em]
      \int\limits^{\mathclap{\LabSet\in\Set_{\TT}}}(\LCob{\LabSet})\alg
          \ar[rr,phantom,"\iso"] \ar[dr,two heads]
        && \CPsh(\dFrTrProf) \ar[dl,two heads] \\
      & \FrTrCat &
    \end{tikzcd}
  \end{equation*}
  By Lemma~\ref{lem:FrCompProf_Psh_Mon}, we have equivalences of 1-categories
  \begin{equation}
    \CPsh(\dFrCompProf)\simeq\Mon(\dFrCompProf)
    \quad\text{and}\quad
    \CPsh(\dFrTrProf)\simeq\Mon(\dFrTrProf).
  \end{equation}
  The result now follows from Corollary~\ref{cor:TrCat_ObjectFree}, which provides equivalences of
  2-categories:
  \[
    \MMon(\dFrCompProf)\equiv\CCompFrObCat
    \quad\text{and}\quad
    \MMon(\dFrTrProf)\equiv\TTrFrObCat. \qedhere
  \]
\end{proof}

\section{A traceless characterization of \texorpdfstring{$\TTrCatStrong$}{TrCat}}
  \label{sec:characterization_of_traced}

In this final section, we briefly record a construction of a 2-category bi-equivalent to the
2-category $\TTrCatStrong$ of traced categories with strong functors between them. A distinguishing feature of this construction is that it
makes no mention of a trace operation, nor anything akin to the usual traced category axioms. It is
a direct consequence of the machinery used to prove our main theorems, and is---to the best of our
knowledge---a new result.

The forgetful functors between the categories of structured monoidal categories commute with the
underlying set functors, i.e.\ the following diagram commutes:
\begin{equation}
    \label{eqn:tetrahedron}
  \begin{tikzcd}
    & \TrCat \ar[dr,"\UTM"] \ar[dd,"\UT" near start] & \\
    \CompCat \ar[rr,crossing over,"\UCM"' near start] \ar[dr,"\UC"'] \ar[ru,"\UCT"]
      && \MonCat \ar[dl,"\UM"] \\
    & \Set &
  \end{tikzcd}
\end{equation}
Because the functor $\UCM\colon\CompCat\to\MonCat$ commutes with the right adjoints of the
adjunctions to $\Set$, i.e.\ $\UM\UCM=\UC$, it induces a monad morphism $\alpha\colon\TM\to\TC$
(i.e.\ a natural transformation $\alpha\colon\TM\to\TC$ compatible with the units and
multiplications), given by the composition of the natural transformations
\begin{equation*} \begin{tikzcd}[column sep=2cm]
  \TM \ar[r,equal] \ar[d,dashed,"\alpha"']
    & \UM\FM \ar[r,"\UM\FM\eta_{\MTCfont{C}}"]
    & \UM\FM\UC\FC \ar[d,equal] \\
  \TC
    & \UM\UCM\FC \ar[l,equal]
    & \UM\FM\UM\UCM\FC \ar[l,"\UM\epsilon_{\MTCfont{M}}\UCM\FC"]
\end{tikzcd} \end{equation*}
The component $\alpha_{\LabSet}$ of this transformation is simply the evident inclusion of the free
monoid on a set $\LabSet$ into the free monoid-with-involution on $\LabSet$. The monad map $\alpha$
induces a functor between the Kleisli categories:
\begin{equation*}
    \label{eqn:FMC}
  \FMC\colon\Set_{\TM}\to\Set_{\TC}.
\end{equation*}
Because the monads $\TM$ and $\TT$ are in fact isomorphic, we have
\[
  \FrTrCat = \Set_{\TT} \iso \Set_{\TM} = \FrMonCat.
\]

The following proposition defines the 2-category $\TTrCatStrong$ of traced categories, and strong functors, purely in terms of
$\dCompProf$, $\dMonProf$, and the adjunctions $\FM\dashv\UM$ and $\FC\dashv\UC$. In particular, it
does not involve any explicit mention of the trace structure defined in \cite{JoyalStreetVerity}. However, it does use the main result of the appendix, Corollary~\ref{cor:object_frees}.

\begin{proposition}
  Consider the functor $\FrMonCat\To{\FMC}\FrCompCat\To{k_{\MTCfont{C}}}\CompCat$ and the induced
  equipment $\dcat{F}\coloneqq(k_{\MTCfont{C}}\circ\FMC)^*(\dCompProf)$. There is a fully faithful local equivalence $\dMod(\dcat{F})\to\dTrProf$ and an equivalence of
  2-categories $\MMon(\dcat{F})\equiv\TTrCatStrong$.
\end{proposition}
\begin{proof}
  By combining the definitions of $\dFrTrProf$ and $\dTrProf$, it is easy to see that the following
  square is a pullback:
  \begin{equation} \begin{tikzcd}
      \label{eqn:FTrProf_pullback}
    \dFrTrProf_1 \ar[r] \ar[d,two heads] \ar[dr,phantom,"\lrcorner" very near start]
      & \dCompProf_1 \ar[d,two heads] \\
    \FrMonCat\times\FrMonCat \ar[r,"k_{\MTCfont{C}}\circ\FMC"']
      & \CompCat\times\CompCat
  \end{tikzcd} \end{equation}
  Thus we have an equivalence $\dcat{F}\equiv\dFrTrProf$, and the result follows by
  Corollary~\ref{cor:TrCat_ObjectFree} and Corollary~\ref{cor:object_frees}.
\end{proof}

\appendix
\chapter{Appendix}
  \label{appendix}

This section is mostly independent from the rest of the paper. It is really only used to prove
Corollary~\ref{cor:object_frees}, the three biequivalences
\[
  \MMonFrObCat \to \MMonCatStrong \qquad
  \TTrFrObCat \to \TTrCatStrong \qquad
  \CCompFrObCat \to \CCompCatStrong.
\]
Here, $\MMonFrObCat$ (resp.\ $\TTrFrObCat$ and $\CCompFrObCat$) is the 2-category of objectwise-free
monoidal (resp.\ traced and compact) categories and strict functors between them (see
Section~\ref{sec:monoids_on_free}), whereas $\MMonCatStrong$ (resp.\ $\TTrCatStrong$ and
$\CCompCatStrong$) is the 2-category of monoidal (resp.\ traced and compact) categories with
arbitrary objects and strong functors between them. This result will not be new to experts, but we
found it difficult to find in the literature.

\section{Arrow objects and mapping path objects}\label{sec:app_arrowObs}

\begin{definition}
    \label{def:arrow_object}
  Let $a$ be an object in a 2-category $\ccat{C}$. An \emph{arrow object} of $a$ is an object
  $\Ar{a}$ together with a diagram
  \[ \begin{tikzcd}
    \Ar{a} \ar[r,bend left,"\dom" {domA,anchor = south}] \ar[r,bend right,"\cod"' {codA,anchor = north}]
        & a
    \twocellA{\kappa}
  \end{tikzcd} \]
  which is universal among such diagrams: any diagram as on the left below factors uniquely as on
  the right
  \begin{equation*}
    \begin{tikzcd}
        x \ar[r,bend left,"d" domA] \ar[r,bend right,"c"' codA]
          & a
        \twocellA{\alpha}
    \end{tikzcd}
    \quad = \quad
    \begin{tikzcd}
        x \ar[r,"\hat{\alpha}"]
          & \Ar{a} \ar[r,bend left,"\dom" {domA,anchor = south}] \ar[r,bend right,"\cod"' {codA,anchor=north}]
          & a
        \twocellA{\kappa}
    \end{tikzcd}
  \end{equation*}
  Moreover, given a commutative square in $\ccat{C}(x,a)$, i.e.\ another $d'\colon x\to a$,
  $c'\colon x\to a$, $\alpha'\colon d'\Rightarrow c'$ as on the left above, and 2-cells $\beta\colon
  d\Rightarrow d'$ and $\gamma\colon c\Rightarrow c'$ such that
  $\alpha'\circ\beta=\gamma\circ\alpha$, there is a unique
  $(\beta,\gamma)\colon\hat{\alpha}\Rightarrow\hat{\alpha'}$ such that $\dom(\beta,\gamma)=\beta$
  and $\cod(\beta,\gamma)=\gamma$.

  We say that $\ccat{C}$ \emph{has arrow objects} if an arrow object $a^2$ exists for each object
  $a\in\cat{C}$.
\end{definition}

\begin{example}
    \label{ex:arrow_objects}
  The 2-categories $\CCat$, $\CCat_{\iso}$, $\MMonCatStrong$, $\TTrCatStrong$, and $\CCompCatStrong$
  have arrow objects. Clearly for an object $\cat{A}\in\CCat$, the usual arrow category $\cat{A}^2$
  of arrows and commutative squares, has the necessary universal property. Similarly, the arrow
  category of $\cat{A}$ in $\CCat_{\iso}$ is the category whose objects are isomorphisms in
  $\cat{A}$, and whose morphisms are commutative squares (in which the other morphisms need not be
  isomorphisms).

  Arrow objects in $\MMonCatStrong$ are preserved by the forgetful functor to $\CCat$. If
  $(\cat{M},I,\otimes)$ is a monoidal category then the arrow object $\cat{M}^2$ (in $\CCat$) has a
  natural monoidal product
  \[
    \cat{M}^2\times\cat{M}^2\iso(\cat{M}\times\cat{M})^2\To{\otimes^2}\cat{M}^2,
  \]
  and monoidal unit given by the identity map $\id_I$ on the unit of $\cat{M}$. The maps
  $\dom,\cod\colon\cat{M}^2\to\cat{M}$ are strict monoidal functors, and the transformation
  $\kappa\colon\dom\to\cod$ is monoidal as well. Suppose given a diagram of strong monoidal
  functors:
  \[ \begin{tikzcd}
    \cat{X} \ar[r,bend left,"d" domA] \ar[r,bend right,"c"' codA]
      & \cat{M}
    \twocellA{\alpha}
  \end{tikzcd} \]
  The universal properties of the arrow object $\cat{M}^2$ in $\CCat$ guarantee that the induced
  functor $\hat{\alpha}\colon\cat{X}\to\cat{M}^2$ is strong monoidal. Note that if $d,c$ are strict
  monoidal functors then $\hat{\alpha}$ will be as well.

  The 2-category $\CCompCatStrong$ also has arrow objects, and they are preserved by the forgetful functor
  $\CCompCatStrong\to\CCat_{\iso}$. Recall from Section~\ref{sec:monoidal,compact,traced} that every
  natural transformation between compact categories is an isomorphism. Thus for a compact category
  $\cat{C}$, the arrow category $\cat{C}^2$ has as objects the isomorphisms $a\To{\iso} b$ in
  $\cat{C}$, and as morphisms the commuting squares. This is compact: the dual of $f\colon a\to b$ is
  $(f^{-1})^\ast\colon a^\ast\to b^\ast$.

  The 2-morphisms between traced categories are also defined to be isomorphisms (see
  Remark~\ref{rem:traced_2morphisms}). For a traced category $\cat{T}\in\TTrCatStrong$, the arrow
  object $\cat{T}^2$ has the isomorphisms in $\cat{T}$ as objects and commuting squares as
  morphisms; i.e.\ here too arrow objects are preserved by the 2-functor $\TTrCatStrong\to\Cat_{\iso}$. To see the
  traced structure of $\cat{T}^2$, suppose given objects $a\colon A\To{\iso} A'$, $b\colon
  B\To{\iso} B'$, and $u\colon U\To{\iso} U'$, as well as a morphism $(f,g)\colon a\otimes u\to
  b\otimes u$ as in the diagram to the left
  \begin{equation}
      \label{eqn:arrow_traced}
    \begin{tikzcd}
      A\otimes U\ar[d,"a\otimes u"']\ar[r,"f"]&B\otimes U\ar[d,"b\otimes u"]\\
      A'\otimes U'\ar[r,"g"']&B'\otimes U'
    \end{tikzcd}
    \qquad\qquad
    \begin{tikzcd}[column sep=large]
      A\ar[d,"a"']\ar[r,"\Tr^U_{A,B}(f)"]&B\ar[d,"b"]\\
      A'\ar[r,"\Tr^{U'}_{A',B'}(g)"']&B'
    \end{tikzcd}
  \end{equation}
  Composing with $\id_{B'}\otimes u^{-1}$, we have
  \[
    (\id_{B'}\otimes u^{-1})\circ(b\otimes u)\circ f
        =(\id_{B'}\otimes u^{-1})\circ g\circ(a\otimes u)
  \]
  as morphisms $A\otimes U\to B'\otimes U$. The commutativity of the right-hand diagram in
  \eqref{eqn:arrow_traced} follows from this equation and the axioms of traced categories
  \cite{JoyalStreetVerity}.
\end{example}

\begin{lemma}
    \label{lem:2_adjunction}
  Let $R\colon\ccat{C}\to\ccat{D}$ be a 2-functor, and suppose that $\ccat{C}$ has arrow objects.
  Then $R$ has a left 2-adjoint if and only if $R$ has a left 1-adjoint and $R$ preserves arrow
  objects.
\end{lemma}
\begin{proof}
  First suppose $R$ has a left 1-adjoint $L$ and preserves arrow objects. We want to show that given
  morphisms $f,g\colon D\to RC$ in $\ccat{D}$ and a 2-cell $\alpha\colon f\Rightarrow g$, there is a
  unique $\alpha'$ in $\ccat{C}$ such that $R(\alpha')\eta_D=\alpha$. From the 1-adjunction, we know
  there are unique $f',g'\colon LD\to C$ such that $Rf'\circ\eta_D=f$ and $Rg'\circ\eta_D=g$. Using
  the arrow object $R(C^2)=(RC)^2$, there is a unique morphism $\hat{\alpha}\colon D\to RC^2$ such
  that $\kappa_{RC}\hat{\alpha}=\alpha$. Using the 1-adjunction again, there is a unique
  $\hat{\alpha}'\colon LD\to C^2$ such that $R\hat{\alpha}'\circ\eta_D=\hat{\alpha}$. Finally, we
  let $\alpha'\coloneqq\kappa_C\hat{\alpha}'$, and check
  \[
    R(\alpha')\eta_D = R(\kappa_C)R(\hat{\alpha}')\eta_D = \kappa_{RC}\hat{\alpha} = \alpha.
  \]
  It is clear that this $\alpha'$ is the unique such 2-cell.

  Conversely, it is easy to check that if $R$ has a left 2-adjoint, then $R$ preserves arrow objects
  (right adjoints preserve limits).
\end{proof}

The following result was promised above; see \eqref{eqn:adjunctions} and footnote~\ref{foot:adjunctions}.

\begin{corollary}
    \label{cor:2_adjunctions_MonTrComp}
  There are 2-adjunctions
  \[
    \FM\colon\CCat\leftrightarrows\MMonCatStrong\cocolon\UM
    \qquad
    \FT\colon\CCat_{\iso}\leftrightarrows\TTrCatStrong\cocolon\UT
    \qquad
    \FC\colon\CCat_{\iso}\leftrightarrows\CCompCatStrong\cocolon\UC
  \]
  that extend the 1-adjunctions constructed in \cite{Abramsky2}.
\end{corollary}
\begin{proof}
  Let $R$ be either $\UM$, $\UT$, or $\UC$. Its underlying 1-functor has a left adjoint, constructed
  in \cite{Abramsky2}. We showed in Example~\ref{ex:arrow_objects} that $\MMonCatStrong$,
  $\TTrCatStrong$, and $\CCompCatStrong$ have arrow objects, which are preserved by $R$. The result
  follows by Lemma~\ref{lem:2_adjunction}.
\end{proof}

\begin{definition}
    \label{def:mapping_path_objects}
  Let $f\colon a\to b$ be a morphism in a 2-category $\ccat{C}$. A \emph{mapping path object} of
  $f$ is an object $\Path{f}$ together with a diagram
  \[ \begin{tikzcd}[column sep=small]
    & |[alias=domA]| \Path{f} \ar[dl,"\pi_a"'] \ar[dr,"\pi_b"] & \\
    a \ar[rr,"f"' codA] && b
    \twocelliso[pos=.6]{A}{\rho}
  \end{tikzcd} \]
  where $\rho$ is an isomorphism, which is universal among such diagrams: any diagram as on the
  left below, in which $\alpha$ is an isomorphism, factors uniquely as on the right
  \begin{equation*}
    \begin{tikzcd}[column sep=small]
      & |[alias=domA]| x \ar[dl,"g"'] \ar[dr,"h"] & \\
      a \ar[rr,"f"' codA] && b
      \twocelliso[pos=.6]{A}{\alpha}
    \end{tikzcd}
    \quad = \quad
    \begin{tikzcd}[column sep=small]
      & x \ar[d,dashed,"\hat{\alpha}"] \ar[ddl,bend right,"g"'] \ar[ddr,bend left,"h"] & \\
      & |[alias=domA]| \Path{f} \ar[dl,"\pi_a"' inner sep=1pt] \ar[dr,"\pi_b" inner sep=1pt] & \\
      a \ar[rr,"f"' codA] && b
      \twocelliso[pos=.6]{A}{\rho}
    \end{tikzcd}
  \end{equation*}
  Moreover, given another $g'\colon x\to a$, $h'\colon x\to b$, $\alpha'\colon fg'\iso h'$ as on
  the left above, and isomorphisms $\beta\colon g\iso g'$ and $\gamma\colon h\iso h'$ such that
  $\alpha'\circ f\beta=\gamma\circ\alpha$, there is a unique isomorphism
  $(\beta,\gamma)\colon\hat{\alpha}\iso\hat{\alpha'}$ such that $\pi_a(\beta,\gamma)=\beta$ and
  $\pi_b(\beta,\gamma)=\gamma$.

  We say that $\ccat{C}$ \emph{has mapping path objects} if a mapping path object $P(f)$ exists for
  each morphism $f\colon a\to b$ in $\cat{C}$.
\end{definition}

\begin{example}
    \label{ex:mapping_paths}
  The 2-categories $\CCat$, $\CCat_{\iso}$, $\MMonCatStrong$, $\TTrCatStrong$, and $\CCompCatStrong$
  have mapping path objects. For a morphism $F\colon\cat{A}\to\cat{B}$ in $\CCat$, the mapping path
  category $P(F)$ is a cousin to the comma category $(F\downarrow\id_{\cat{B}})$: the objects are
  triples
  \[
    \Ob(P(F))\coloneqq
        \{ (A,B,i) \mid
          A\in\Ob(\cat{A}), B\in\Ob(\cat{B}), i\colon F(A)\To{\iso}B\text{ is an isomorphism}\}
  \]
  and a morphism $(A,B,i)\to (A',B',i')$ in $P(F)$ consists of a pair of morphisms $A\to A'$ in
  $\cat{A}$ and $B\to B'$ in $\cat{B}$ such that the evident diagram commutes. The 2-category
  $\CCat_{\iso}$  has exactly the same mapping path objects as $\CCat$.

  The mapping path object of a strong functor $F\colon\cat{A}\to\cat{B}$ between monoidal, traced,
  or compact categories exists and is preserved by the forgetful functors to $\CCat$ and
  $\CCat_{\iso}$. In the monoidal case, the mapping path object $P(F)$ of the functor between
  underlying categories has a canonical monoidal structure, e.g.,
  \[
    (A,B,i)\otimes (A',B',i')\coloneqq(A\otimes A',B\otimes B', (i\otimes i')\circ\mu_{A,A'}^{-1})
  \]
  where $\mu_{A,A'}$ is the coherence isomorphism for $F$. The projection functors
  $A\From{\pi_A}P(F)\To{\pi_B}B$ are strict. Given a diagram
  \[ \begin{tikzcd}[column sep=small]
    & |[alias=domA]| \cat{X} \ar[dl,"G"'] \ar[dr,"H"] & \\
    \cat{A} \ar[rr,"F"' codA] && \cat{B}
    \twocelliso[pos=.6]{A}{\alpha}
  \end{tikzcd} \]
  in which $G$ and $H$ are strong (resp.\ strict) monoidal functors, the induced functor
  $\hat{\alpha}\colon\cat{X}\to P(F)$, given on objects by $x\mapsto(G(x),H(x),\alpha_x)$, will be
  strong (resp.\ strict) as well.

  If $\cat{A}$ and $\cat{B}$ are traced categories and $F$ is a traced functor, one obtains a
  canonical trace structure on the monoidal category $P(F)$ in a manner similar to that shown in
  Example~\ref{ex:arrow_objects}. If $\cat{A}$ and $\cat{B}$ are compact categories, then the
  mapping path monoidal category $P(F)$ is naturally compact: the dual of $(A,B,i)$ is $(A^*,B^*,
  (i^{-1})^*)$.
\end{example}

\begin{remark}
    \label{rem:pres_joint_detect}
  The arrow objects and mapping path objects for the 2-categories $\MMonCatStrong$, $\TTrCatStrong$,
  and $\CCompCatStrong$ were discussed in
  Examples~\ref{ex:arrow_objects}~and~\ref{ex:mapping_paths}. Each has a notion of cone, in fact a
  certain weighted limit cone in $\CCat$, though we will not discuss that notion here. We mentioned
  in passing that the structure morphisms for that cone are strict monoidal functors and that they
  ``preserve and jointly detect'' strictness in the sense of
  Definition~\ref{def:preserve_jdetect_strictness} below. In particular, the 2-categories
  $\MMonCat$, $\TTrCat$, and $\CCompCat$ also have arrow objects and mapping path objects, and the
  inclusions of strict-into-strong (e.g.\ $\MMonCat\to\MMonCatStrong$) preserve them. Looking back
  at Examples~\ref{ex:arrow_objects}~and~\ref{ex:mapping_paths}, we see that the forgetful functors
  \[
    \UM\colon\MMonCat\to\CCat
    \qquad
    \UT\colon\TTrCat\to\CCat_{\iso}
    \qquad
    \UC\colon\CCompCat\to\Cat_{\iso}
  \]
  preserve arrow objects and mapping path objects.
\end{remark}

\begin{definition}
    \label{def:fully_faithful}
  A morphism $f\colon a\to b$ in a 2-category $\ccat{C}$ is \emph{fully faithful} if the functor
  $f^*\colon\ccat{C}(x,a)\to\ccat{C}(x,b)$, induced by composition with $f$, is fully faithful for every $x$.
  That is, $f$ is fully faithful if, for every diagram
  \begin{equation*} \begin{tikzcd}[row sep=large]
    x \ar[r,bend left,"u"] \ar[r,bend right,"v"']
        \ar[d,equal]
      & a \ar[d,"f"] \\
    x \ar[r,bend left,"u'" domA] \ar[r,bend right,"v'"' codA]
      & b
    \twocellA{\alpha'}
  \end{tikzcd} \end{equation*}
  such that $fu=u'$ and $fv=v'$, there exists a unique $\alpha\colon u\Rightarrow v$ such that
  $f\alpha=\alpha'$.

  A morphism $f\colon a\to b$ in a 2-category $\ccat{C}$ is \emph{bijective-on-objects} if it is
  left orthogonal to every fully faithful morphism.
\end{definition}

\begin{definition}
    \label{def:surjective_equivalence}
  Say that a morphism $f\colon a\to b$ in a 2-category $\ccat{C}$ is a \emph{surjective equivalence}
  if it can be extended to an adjoint equivalence $g\dashv f$ in which the unit is the identity.
  That is, there is a morphism $g\colon b\to a$ and 2-cell $\epsilon\colon gf\iso 1_a$ such that
  $fg=1_b$, $\epsilon g=1_g$, and $f\epsilon=1_f$.
\end{definition}

\begin{lemma}
    \label{lem:fully_faithful_surjective_equiv}
  Let $f\colon a\to b$ and $g\colon b\to a$ be morphisms in a 2-category such that $fg=1_b$. Then
  $f$ (together with $g$) is a surjective equivalence if and only if $f$ is fully faithful in the
  sense of Definition~\ref{def:fully_faithful}.
\end{lemma}
\begin{proof}
  Suppose $g\dashv f$ is a surjective equivalence. Then for any $x$,
  $f^*\colon\ccat{C}(x,a)\to\ccat{C}(x,b)$ is an equivalence of categories, hence fully faithful.
  Thus $f$ is fully faithful.

  Conversely suppose $f$ is fully faithful. Then because $fgf=f=f 1_a$, there is a unique
  $\epsilon\colon gf\Rightarrow 1_a$ such that $f\epsilon=1_f$. It is easy to check that $\epsilon$
  is an isomorphism, and that $\epsilon g=1_g$.
\end{proof}

\begin{lemma}
    \label{lem:mapping_path_equiv}
  For any morphism $f\colon a\to b$ with a mapping path object $\Path{f}$, the projection
  $\pi_a\colon \Path{f}\to a$ is a surjective equivalence, hence fully faithful.
\end{lemma}
\begin{proof}
  By the universal property of $\Path{f}$ there is a unique morphism $s\colon a\to \Path{f}$ such
  that
  \begin{equation*}
    \begin{tikzcd}[column sep=small]
      & |[alias=domA]| a \ar[dl,equal] \ar[dr,"f"] & \\
      a \ar[rr,"f"' codA] && b
      \twocellA[pos=.6]{1_f}
    \end{tikzcd}
    \quad = \quad
    \begin{tikzcd}[column sep=small]
      & a \ar[d,dashed,"s"] \ar[ddl,bend right,equal] \ar[ddr,bend left,"f"] & \\
      & |[alias=domA]| \Path{f} \ar[dl,"\pi_a"' inner sep=1pt] \ar[dr,"\pi_b" inner sep=1pt] & \\
      a \ar[rr,"f"' codA] && b
      \twocelliso[pos=.6]{A}{\rho}
    \end{tikzcd}
  \end{equation*}
  Because $\pi_a s \pi_a=\pi_a$ and $\pi_b s \pi_a = f\pi_a\iso\pi_b$, we can use the 2-dimensional
  universality of $\Path{f}$ to obtain a unique isomorphism $\epsilon\colon s\pi_a\iso
  1_{\Path{f}}$ such that $\pi_a\epsilon=1_{\pi_a}$ and $\pi_b\epsilon=\rho$. By 2-dimensional
  universality once more, we obtain $\epsilon s=1_{s}$ from the following facts
  \[
    \pi_a\epsilon s=1_{\pi_a}s=\pi_a1_s
    \qquad \text{and} \qquad
    \pi_b\epsilon s=\rho s=1_f=\pi_b1_s.
  \]

  It follows from Lemma~\ref{lem:fully_faithful_surjective_equiv} that $\pi_a$ is fully faithful.
\end{proof}

\section{Strict vs.\ strong morphisms}
  \label{sec:strict_vs_strong}

Between monoidal categories, there are several notions of functor: strict, strong, lax, and colax.
While researchers tend to be most interested in the 2-category $\MMonCatStrong$ of monoidal
categories and strong functors, and similarly $\TTrCatStrong$ and $\CCompCatStrong$, the strict
functors are theoretically important. In this section, we will present a formal framework which
abstracts our examples of interest, and which provides tools for working with and connecting strict
and strong morphisms.

In the case of monoidal categories, there is an inclusion $\iota\colon\MMonCat\to\MMonCatStrong$ as
well as a forgetful functor $\MMonCatStrong\to\CCat$. The cases of traced and compact monoidal
categories are similar, except there we can factor the forgetful functor through the 2-category (or,
if one prefers, the (2,1)-category) $\CCatIso$ of categories, functors, and natural
\emph{isomorphisms}. In these examples, we will want to be able to represent strong functors in
terms of strict ones, by means of a left adjoint to the inclusion of strict into strong. In
Definition~\ref{def:strictifiable} we will enumerate properties which are sufficient to prove the
existence of this left adjoint, and which are satisfied by all of our motivating examples; see
Example~\ref{ex:strictifiable}.

\begin{definition}
    \label{def:preserve_jdetect_strictness}
  Let $\scat{D}$ and $\pcat{D}$ be 2-categories and let $\iota\colon\scat{D}\to\pcat{D}$ be a
  2-functor that is identity-on-objects, faithful, and locally fully faithful. We say that the
  triple $(\scat{D},\pcat{D},i)$ \emph{has mapping path objects} if $\pcat{D}$ has mapping path
  objects as in Definition~\ref{def:mapping_path_objects} such that
  \begin{itemize}
    \item for any $f\colon a\to b$ in $\pcat{D}$, the structure morphisms
      $a\From{\pi_a}P(f)\To{\pi_b}b$ are in $\scat{D}$, and
    \item the pair $(\pi_a,\pi_b)$ \emph{preserves and jointly detects morphisms in $\scat{D}$} in
      the following sense: for any morphism $\ell\colon x\to P(f)$ in $\pcat{D}$, we have that
      $\ell$ is in $\scat{D}$ if and only if the compositions $\pi_a\circ\ell$ and $\pi_b\circ\ell$
      are in $\scat{D}$.
  \end{itemize}
  We say that the triple $(\scat{D},\pcat{D},\iota)$ \emph{has arrow objects} if the analogous
  conditions hold.
\end{definition}

For the following definition, one may keep in mind the case $\scat{D}=\MMonCat$,
$\pcat{D}=\MMonCatStrong$, and $\ccat{C}=\CCat$. See Example~\ref{ex:strictifiable} below.

\begin{definition}
    \label{def:strictifiable}
  Let $\scat{D}$, $\pcat{D}$, and $\ccat{C}$ be 2-categories, and let $U\colon\pcat{D}\to\ccat{C}$
  and $\iota\colon\scat{D}\to\pcat{D}$ be 2-functors. We say that the collection
  $(\scat{D},\pcat{D},\ccat{C},U,\iota)$ \emph{admits strong morphism classifiers} if it satisfies
  the following properties:
  \begin{enumerate}
    \item\label{item:boff}
      The 2-category $\scat{D}$ has a bijective-on-objects/fully faithful factorization.
    \item\label{item:ioofflff}
      The functor $\iota$ is identity-on-objects, faithful, and locally fully faithful.
    \item\label{item:arrowMPO}
      The triple $(\scat{D},\pcat{D},\iota)$ has both arrow objects and mapping path objects
      (Definition~\ref{def:preserve_jdetect_strictness}).
    \item\label{item:left2adj}
      The functor $U\iota\colon\scat{D}\to\ccat{C}$ has a left 2-adjoint $F$.
    \item\label{item:pres_ff}
      The functor $U\iota$ preserves fully faithful morphisms (equivalently, $F$ preserves
      bijective-on-objects morphisms).
    \item\label{item:pres_MPO}
      The functor $U$ preserves mapping path objects.
    \item\label{item:reflids}
      The functor $U$ reflects identity 2-cells.
    \item\label{item:createSEs}
      The pair ($U\iota,U)$ creates surjective equivalences: given any morphism $f\colon A\to B$ in
      $\scat{D}$ and surjective equivalence $g\dashv U\iota(f)$ in $\ccat{C}$, there is a unique
      surjective equivalence $\tilde{g}\dashv\iota f$ in $\pcat{D}$ such that $U\tilde{g}=g$.
  \end{enumerate}
\end{definition}

\begin{remark}
  Other than those involving bijective-on-objects or fully faithful morphisms, all of the properties
  enumerated in Definition~\ref{def:strictifiable} (namely, Properties~\ref{item:ioofflff},
  \ref{item:arrowMPO}, \ref{item:left2adj}, \ref{item:pres_MPO}, \ref{item:reflids}, and
  \ref{item:createSEs}) hold whenever $\scat{D}$ is the 2-category of strict algebras and strict
  morphisms for a 2-monad on $\ccat{C}$, and $\pcat{D}$ is the 2-category of strict algebras and
  pseudo-morphisms. While our main examples can be seen to be algebras for some 2-monad, we have
  found it easier to isolate just those properties we needed to prove
  Theorem~\ref{thm:strong_classifier}.

  This section was strongly inspired by \cite{Bourke} and \cite{LackHomotopy}.
\end{remark}

\begin{example}
    \label{ex:strictifiable}
  Suppose that the collection $(\scat{D},\pcat{D},\ccat{C},\iota,U)$ is defined as in one of the
  following cases:
  \begin{itemize}
    \item $\scat{D}=\MMonCat, \quad \pcat{D}=\MMonCatStrong, \quad \ccat{C}=\CCat$, where
      $\iota\colon\scat{D}\to\pcat{D}$ is the inclusion and $U\colon\pcat{D}\to\ccat{C}$ is the
      forgetful functor;
    \item $\scat{D}=\TTrCat, \quad \pcat{D}=\TTrCatStrong, \quad \ccat{C}=\CCat_{\iso}$, where
      $\iota\colon\scat{D}\to\pcat{D}$ is the inclusion and $U\colon\pcat{D}\to\ccat{C}$ is the
      forgetful functor; or
    \item $\scat{D}=\CCompCat, \quad \pcat{D}=\CCompCatStrong, \quad \ccat{C}=\CCat_{\iso}$, where
      $\iota\colon\scat{D}\to\pcat{D}$ is the inclusion and $U\colon\pcat{D}\to\ccat{C}$ is the
      forgetful functor.
  \end{itemize}
  We will now show that in each case the collection admits strong morphism classifiers.

  Property \ref{item:boff} is proved as Proposition~\ref{prop:(bo,ff)_really_is} and the exactness of
  $\dMonProf$, $\dTrProf$, and $\dCompProf$; see Section~\ref{sec:exactness_proofs}. Property
  \ref{item:ioofflff} is obvious for $\MMonCat$ and $\CCompCat$, and by definition (see
  Remark~\ref{rem:traced_2morphisms}) for $\TTrCat$. Property \ref{item:arrowMPO} is shown in
  Remark~\ref{rem:pres_joint_detect}. Property~\ref{item:left2adj} is shown in
  Corollary~\ref{cor:2_adjunctions_MonTrComp}. Property~\ref{item:pres_ff} is a consequence of
  Proposition~\ref{prop:(bo,ff)_really_is} and
  Propositions~\ref{prop:MonProf_exact},~\ref{prop:CompProf_exact},~and~\ref{prop:TrProf_exact}.
  Property~\ref{item:pres_MPO} is shown in Remark~\ref{rem:pres_joint_detect}.
  Property~\ref{item:reflids} is obvious: if $\alpha\colon F\to G$ is a 2-cell in $\scat{D_s}$ whose
  underlying natural transformation (in $\CCat$) is the identity then it is the identity. It remains
  to prove Property~\ref{item:createSEs}; we first treat the case $\scat{D}=\MMonCat$.

  Suppose that $F\colon A\to B$ is a strict monoidal functor and that there is a surjective
  equivalence $g\dashv U\iota(F)$ in $\CCat$. Let $f\colon a\to b$ denote $U\iota(F)$, so $g\colon
  b\to a$. By Definition~\ref{def:surjective_equivalence}, we have a 2-cell $\epsilon\colon gf\iso
  1_a$ and equalities $1_b=fg$, $\epsilon g=1_g$ and $f\epsilon=1_f$. A strong functor $G\colon B\to
  A$ with $UG=g$ of course acts the same as $g$ on objects and morphisms. Thus it suffices to give
  the coherence isomorphisms $\mu\colon I_A\To{\cong}G(I_B)$ and $\mu_{x,y}\colon G(x)\otimes
  G(y)\To{\cong} G(x\otimes y)$ for objects $x,y\in B$, which satisfy the required equations. Define
  $\mu$ to be the composite $I_A\To{\epsilon^{-1}}gf(I_A)=g(I_B)$, and define $\mu_{x,y}$ to be the
  composite \[ gx\otimes gy\To{\epsilon^{-1}}gf(gx\otimes gy)=g(fgx\otimes fgy)=g(x\otimes y). \]
  The requisite equations can be checked by direct computation, though they actually follow from a
  more general theory (doctrinal adjunctions); see \cite{Kelly}.

  Property~\ref{item:createSEs} holds for the case $\scat{D}=\CCompCat$ because it is a full
  subcategory of $\MMonCat$. For the case $\scat{D}=\TTrCat$, suppose given a strict traced functor
  $F\colon A\to B$, and let $G\colon B\to A$ be the associated monoidal functor constructed above.
  To see that it is traced, note that $B\To{G}A\To{F}B$ is the identity, so $G$ is fully faithful,
  and the result follows from Remark~\ref{rmk:fully_faithful_and_trace}.
\end{example}

Since $\iota$ is identity on objects, we often suppress it for convenience. We draw ordinary arrows
$\;\cdot\to\cdot\;$ for morphisms in $\scat{D}$ and snaked arrows, $\;\cdot\rightsquigarrow\cdot\;$
for morphisms in $\pcat{D}$.

\begin{theorem}
    \label{thm:strong_classifier}
  Suppose that $(\scat{D},\pcat{D},\ccat{C},U,\iota)$ admits strong morphism classifiers. Then the
  functor $\iota$ has a left 2-adjoint $Q\colon\pcat{D}\to\scat{D}$. The counit $q_A\colon QA\to A$
  of this adjunction is given by factoring the counit $\epsilon_A$ of the $F\dashv U\iota$
  adjunction:
  \begin{equation} \begin{tikzcd}
      \label{eqn:factorization_counit}
    FUA \ar[r,two heads,"r_A"] \ar[dr,"\epsilon_A"']
      & QA \ar[d,hook,"q_A"] \\
    & A
  \end{tikzcd} \end{equation}
\end{theorem}
\begin{proof}
  Define $Q$, $r$, and $q$ as in \eqref{eqn:factorization_counit}. We will begin by showing that
  $q_A$ is a surjective equivalence for any $A$, whose inverse $p_A\colon A\pto QA$ will become the
  unit of the $Q\dashv\iota$ adjunction. We write $U$ to denote $U\iota$, in a minor abuse of
  notation. Because $U$ creates surjective equivalences, it suffices to show that $Uq_A$ is a
  surjective equivalence. But $q_A$ is fully faithful by construction, so $Uq_A$ is fully faithful,
  hence by Lemma~\ref{lem:fully_faithful_surjective_equiv} it suffices to construct a section of
  $Uq_A$. We can easily check that $Ur_A\circ\eta_{UA}$ is such a section:
  \[ \begin{tikzcd}
    UA \ar[r,"\eta_{UA}"] \ar[rrd,bend right=18,equal]
      & UFUA \ar[r,"Ur_A"] \ar[dr,"U\epsilon_A"' inner sep=0pt]
      & UQA \ar[d,hook,"Uq_A"] \\
    && UA
  \end{tikzcd} \]
  Thus there is a unique surjective equivalence $p_A\dashv q_A$ in $\pcat{D}$ such that $Up_A=
  Ur_A\circ\eta_{UA}$.

  We next must show that the morphism $p_A\colon A\pto QA$ has the following universal property: for
  any morphism $f\colon A\pto B$ in $\pcat{D}$, there is a unique morphism $f'\colon QA\to B$ in
  $\scat{D}$ for which $f=f'\circ p_A$. From this, it follows that $Q$ extends to a 1-functor which
  is left adjoint to $\iota$. It then follows from Lemma~\ref{lem:2_adjunction} that $Q$ extends to
  a 2-functor which is left 2-adjoint to $\iota$, completing the proof of the theorem.

  First, given an $f\colon A\pto B$, define a morphism $\tilde{f}\colon FUA\to\Path{f}$ in
  $\scat{D}$ as the adjoint of the section $s_{Uf}\colon UA\to U(\Path{f})=\Path{Uf}$ defined as in
  Lemma~\ref{lem:mapping_path_equiv}, i.e.,~$\tilde{f}\coloneqq\epsilon_{\Path{f}}\circ F(s_{Uf})$.
  It follows by adjointness that the following diagram in $\scat{D}$ commutes:
  \[ \begin{tikzcd}
    & FUA \ar[dl,"\epsilon_A"'] \ar[d,"\tilde{f}"] \ar[r,"FUf"]
      & FUB \ar[d,"\epsilon_B"] \\
    A & \Path{f} \ar[l,hook',"\pi_A"] \ar[r,"\pi_B"']
      & B
  \end{tikzcd} \]
  Then by orthogonality there is a unique morphism $\tilde{f}$ in the diagram
  \[ \begin{tikzcd}
    FUA \ar[r,"\tilde{f}"] \ar[d,two heads,"r_A"']
      & \Path{f} \ar[d,hook,"\pi_A"] \ar[r,"\pi_B"]
      & B \\
    QA \ar[r,"q_A"'] \ar[ur,dashed,"\hat{f}" inner sep=0pt] & A
  \end{tikzcd} \]
  making the square commute, and we define $f'\coloneqq \pi_B\circ\hat{f}$.

  We next must check that our definition of $f'$ satisfies $f=f'\circ p_A$. We can construct an
  isomorphism 2-cell $f\iso f'\circ p_A$:
  \[ \begin{tikzcd}[column sep=tiny]
    & QA \ar[dr,hook,"q_A" inner sep=1pt] \ar[rr,"\hat{f}"]
      && |[alias=domA]| \Path{f} \ar[dl,hook,"\pi_A"' inner sep=1pt]
        \ar[dr,"\pi_B" inner sep=1pt] & \\
    A \ar[ur,squiggly,"p_a" inner sep=1pt] \ar[rr,equal]
      && A \ar[rr,squiggly,"f"' codA]
      && B
    \twocelliso[pos=.65]{A}{\rho_f}
  \end{tikzcd} \]
  We can check directly that the underlying 2-cell of $\rho_f\hat{f}p_a$ is the identity on $Uf$,
  \begin{align*}
    U(\rho_f\hat{f}p_a) &= \rho_{Uf}U(\hat{f})U(r_A)\eta_{UA} \\
    &= \rho_{Uf}U(\tilde{f})\eta_{UA} \\
    &= \rho_{Uf}s_{Uf} \\
    &= 1_{Uf}.
  \end{align*}
  Since $U$ reflects identity 2-cells, it follows that $\rho_f\hat{f}q_A$ is the identity,
  $f=f'\circ p_A$.

  Finally, we need to verify that if $f''\colon QA\to B$ is any other strict morphism such that
  $f=f''\circ p_A$, then $f''=f'$. We begin by factoring $f''=\pi_B\circ\hat{f}''$:
  \begin{equation}
      \label{eq:f_hat_fact}
    \begin{tikzcd}[column sep=0em]
      & |[alias=domA]| QA \ar[dl,"q_A"'] \ar[dr,equal] & \\
      A \ar[rr,squiggly,"p_A"' codA] \ar[d,equal]
        && QA \ar[d,"f''"] \\
      A \ar[rr,squiggly,"f"'] && B
      \twocelliso[pos=.6]{A}{}
    \end{tikzcd}
    \quad = \quad
    \begin{tikzcd}[column sep=.7em]
      & QA \ar[d,dashed,"\hat{f}''"] \ar[ddl,bend right,"q_A"'] \ar[ddr,bend left,"f''"] & \\
      & |[alias=domA]| \Path{f} \ar[dl,"\pi_A"' inner sep=1pt] \ar[dr,"\pi_B" inner sep=1pt] & \\
      A \ar[rr,squiggly,"f"' codA] && B
      \twocelliso[pos=.6]{A}{\rho}
    \end{tikzcd}
  \end{equation}
  It will then suffice to show that the diagram
  \[ \begin{tikzcd}
    FUA \ar[r,"\tilde{f}"] \ar[d,two heads,"r_A"']
      & \Path{f} \ar[d,hook,"\pi_A"] \\
    QA \ar[r,"q_A"'] \ar[ur,"\hat{f}''" inner sep=0pt] & A
  \end{tikzcd} \]
  commutes, as then $\hat{f}''=\hat{f}$ by orthogonality, and $f''=\pi_B\hat{f}''=\pi_B\hat{f}=f'$.
  The lower triangle $\pi_A\circ\hat{f}''=q_A$ follows directly from \eqref{eq:f_hat_fact}. To show
  that the upper triangle $\tilde{f}=\hat{f}''\circ r_A$ commutes, it suffices to check equality of
  the adjoints $s_{Uf}=U(\hat{f}''\circ r_A)\circ\eta_{UA}$. We will check this using the universal
  property of the mapping path object $U\Path{f}=\Path{Uf}$ by showing that
  $\rho_{Uf}U(\hat{f}''r_A)\eta_{UA}=\rho_{Uf}s_{Uf}$:
  \begin{align*}
    \begin{tikzcd}[ampersand replacement=\&,column sep=-0.5em]
      UA \ar[rr,"\eta_{UA}"] \ar[dr,"Up_A"']
      \&\& UFUA \ar[dl,"Ur_A"] \\
      \& UQA \ar[d,"U\hat{f}''"] \& \\
      \& |[alias=domA]| U\Path{f} \ar[dl,"\pi_{UA}"'] \ar[dr,"\pi_{UB}"] \& \\
      UA \ar[rr,"Uf"' codA] \&\& UB
      \twocelliso[pos=.6]{A}{\rho_{Uf}}
    \end{tikzcd}
    \quad &= \quad
    \begin{tikzcd}[ampersand replacement=\&,column sep=-0.5em]
      \& UA \ar[d,"Up_A"] \& \\
      \& |[alias=domA]| UQA \ar[dl,"Uq_A"'] \ar[dr,equal] \& \\
      UA \ar[rr,"Up_A"' codA] \ar[d,equal]
        \&\& UQA \ar[d,"Uf''"] \\
      UA \ar[rr,"Uf"'] \&\& UB
      \twocelliso[pos=.6]{A}{}
    \end{tikzcd}
    \\ &= \quad
    \begin{tikzcd}[ampersand replacement=\&]
      UA \ar[r,"Uf"] \ar[d,equal] \& UB \ar[d,equal] \\
      UA \ar[r,"Uf"'] \& UB
    \end{tikzcd}
    \quad = \quad
    \begin{tikzcd}[ampersand replacement=\&,column sep=0.5em]
      \& UA \ar[d,"s_{Uf}"] \ar[ddl,bend right,equal] \ar[ddr,bend left,"Uf"] \& \\
      \& |[alias=domA]| U\Path{f} \ar[dl,"\pi_{UA}"' inner sep=1pt]
        \ar[dr,"\pi_{UB}" inner sep=1pt] \& \\
      UA \ar[rr,"Uf"' codA] \&\& UB
      \twocelliso[pos=.6]{A}{\rho_{Uf}}
      \qedhere
    \end{tikzcd}
  \end{align*}
\end{proof}

\begin{example}
  Consider the case of Theorem~\ref{thm:strong_classifier} applied to the case where
  $\scat{D}=\MMonCat$, $\pcat{D}=\MMonCatStrong$, and $\ccat{C}=\CCat$. Given any monoidal category
  $A$, we construct the monoidal category $QA$ by factoring the counit:
  \[ \begin{tikzcd}
    \FM\UM A \ar[r,two heads,"r_A"] \ar[dr,"\epsilon_A"']
      & QA \ar[d,hook,"q_A"] \\
    & A
  \end{tikzcd} \]
  Concretely, this says that the underlying monoid of objects of $QA$ is the free monoid on
  $\Ob(A)$, and that given two elements $[x_1,\dots,x_n]$ and $[y_1,\dots,y_m]$ of the free monoid,
  the hom set is defined
  \[
    QA([x_1,\dots,x_n],[y_1,\dots,y_m])
      \coloneqq A(x_1\otimes\cdots\otimes x_n,y_1\otimes\cdots\otimes y_m)
  \]
  Theorem~\ref{thm:strong_classifier} then says that strong monoidal functors out of $A$ are the
  same as strict monoidal functors out of $QA$, or more precisely that for any monoidal category $B$
  there is an isomorphism of categories $\MMonCatStrong(A,B)\iso\MMonCat(QA,B)$.

  The cases of $\TTrCat$ and $\CCompCat$ are analogous.
\end{example}

\section{Objectwise-free monoidal, traced, and compact categories}\label{sec:ObjectwiseFree}

Our next goal is to show, continuing the assumptions of the Theorem~\ref{thm:strong_classifier},
that $\pcat{D}$ is 2-equivalent to the full subcategory of $\scat{D}$ spanned by those objects which
are "objectwise-free''. To make this precise, we will further assume we have a 1-category $\cat{S}$,
together with a fully faithful functor $\Disc\colon\cat{S}\to\ccat{C}_0$ into the underlying
category of $\ccat{C}$ with right adjoint $\Ob$, such that a morphism $f$ in $\ccat{C}_0$ is $\bo$
if and only if $\Ob f$ is an isomorphism. The reader may recognize this situation from
Proposition~\ref{prop:objectfree_Mod_bo}. We will write $\ccatFrOb{D}$ for the full sub-2-category
of $\scat{D}$ spanned by those objects $A$ for which there exists an object $s\in\cat{S}$ and a
$\bo$ morphism $F(\Disc(s))\twoheadrightarrow A$. Then we have that:

\begin{theorem}
    \label{thm:free_on_objects_strong_equivalence}
  The following composition is a biequivalence of 2-categories:
  \[ \begin{tikzcd}[column sep=small]
    \ccatFrOb{D} \ar[r,hook]
      & \scat{D} \ar[r,"\iota"]
      & \pcat{D}.
  \end{tikzcd} \]
\end{theorem}
\begin{proof}
  We first need to show that $\iota$ induces equivalences of categories
  $\scat{D}(A,B)\iso\pcat{D}(\iota A,\iota B)$ for any $A$ and $B$ which are objectwise-free. In
  fact, this will hold as long as $A$ is objectwise-free.

  If $A\in\scat{D}$ is objectwise-free, then there exists an object $s\in\cat{S}$ and a $\bo$
  morphism $f\colon F(\Disc(s))\twoheadrightarrow A$ in $\scat{D}$. By the $F\dashv U$ adjunction%
  \footnote{Note that we continue to commit the abuse of notation writing $U$ for $U\iota$.}, there
  is a unique morphism $\tilde{f}\colon\Disc(s)\to UA$ such that $f=\epsilon_A\circ Ff$. Factoring
  $\epsilon_A=q_A\circ r_A$ as in Theorem~\ref{thm:strong_classifier}, we obtain by orthogonality a
  unique lift $p'_A\in\scat{D}$ in the square
  \[ \begin{tikzcd}[column sep=1.7em]
    F\big(\Disc(s)\big) \ar[d,two heads,"f"'] \ar[r,"\tilde{f}"]
      & FUA \ar[r,two heads,"r_A"]
      & QA \ar[d,hook,"q_A"] \\
    A \ar[urr,dashed,"p'_A"] \ar[rr,equal] && A
  \end{tikzcd} \]
  By Lemma~\ref{lem:fully_faithful_surjective_equiv}, it follows that $q_A$ is an equivalence in
  $\scat{D}$. Hence composition with $q_A$ induces the left equivalence in
  \[ \begin{tikzcd}
    \scat{D}(A,B) \ar[r,"\equiv"]
      & \scat{D}(Q\iota A,B) \ar[r,"\iso"]
      & \pcat{D}(\iota A,\iota B)
  \end{tikzcd} \]
  and it is easy to check that the composition is precisely $\iota$ on hom categories.

  Finally, to prove essential surjectivity, consider an object $A\in\pcat{D}$. We know that in the
  factorization
  \[ \begin{tikzcd}
    FUA \ar[dr,"\epsilon_A"] \ar[d,two heads,"r_A"'] & \\
    QA \ar[r,hook,"q_A"'] & A.
  \end{tikzcd} \]
  $\iota q_A$ is an equivalence in $\pcat{D}$. We will be done if we can show that $QA$ is
  objectwise-free.

  Consider the counit $\epsilon_{UA}\colon\Disc(\Ob(UA))\to UA$. Because $\Disc$ is fully faithful,
  it follows that $\Ob(\epsilon_{UA})$ is an isomorphism, hence $\epsilon_{UA}\in\bo$, and therefore
  $F(\epsilon_{UA})\in\bo$ as well. Thus we can take the composition
  \[ \begin{tikzcd}
    F\big(\Disc\big(\Ob(UA)\big)\big) \ar[r,two heads,"F(\epsilon_{UA})"]
      &[1em] FUA \ar[r,two heads,"r_A"]
      & QA
  \end{tikzcd} \]
  showing that $QA$ is objectwise-free.
\end{proof}

\begin{corollary}
    \label{cor:object_frees}
  The canonical inclusions
  \begin{align*}
    \MMonFrObCat &\to \MMonCatStrong \\
    \TTrFrObCat &\to \TTrCatStrong \\
    \CCompFrObCat &\to \CCompCatStrong
  \end{align*}
  are biequivalences of 2-categories.
\end{corollary}
\begin{proof}
  Let $\cat{S}=\Set$, and let $\Disc\colon\cat{S}\leftrightarrows\Cat\cocolon\Ob$ be the discrete
  adjunction. Note that a morphism $f$ in $\CCat$ is $\bo$ if and only if $\Ob(f)$ is an
  isomorphism. The result follows by Example~\ref{ex:strictifiable} and
  Theorem~\ref{thm:free_on_objects_strong_equivalence}.
\end{proof}

\end{document}